\documentclass{article}
\usepackage[T1]{fontenc}
\usepackage[utf8]{inputenc}
\usepackage[english]{babel}
\usepackage{lmodern}
\usepackage[margin=3cm]{geometry}
\usepackage[backend=biber,style=alphabetic,sorting=nty,doi=false,isbn=false,url=false,minalphanames=2,maxbibnames=4,maxcitenames=4]{biblatex}
\AtEveryBibitem{\clearlist{language}}

\usepackage{empheq}

\usepackage{authblk}
\usepackage[hidelinks]{hyperref}
\usepackage{enumitem}
\usepackage{tikz-cd} 	
\usepackage{mathtools}
\usepackage{alltt}
\usepackage{amsfonts}
\usepackage{amsmath}
\usepackage{amssymb}
\usepackage{amsthm}
\usepackage{tikz-cd} 	
\usepackage[nottoc]{tocbibind}
\usepackage{graphicx}
\usepackage[font={small,it}]{caption}
\usepackage{subcaption}
\usepackage{aliascnt}
\usepackage{cases}
\usepackage{comment}	
\usepackage[capitalize,nameinlink]{cleveref}
\usepackage{tikz}
\usetikzlibrary{tikzmark}
\usepackage{slashed}

\crefname{figure}{Fig.}{Fig.}
\Crefname{figure}{Fig.}{Fig.}
\crefname{subfigure}{Fig.}{Fig.}
\Crefname{subfigure}{Fig.}{Fig.}

\usepackage{xargs}   
\usepackage[colorinlistoftodos,prependcaption,textsize=tiny]{todonotes}
\newcommandx{\typo}[2][1=]{\todo[linecolor=red,backgroundcolor=red!25,bordercolor=red,#1]{#2}}
\newcommandx{\change}[2][1=]{\todo[linecolor=blue,backgroundcolor=blue!25,bordercolor=blue,#1]{#2}}
\newcommandx{\answer}[1]{\todo[linecolor=pink,backgroundcolor=pink!25,bordercolor=pink]{#1}}
\newcommandx{\unsure}[2][1=]{\todo[linecolor=green,backgroundcolor=green!25,bordercolor=green,#1]{#2}}
\newcommandx{\improve}[2][1=]{\todo[linecolor=violet,backgroundcolor=violet!25,bordercolor=violet,#1]{#2}}
\newcommandx{\thiswillnotshow}[2][1=]{\todo[disable,#1]{#2}}

\usepackage{tocloft}
\crefformat{equation}{(#2#1#3)}
\numberwithin{equation}{section}

\theoremstyle{definition}

\newtheorem*{notation}{Notation}

\newtheorem*{assumption}{Assumption}

\newtheorem*{theorem*}{Theorem}
\newtheorem*{conjecture*}{Conjecture}

\theoremstyle{plain}
\newtheorem{theorem}{Theorem}[section]

\newtheorem{lemma}{Lemma}[section]
\newtheorem{prop}{Proposition}[section]
\newtheorem{conjecture}{Conjecture}[section]
\newtheorem{cor}{Corollary}[section]
\newtheorem*{claim}{Claim}
\newtheorem{definition}{Definition}[section]

\theoremstyle{remark}
\newtheorem{remark}{Remark}[section]

\crefname{lemma}{Lemma}{Lemmas}
\crefname{prop}{Proposition}{Proposition}
\crefname{conjecture}{Conjecture}{Conjecture}
\crefname{cor}{Corollary}{Corollary}
\crefname{remark}{Remark}{Remark}
\crefname{defi}{Definition}{Definition}
\crefname{equation}{}{}
\crefname{enumi}{}{}
\crefname{appendix}{}{}

\newcommand{\dd}{\mathop{}\!\mathrm{d}}

\newenvironment{nalign}{
	\begin{equation}
		\begin{aligned}
		}{
		\end{aligned}
	\end{equation}
	\ignorespacesafterend
}

\newcommand{\C}{\mathcal{C}}

\newcommand{\N}{\mathbb{N}}

\newcommand{\R}{\mathbb{R}}

\newcommand{\T}{\mathbb{T}}

\newcommand{\scri}{\mathcal{I}}

\newcommand{\abs}[1]{\left\lvert #1\right\rvert}

\newcommand{\jpns}[1]{\langle #1 \rangle}

\newcommand{\norm}[1]{\left\lVert #1\right\rVert}

\newcommand{\floor}[1]{\lfloor #1 \rfloor}
\newcommand{\ceil}[1]{\lceil #1 \rceil}

\newcommand{\E}{\mathcal{E}}

\newcommand{\B}{\mathcal{B}}

\newcommand{\D}{\mathcal{D}}

\newcommand{\F}{\mathbb{F}}

\newcommand{\Diff}{\mathrm{Diff}}

\DeclareMathOperator{\supp}{supp}



\title{Matching conditions for scattering solutions of scalar wave equations on extremal Reissner–Nordstr\"{o}m spacetimes}
\author[1]{Yannis Angelopoulos\thanks{yannis@bimsa.cn}}
\author[2]{Istvan Kadar\thanks{istvan.kadar@math.ethz.ch}}
\affil[1]{\small Beijing Institute of Mathematical Sciences and Applications,
	
	No.~544, Hefangkou Village, Huairou District, 101408 Beijing, China} 
\affil[2]{\small \textit{Department of Mathematics, ETH Zurich, R\"amistrasse 101, 8006 Zurich, Switzerland}}

\usepackage[makeroom]{cancel}
\usepackage{csquotes}

\newcommand{\ern}{g_{\mathrm{ERN}}}
\renewcommand{\F}{\mathcal{F}}
\newcommand{\K}{\mathcal{K}}
\newcommand{\m}{\mathrm{m}}
\newcommand{\s}{\mathrm{s}}
\newcommand{\Div}{\mathrm{div}}
\newcommand{\M}{\mathcal{M}}
\newcommand{\Mcomp}{\overline{\M} }
\newcommand{\McompIn}{\overline{\M}_{\mathrm{in}}}
\newcommand{\MMink}{\overline{\M_{\mathrm{M}}}}
\newcommand{\MMinkempty}{\overline{\M_{\mathrm{M},\emptyset}}}
\newcommand{\multiComp}{\overline{\M}_\m}
\newcommand{\multiCompin}{\overline{\M}_\m^{\mathrm{in},z}}

\newcommand{\Hb}{H_{\mathrm{b}}}
\newcommand{\hor}{\mathcal{H}}
\newcommand{\Vb}{\mathcal{V}_{\mathrm{b}} }

\newcommand{\Vc}{\mathcal{V}_{\mathrm{c}} }
\newcommand{\Ve}{\mathcal{V}_{\mathrm{e}} }
\renewcommand{\b}{\mathrm{b}}
\newcommand{\Diffb}{\Diff_{\b}}
\newcommand{\Diffe}{\Diff_{\mathrm{e}}}
\newcommand{\A}[1]{\mathcal{A}_{\mathrm{#1}}}
\newcommand{\Aext}[1]{\bar{\mathcal{A}}_{\mathrm{#1}}}
\newcommand{\Rcomp}{\overline{\R^3_{\times}}}
\newcommand{\Bcomp}{\overline{B}_{\times}}
\newcommand{\mindex}[1]{\overline{(#1,0)}}
\newcommand{\cupdex}{\overline{\cup}}
\renewcommand{\O}{\mathcal{O}}

\newcommand{\Normal}[1]{A(\Box_{\ern},#1)}
\newcommand{\sphere}{\mathbb{S}^2}

\newcommand{\ip}{i^+}
\graphicspath{{figures}}

\begin{document}
	\maketitle
	\begin{abstract}
		We study scattering solutions $\phi$ of the linear wave equation on extremal Reissner-Nordstr\"{o}m spacetimes, satisfying the following properties: i) $\phi$ attains a prescribed radiation field $\psi_\scri$ through future null infinity, which decays at an inverse polynomial rate; ii) $\phi$ is regular in the exterior region up to and including the future event horizon, i.e.~$\phi\in C^N$, where $N\gg1$ is independent of the decay rate of $\psi_\scri$.
		We prove that such solutions exist for arbitrary $N$, and that they are not unique.
		The proof consists of: 1) finding an approximate solution $\phi_{\mathrm{app}}$ with fast decaying error; 2) using backwards energy estimates in order to correct $\phi_{\mathrm{app}}$ to an exact solution.
		Extremality is used only in the second step.
		The methods of the linear case described above are then used to  show the same result for i) a semilinear equations where the nonlinearity satisfies the null condition ii) geometries describing the hyperbolic orbit of multiple extremal Reissner-Nordstr\"{o}m black holes. 
	\end{abstract}
	
	\setcounter{tocdepth}{1}
	\tableofcontents
	\section{Introduction}
	
	In this paper, we study the scattering problem for scalar wave equations on the exterior $r\geq M$ up to and including the event horizon of extremal Reissner--Nordstr\"{o}m black hole spacetimes
	\begin{equation}\label{intro:eq:scattering}
		\Box_{\ern}\phi=\mathcal{N} [\phi ] ,\quad r\phi|_{\scri}\doteq \psi_{\scri},\quad \phi|_{\hor}\doteq \phi_{\hor},
	\end{equation}
	where $\Box_{\ern}$ is the induced d'Alembertian, and $\psi_{\scri}$ and $\phi_{\hor}$ are data posed at null infinity and the event horizon respectively.
	We also study scattering in the extended region $(r-M)\geq -v^{-1}$ with $\mathcal{B}=\{M-r=v^{-1}\}$ a spacelike hypersurface
	\begin{equation}\label{intro:eq:scatteringB}
		\Box_{\ern}\phi=\mathcal{N} [\phi ] ,\quad r\phi|_{\scri}\doteq \psi_{\scri},\quad (\phi,(r-M)T\phi)|_{\mathcal{B}}\doteq (\phi_{0},\phi_1).
	\end{equation}
	
	The classical problem of constructing the unique scattering (backwards) solutions for \eqref{intro:eq:scattering} from prescribed $\psi_\scri,\phi_\hor$ has already been investigated in the linear case (i.e. for $\mathcal{N}[\phi] = 0$) in great detail in \cite{angelopoulos_non-degenerate_2020} (see below \cref{intro:thm:aag}). 
	
	In the present work, we study a variation of the classical scattering problem, as a toy model for the dynamical construction of multi-black hole spacetimes. 
	For given scattering data at infinity, we aim to find ``compatible'' data at the event horizon from which a smooth solution of \eqref{intro:eq:scattering} can be constructed. 
	Before presenting the preliminary version of the main theorem, we introduce a notation: for $\psi\in C^\infty(\{\R_{t_*}\times \sphere\})$ we write $\psi\in\O(t_*^{-q})$ to mean that $\abs{\Gamma^\alpha\psi_\scri}\leq t_*^{-q}$ for any multi index $\alpha$ and vectorfields $\Gamma=\{t_*\partial_{t_*},x_i\partial_{x_j}-x_j\partial_{x_i}\}$.

	\begin{theorem}[Rough version]\label{intro:thm:rough}
		Let $t_*$ be an appropriate hyperboloidal ``time'' coordinate function on the exterior of an extremal Reissner–Nordstrom spacetime $(\M , g_{\ern})$.
		Let $q>1/2$ and $\psi_{\scri}\in\mathcal{O}(t_*^{-q})$ be partial scattering data at null infinity for the scalar wave equation \cref{intro:eq:scattering}  where
		$$ \mathcal{N} [ \phi ] = \phi^3 + g_{\ern}^{-1} ( d \phi , d\phi ).$$
		
		Then, there exists additional smooth scattering data on the horizon  
		$\phi_{\hor}\in \mathcal{O}(t_*^{-q-1}), $
		such that there exists a scattering solution $\phi$ to \cref{intro:eq:scattering} in a neighbourhood of timelike infinity with scattering data $\psi_{\scri}$ and $\phi_{\hor}$, moreover $\phi$ satisfies
		\begin{equation}
			\phi|_\K\in C^\infty(\K) \mbox{  for all compact $K\subset \M$.}     
		\end{equation}
		The above result can be extended to the interior so that $\phi$ solve \cref{intro:eq:scatteringB}.
	\end{theorem}

	\begin{figure}[htbp]
		\centering
		\includegraphics[width=0.4\textwidth]{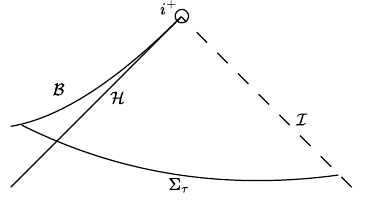}
		\caption{Depicted is the confromal compactification (Penrose diagram) of $\M_{\mathfrak{t}}$ together with the hypersurfaces: (1) the horizon $\hor=\{r=M\}$, (2) partial Cauchy hypersurface $\Sigma_{\tau}=\{t_*=\tau\}$, (3) spacelike hypersurfaces $\mathcal{B}=(M-r)v=\mathfrak{t}$, (4) fictive asymptotic boundary called null infinity $\scri$.}
		\label{fig:penrose}
	\end{figure}
	
	\paragraph{Overview:} The rest of the introduction is structured as follows:
	
	i) In \cref{intro:sec:motivation}, we review previous results regarding scattering theory on black hole spacetimes and put \cref{intro:thm:rough} into a wider context. We also expand on its possible extensions in \cref{intro:sec:mult_bh}.
	
	ii) Next, in \cref{intro:sec:results}, we present a precise form of the main theorems  together with the appropriate bookkeeping tools (function spaces).
	
	iii) Finally, in \cref{intro:sec:proof}, we explain the main ideas of the proof.
	
	\subsection{Motivation}\label{intro:sec:motivation}
	
	It has been proven (see \cite{angelopoulos_late-time_2018,angelopoulos_late-time_2023,hintz_linear_2023} and the heuristic work \cite{price_nonspherical_1972}) that generic forward solutions of the scalar wave equation that emanate from smooth compactly supported initial data on the exterior (up to and including the event horizon) of \textit{subextremal} black hole backgrounds have polynomial tails, i.e. that

	\begin{equation}
		\Box_{g}\phi=0,\quad \phi, \quad \partial_{t_*}\phi\in C^\infty(\Sigma_0) \quad \underbrace{\implies}_{\text{generic}} \quad r\phi|_{\scri}\sim t_*^{-2}, \quad \phi|_{\hor}\sim t_*^{-3}.
	\end{equation}
	It is also well understood, that generic polynomially decaying scattering data at $\scri$ and $\hor$ yield non smooth solutions in the neibhourhood of $\hor$, \cite{dafermos_scattering_2018,dafermos_time-translation_2017}: for all $p>1$  there exists smooth $\psi_\scri \in \mathcal{O}(t_*^{-p})$, such that the unique scattering solution of \cref{intro:eq:scattering} with $\phi_\hor=0$ satisfies $\phi\notin H^1_{\mathrm{loc}}(\M)$.
	These two imply that the \emph{smoothness} of $\phi$ implies a certain \emph{matching} condition between $\phi_\hor$ and $\psi_\scri$.
	
	In what follows, we present in more detail the results described above in order to motivate the study of \cref{intro:thm:rough} and discuss its possible extensions.
	
	\subsubsection{Scattering on black hole backgrounds}\label{intro:sec:scat}
	
	In the important work \cite{dafermos_scattering_2024}  Dafermos-Holzegel-Rodnianski construct \textit{scattering} solutions to the Einstein vacuum equations that settle down to subextremal black hole solutions. The starting point of their analysis was the prescription of exponentially decaying scattering data on $\scri$ and $\hor$. This exponential decay is necessary in light of the blue-shift effect that takes place in a neighbourhood of $\hor$ for scattering problems. Indeed, the more regular the required spacetime is, the faster exponential decay is needed for the scattering data in \cite{dafermos_scattering_2024}. This result is unsatisfying from the perspective of the forward evolution where we expect precisely polynomially decay for the geometric components of the Einstein equations at $\scri$, see \cite{dafermos_lectures_2008,angelopoulos_late-time_2018,hintz_sharp_2022,luk_late_2024,ma_sharp_2023}.
	
	The interplay between decay and regularity is already present in the case of scalar wave equations on black holes spacetimes. In \cite{dafermos_time-translation_2017}, it was shown that on all subextremal Kerr black holes there exist polynomially decaying scattering data on $\scri$, such that if trivial data ($\phi_\hor=0$) is prescribed on $\hor$, the resulting solution will not be in $H^1_{\mathrm{loc}}$ near $\hor$.\footnote{These results can be improved to allow for slow exponential decay along $\scri$, depending on the surface gravity of the black hole.} In this case, we know that the generic decay rate of forward solution is polynomial for the scalar field, see the references above.
	
	All of these results point to the fact that on subextremal black holes, in order to obtain smooth solutions from scattering data, one needs to find a \textit{matching} between $\psi_\scri$ and $\phi_\hor$ up to super exponentially decaying terms. More precisely, $H^N_{\mathrm{loc}}$ regularity requires that $\psi_\scri$ and $\phi_\hor$ are matched up to $e^{-N\kappa t_*}$ decay rate, where $\kappa$ is the surface gravity of the black hole (and we note that $\kappa > 0$ on a subextremal black holes).\footnote{This is not to say that $\psi_\scri$ and $\phi_\hor$ have no \emph{free} exponentially decaying part that is not influenced by the other, see \cref{an:sec:sharp} for the analogous free part in the extremal setting.} The lack of this matching for generic polynomially decaying data leads to the following conjecture:
	\begin{conjecture}[\cite{dafermos_scattering_2024}]\label{intro:conj:dhr}
		For the Einstein vacuum equations, given smooth scattering data settling down to a (possibly extremal) Kerr solution on $\hor$ and $\scri$ at a fast enough inverse polynomial decay rate, there exists a vacuum spacetime $(M, g)$ ``bounded by” $\hor$ and $\scri$, attaining the data, and that is regular away from $\hor$. However, for generic such data converging to a subextremal Kerr solution, the Christoffel symbols of the resulting metric fail to be locally square integrable near the horizon.
	\end{conjecture}
	
	Focusing on the case of extremal Reissner--Nordstr\"{o}m black hole spacetimes, where $\kappa=0$, in \cite{angelopoulos_non-degenerate_2020} it was proven that the scattering data for $\phi_\hor,\psi_\scri$ can be freely prescribed at a super polynomial decay rate and one still obtains smooth solutions.
	More precisely, in \cite{angelopoulos_non-degenerate_2020} it was shown that $H^N_{\mathrm{loc}}$ regularity requires $t_*^{-N}$ decay of the respective data:
	
	\begin{theorem}[Theorem 4.4 of \cite{angelopoulos_non-degenerate_2020}]\label{intro:thm:aag}
		Let $\psi_\scri \in \O(t_*^{-N})$, $N>1$ and $\phi_{\hor} \in \O(t_*^{-N-1})$. Then the unique solution of \cref{lin:eq:main}, $\phi$, satisfies $\phi\in H_{\mathrm{loc}}^N(\M)$ all the way up to and including the event horizon.
	\end{theorem}
	
	Let us summarize the main differences between \cref{intro:thm:rough,intro:thm:aag}:
	\begin{itemize}
		\item While the result of \cite{angelopoulos_non-degenerate_2020} requires fast decay of the scattering data to obtain that \emph{the unique} solution is of higher regularity, we obtain high regularity even when the data has a slow polynomial decay.
		\item \cref{intro:thm:aag} studies a well posed scattering problem, and the regularity of its unique resulting solution, whereas in \cref{intro:thm:rough} we prescribe the \emph{data at $\scri$} and \emph{smoothness} in order to obtain a (non-unique) $\phi_\hor$.
		\item The methods of \cite{angelopoulos_non-degenerate_2020} are more precise in terms of the finite regularity assumptions made on $\phi_\hor,\psi_\scri$ and the solution $\phi$.  Restricting to finite $N$ regularity in \cref{intro:thm:rough} for $\phi_\hor,\psi_\scri$, the corresponding solution $\phi$ is only $H_{\mathrm{loc}}^{N'}$ regular for $N'\sim N$ with an implicit constant we do not keep track of.\footnote{This loss of regularity is not a problem even for nonlinear problems, as one can simply peel the leading order part of the solution as done for instance in \cite[Theorem 6.3]{kadar_scattering_2025}.}
		\item The main results of \cite{angelopoulos_non-degenerate_2020} construct solutions of finite regularity as mentioned above. They can also give smooth solutions, but only if the assumed decay of the scattering data, both at the horizon and at infinity, is \emph{superpolynomial}, see in particular Theorem 4.5 of \cite{angelopoulos_non-degenerate_2020}.
		\item The results of \cite{angelopoulos_non-degenerate_2020} apply only to linear waves on extremal Reissner--Nordstr\"{o}m backgrounds, while we extend \cref{intro:thm:rough} to both nonlinear equations, and to equations on multi-black hole  backgrounds, see \cref{intro:sec:results}.
	\end{itemize}
	
	\subsubsection{Construction of multi-black hole solutions}\label{intro:sec:mult_bh}
	The ``final state conjecture'' in general relativity\footnote{It can be considered to be the general relativistic analogue of the ``soliton resolution conjecture" for dispersive equations, that is, a soliton resolution conjecture for the Einstein equations.} roughly states, that electro-vacuum solutions, that is, solutions of the Einstein--Maxwell equations
	\begin{equation}\label{eq:einstein-maxwell}
		\begin{split}
			& \mathrm{Ein}[g]=R_{\mu \nu} - \frac{1}{2} g_{\mu \nu} R = T^{EM}_{\mu \nu} , \\ &  dF=0,\quad d\star F=0 \mbox{   where   } T^{EM}_{\mu \nu} \doteq F_{\mu\alpha}F^\alpha{}_\nu-\tfrac 14g_{\mu\nu}F_{\alpha\beta}F^{\alpha\beta} ,
		\end{split}
	\end{equation}
	emanating from ``generic'' asymptotically flat Cauchy data, eventually settle down to a finite collection of Kerr–Newman black holes that move apart from one another. 
	
	Before addressing the ``final state conjecture'', one may attempt to construct multi-black hole solutions, that are also asymptotically flat spacetimes, with black holes following approximately Keplerian orbits. 
	In the vacuum case, i.e. $F=0$, a naive approach to construct such solutions proceeds by i) take a metric $g_\m$, that is isometric to black holes in a neighbourhood of timelike geodesics $x=zt$ for some finite collection of $z\in\{\abs{x}<1\}$; ii) solve $\mathrm{Ein}[g_\m+h]=0$ with trivial scattering data on the horizons and $\scri$ for $h$.
	Since $\mathrm{Ein}[g_\m]\sim t_*^{-p}$ decays only polynomially, \cref{intro:conj:dhr} is a clear obstruction to obtain a smooth solution $h$ in the subextremal case.
	We believe that this obstruction still exists in the extremal case:
	\begin{conjecture}\label{intro:conj:extremal}
		For smooth scattering data for \eqref{eq:einstein-maxwell} settling down to an extremal Kerr--Newman solution on $\hor$ and $\scri$ at a fast enough inverse polynomial rate, that is $t_*^{-N}$ for $N$ large enough, there exists a solution of \eqref{eq:einstein-maxwell} $(M, g)$ ``bounded by'' $\hor$ and $\scri$, attaining the data, that is regular away from $\hor$.
		However, for generic data as above, the Christoffel symbols of the resulting metric will not be $C^N$ regular near the horizon.
	\end{conjecture}
	
	Clearly \cref{intro:conj:extremal} could be an obstruction to obtain high regularity scattering solutions of \eqref{eq:einstein-maxwell}.
	In this sense, \cref{intro:thm:rough} is an attempt to solve a matching problem in the scalar setting, and to develop techniques that could be applied to construct solutions that settle down to smooth extremal black holes via scattering for the fully nonlinear problem in the context of the Einstein--Maxwell equations.
	
	\subsection{Main results and proof ideas}\label{intro:sec:results}
	
	In this section, we present some analytic tools and then state the main results of the paper.
	We begin by introducing a compact manifold $\Mcomp$, such that the exterior of an extremal Reissner-Nordstorm spacetime is $\M\cong\mathring{\Mcomp}$, and the regularity-decay statements on $\Mcomp$ are easily captured by conormal regularity on $\Mcomp$.
	We recall that conormal regularity on a manifold with boundary corresponds to regularity with respect to vector fields that are tangential to the boundary.
	For instance on $X=[0,1)_x\times \sphere_y$, we can take the $C^\infty(X)$ span of  $\Vb(X)=\{x\partial_x,\partial_y\}$.
	
	We define (see \cref{not:def:Mcomp}) $\Mcomp$ so that $\partial\Mcomp$ is composed of the following parts: null infinity ($\scri$), the endpoint of (escaping) null geodesics; timelike infinity ($i^+$), the endpoint of (hyperbolic) massive geodesics; spatially compact infinity ($\K$), the endpoint of curves at a fixed non-zero distance form $\hor$; the near horizon region ($\F$), the endpoint of curves that satisfy that $(r-M)t_*$ remains bounded. For a pictorial representation see \cref{fig:compactification}.
	The reason for the introduction of $\Mcomp$, with such a detailed boundary description, is that there are nontrivial "dynamics" happening at each scale corresponding to the different boundaries.\footnote{These dynamics, are present due to the fact that $\ern$ remains a non-degenerate Lorentzian metric all the way up to the boundary on $\Mcomp$, provided that it is expressed relative to an appropriate base. See \cite{hintz_microlocal_2023-1} for more on this point.}
	
	The function spaces that we use to describe the solution $\phi$ on $\Mcomp$ is $\Aext{b}^{q_\scri,q_+,q_\K,q_\F}(\Mcomp)$, which includes functions that decay at the rate $\rho^{-q_\bullet}$ towards the boundary $\bullet\in\{\scri,i^+,\K,\F\}$ (for $\rho$ an appropriate variable that captures such decay), together with all conormal derivatives (adapted to each boundary), and which are also smooth near $\hor$.
	For instance $r-M\in\Aext{b}^{-1,-1,0,1}(\Mcomp)$.
	See \cref{not:sec:function_spaces} for more detail.
	
	We write $\Vb$ for a finite collection of vector fields that span all conormal vector fields over $C^\infty(\Mcomp)$.
	For instance, in a neighbourhood of $i^+\cap \K$, these would be given by $\{t\partial_t,r\partial_r,\Omega_{i}\}$, where $\Omega_{i}, i\in \{1,2,3\},$ are the standard rotation vectorfields with respect to the unit round sphere.

	\begin{figure}[htbp]
		\centering
		\includegraphics[width=0.4\textwidth]{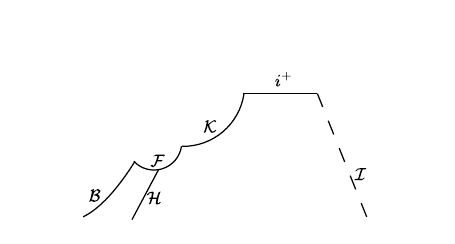}
		\caption{Above we have drawn the analytic compactification of an extremal Reissner--Nordstr\"{o}m spacetime denoting by $\mathcal{B}$ a spacelike hypersurface in the interior of the black hole region, by $\mathcal{H}$ the event horizon, by $\mathcal{I}^+$ future null infinity, and by $i^+$ timelike infinity, while $\mathcal{F}$ and $\mathcal{K}$ are defined above.
			Note that the local coordinates around $\F,\K,i^+$ are $\{(r-M)t_*,x/\abs{x},t_*^{-1}\}$, $\{x,t_*^{-1}\}$,and $\{\frac{x}{t},t_*^{-1}\}$ respectively.}
		\label{fig:compactification}
	\end{figure}
	
	\subsubsection{Statements of the results}
	
	We next state the main theorems in the case of a single background of an extremal Reissner--Nordstr\"{o}m black hole spacetime that we denote by $(\mathcal{M} , \ern )$.
	\begin{theorem}[Matching conditions for semilinear waves on $(\mathcal{M} , \ern)$]\label{intro:thm:sem}
		Let 
		$$\mathcal{N}[\phi]=\phi^3+\ern(\dd\phi,\dd\phi), $$
		and let $q>3/2$ and $\psi_\scri\in \A{b}^{q-1}(\scri)$.
		Then, there exists \emph{non-unique} scattering data $\phi_{\hor}\in\A{b}^{q}(\hor)$ on the horizon, such that there exists a unique scattering solution $\phi$ to \cref{intro:eq:scattering} existing in $\{t_*\in(\tau_0,\infty)\}$ for $\tau_0\gg1$ that depends on $\phi_{\hor}$. Furthermore
		$$ 		\phi\in\Aext{b}^{1,(q,q,q)-}(\Mcomp) , $$
		and in particular it is smooth in a neighbourhood of the horizon $\hor$.
		
		The same result holds for \cref{intro:eq:scatteringB}.
		
		The same result holds for linear waves, i.e. for $ \mathcal{N}[\phi ] = 0 $, but in this case $\tau_0$ can be arbitrary.
	\end{theorem}
	
	\begin{remark}[Interior solution]
		Since, for quasilinear problems, the final horizon is not a-priori fixed, it is desirable to extend the scattering solution into the black hole regions.
		Indeed \cref{intro:thm:sem} is shown in a spacetime region $\{r-M>-t_*\}$ penetrating the black hole.
	\end{remark}
	
	\begin{remark}[Conformal, conormal regularity]\label{intro:rem:conf_vs_con}
		Let us already note that the solution constructed in \cref{intro:thm:sem} are \emph{not} conformally smooth at $\scri$, i.e.~are not regular with respect to the $r^2\partial_r |_u$ commutators.
		This is in contrast with the solutions generated by \cref{intro:thm:aag} in the superpolynomially decaying data case.
		Nonetheless, as discussed in \cite{gajic_relation_2022,kehrberger_case_2024-1,kadar_scattering_2025}, conformal smoothness is not a physically justified assumption, whereas regularity across $\hor$ certainly is.
		This is the motivation to work in a setting that is not symmetric with respect to $\hor$ and $\scri$, as for instance is the case for \cref{intro:thm:aag}.
	\end{remark}
	
	\begin{remark}[Nonlinearity]\label{intro:rem:nonlinearity}
		Regarding the term $\mathcal{N}$ and the corresponding decay rate $q$, we note that \cref{intro:thm:sem} is not optimal or borderline in any sense and it was chosen rather arbitrarily.
		Nonetheless, the proof yields a clear way to classify all $(\mathcal{N},q)$ pairs that are admissible for nonlinear problems, see \cref{in:rem:admissible,an:rem:optimality} for more details.
	\end{remark}

	\begin{remark}[Finite regularity]\label{intro:rem:regularity}
		From the proof it also follows that if we only assume that $\abs{\Vb^N\psi_\scri}\leq t_*^{-q+1}$, then correspondingly $\phi$ will only be $H^{N'}_{loc}$ regular for $\frac{1}{C} N \leq N'\leq C N$, where $C$ is an implicit constant that we do not keep track of, but which is independent of $N$ and $q$.
		The main reason for being rather ignorant about the regularity is that for the construction of many-body solutions, the scattering data in the physically important setting is 0, see \cite{kehrberger_case_2024-1}, and only the velocity and masses of the scattering bodies are given. 
		The latter is a discrete set of numbers and puts no restriction on the regularity of the solutions that are constructable.
	\end{remark}
	
	\begin{remark}[Polyhomogeneity]
		In \cref{sec:poly}, we prove an analogous result to \cref{intro:thm:sem}, by assuming that the scattering data at null infinity admit a polyhomogeneous expansion.
		In particular, if we assume that $\psi_\scri$ has a polyhomogeneous expansion, that is
		\begin{equation}\label{intro:eq:main_data_assumption_poly}
			\psi_\scri=\sum_{p=2}^q u^{-p} \psi_{p,k}+\psi^{\mathfrak{Err}}_q,\quad \psi^{\mathfrak{Err}}_q\in\A{b}^{q+1}(\scri),\quad \psi_{p,k}\in C^\infty(\sphere),
		\end{equation}
		for some $q\in\N_{>1}$, then there exists $\phi_{\hor}\in \A{b}^{3}(\hor)$ that admits a polyhomogeneous expansion, such that for the unique scattering solution $\phi$ to \cref{intro:eq:scattering} we also have a polyhomogeneous expansion towards $\partial\Mcomp$.
	\end{remark}

	Although, the above results already imply the existence of \emph{some} solutions to \cref{intro:eq:scattering}, the proof also yields a large number of such solutions parametrized by residual freedom, connected to the so called Aretakis and Newman--Penrose charges.
	\begin{cor}\label{intro:cor:non-unique}
		Let $\bar{\phi}_p=\phi_p\big((r-M)t_*\big)t_*^{-p}\in\Aext{b}^{\infty,\infty,p+1,p}(\Mcomp)$, for some $\phi_p$ smooth in $\mathring{F}$ and $p\geq1$ satisfying $\Box_{\ern}\bar{\phi}_p\in \A{ext}^{\infty,\infty,p+1,p+1}(\Mcomp)$.
		There exists a solution $\Box_{\ern}\phi_p=0$ with $r\phi_p|_\scri=0$ and $\phi_p-\bar{\phi}_p\in\Aext{b}^{1,p+1,p+1,p+1}(\Mcomp)$.
	\end{cor}
	In particular, \cref{intro:cor:non-unique} yields a way to construct solutions that are regular in $\M$ and have a non-vanishing Aretakis charge 
	$$H_0 [ \phi ] \doteq \left. \int_{\sphere}\partial_r(r\phi) \, d\omega \right|_\hor ,$$ 
	but zero outgoing radiation $r\phi|_{\scri}=0$.
	
	Finally, to show the robustness of the methods and their relevance to \cref{intro:sec:mult_bh}, we also construct scattering solutions on multi-black hole spacetimes.
	In particular, we focus on asymptotically flat spacetimes $(\M_\m,g_\m)$, which have two or more disjoint subsets $U_z\subset \M_\m$, such that each $U_z$ is isometric to a subset of an exact extremal Reissner--Nordst\"{o}rm black hole with $r/t<\delta$ for Boyer–Lindquist coordinates $(r,t)$.
	We emphasise, that $(\M_\m,g_\m)$ is \emph{not} a solution to the Einstein--Maxwell equations and only serves as a toy problem.
	
	\begin{theorem}[Scattering solutions on extremal multi-black hole spacetimes]\label{intro:thm:multi}
		Let $(\M_\m,g_\m)$ be an ``extremal multi-black hole spacetime'' as described in \cref{sec:multi}, $\psi_{\mathcal{I}} \in \mathcal{A}_b^{q-1} (\mathcal{I} )$, $q > 3/2$, and let $$\mathcal{N}_\m[\phi]= \phi^3 +g_\m(\dd\phi,\dd\phi) . $$
		Then, there exists a smooth solution $\phi$ of the partial scattering problem 
		\begin{equation}
			\Box_{g_\m}\phi=\mathcal{N}_\m[\phi],\qquad r\phi|_\scri=\psi_\scri.
		\end{equation}
	\end{theorem}
	
	\subsection{Ideas of the proof}\label{intro:sec:proof}
	The proof for each theorem consists of two steps, both of which have precursors for instance in \cite{angelopoulos_non-degenerate_2020,hintz_gluing_2023,hintz_lectures_2023}.
	For concreteness, let us discuss the proof of \cref{intro:thm:sem}.
	First, one constructs an approximate solution $\phi_{\mathrm{app}}\in\Aext{b}^{1,q,q,q}(\Mcomp)$ satisfying
	\begin{equation}\label{intro:eq:approx}
		\Box_{\ern}\phi_{\mathrm{app}}-\mathcal{N}[\phi_{\mathrm{app}}]\in \Aext{b}^{3,\infty,\infty,\infty}(\Mcomp), \quad r\phi_{\mathrm{app}}|_{\scri}=\psi_\scri.
	\end{equation}
	Next, we prove $T$ energy estimates for the linear problem $\Box_{\ern}\phi=f$ on function spaces with decaying weight.\footnote{Following \cite{dafermos_scattering_2018}, we refer to the estimate from the current $\partial_t\cdot\T$, for $\T$ the energy momentum tensor, as $T=\partial_t$ energy.}
	Using these inhomogeneous estimates, one obtains the solution via a limiting argument, as all nonlinear terms are perturbative.
	For both approaches the following mapping property, see \cref{map:lemma:normal_ops}, is important to keep in mind with respects to the various weights at infinity:
	\begin{equation}\label{intro:eq:mapping}
		\Box_{\ern}:\Aext{b}^{q_\scri,q_+,q_\K,q_\F}(\Mcomp)\to \Aext{b}^{q_\scri+1,q_++2,q_\K,q_\F}(\Mcomp).
	\end{equation}
	Let us expand on both points.
	
	\paragraph{Construction of an approximate solution}
	For this part of the proof, extremality is not an important assumption.
	Indeed, it follows from the methods of the paper that \cref{intro:eq:approx} can be proved on all subextremal Reissner--Nordstr\"{o}m spacetimes,\footnote{For the approximate solution the only important point is the invertibility of the zero frequency operator $\hat{\Box}_g$, which holds for all Kerr-Newman geometries, and the invertibility of a near horizon model operator defined in \cref{map:eq:model_op} for $\ern$.}
	see \cref{an:sec:subextremal} for more details (however, a polynomial approximation as given by \cref{intro:eq:approx} is not a useful starting point in the subextremal setting, one would need to obtain super exponentially decaying errors).

	The framework and the ideas presented here largely follow the general theme of geometric singular analysis as discussed for instance in \cite{hintz_lectures_2023,hintz_gluing_2023,hintz_gluing_2024-2}.
	
	The solution $\phi_{\mathrm{app}}$ will be constructed inductively. For the construction of the ansatz, the boundary hypersurfaces $i^+,\mathcal{K},\mathcal{F}$ play a crucial role.
	Indeed, around each of these hypersurfaces, one can approximate the linear operator $\Box_{\ern}$ with a simpler model.
	These are in order
	\begin{itemize}
		\item $\Box_{\eta}$, where $\eta$ is the Minkowski metric;
		\item $\hat{\Box}_{\ern}(0)$ is the zero frequency operator corresponding to $\Box_{\ern}$;
		\item $\Box_{g_{\mathrm{near}}}$, where $g_{\mathrm{near}}$ is the near horizon metric, see \cref{map:lemma:normal_ops}.
	\end{itemize}
	Importantly, each one of these capture the behaviour of $\Box_{\ern}$ to leading order in terms of \emph{decay} towards the respective boundaries with respect to \cref{intro:eq:mapping}.
	For instance
	\begin{equation}
		\Box_{\ern}-\Box_\eta:\Aext{b}^{q_\scri,q_+,q_\K,q_\F}(\Mcomp)\to \Aext{b}^{q_\scri+1,q_++3,q_\K,q_\F}(\Mcomp).
	\end{equation}
	Hence, even though, at each step of the construction, we lose some regularity, we obtain a solution that has a faster polynomially decaying error term.
	We obtain $\phi_{\mathrm{app}}$ by repeatedly solving one of the model operators in each turn.
	\begin{remark}[Admissible nonlinearities and weights]\label{in:rem:admissible}
		A nonlinearity $\mathcal{N}$ and decay rate $q\in\R$ can be treated perturbatively with the current iteration scheme if for $\phi\in\Aext{b}^{1,q,q,q}(\Mcomp)$ it holds that $\mathcal{N}[\phi]\in\Aext{b}^{(q+1,q+2,q,q)+\epsilon}(\Mcomp)$.
		This in itself is \emph{not} sufficient with our methods to to prove \cref{intro:thm:sem}, due to other top order regularity behaviour as explained near \cref{intro:eq:energy_top} and \cref{lin:rem:Morawetz}.
		In particular, we need that for $\phi\in\Aext{b}^{1,q,q,q}$ and $q_\bullet$ chosen appropriately the linearisation satisfies 
		\[
		D_\phi \mathcal{N} \mathring{\phi}\in \rho_\scri^{q_\scri+1}\rho_+^{q_++2}\rho_{\K}^{q_\K+1}\rho_{\F}^{q_\F}L^2(\Mcomp)\qquad \text{whenever }\Ve\mathring{\phi} \in\rho_\scri^{q_\scri}\rho_+^{q_+}\rho_{\K}^{q_\K}\rho_{\F}^{q_\F}L^2(\Mcomp),
		\]
		where $\Ve$ is as in \cref{lin:eq:Ve}.
	\end{remark}
	
	Let us explain a step of the iteration, with the base case being simpler.
	We induction on $N$ below.
	Assume there exists some $\phi_N$ such that
	\begin{equation}\label{intro:eq:ansatzN}
		f_N=\Box_{\ern}\phi_N-\mathcal{N}[\phi_N]\in \Aext{b}^{3,N+2,N,N-1}(\Mcomp).
	\end{equation}
	We write $f_N^+=f\chi_{>}(r)$, $f_N^\F=f\chi_{<}(r)$ for cutoff functions $\chi_>,\chi_<$ localising to a neighbourhood of $i^+,\F$ respectively.
	We find some functions $\phi_{N}^+\in\Aext{b}^{1,N,N,\infty}(\Mcomp)$, given in \cref{an:sec:i+-inverse,an:sec:across_horizon}, and (non-unique) $\phi_N^\F\in\Aext{b}^{\infty,\infty,N,N-1}(\Mcomp)$ solving
	\begin{equation}
		\Box_{\eta}\phi_N^+=f_N^+,\quad r\phi_N^+|_{\scri}=0;\qquad \Box_{g_{\mathrm{near}}}\phi_N^\F=f_N^\F
	\end{equation}
	within the support of $\chi_<,\chi_>$ respectively.
	We emphasise, that while $\phi^+_N$ is unique, $\phi^\F_N$ can be changed by an arbitrary smooth function that has the necessarily decay rate and produces no error term.
	Truncating $\phi_N^+,\phi^\F$, we obtain an improved ansatz $\phi_N'=\phi_N	+\phi_N^++\phi^\F$, satisfying 
	\begin{equation}
		f_N'=\Box_{\ern}\phi_N'-\mathcal{N}[\phi_N']\in \Aext{b}^{3,N+3,N,N}(\Mcomp).
	\end{equation}
	Next, we look for a function $\phi_N^\K\in\Aext{b}^{N+1,N,N}(\Mcomp)$, in \cref{an:sec:K-inverse}, satisfying
	\begin{equation}
		\hat{\Box}_{\ern}(0)\phi_N^\K=f_N'\chi((r-M)t_*)\chi(r/t_*),
	\end{equation}
	where the two cutoff functions localise to an open neighbourhood of $\K$.
	Truncating $\phi_N^\K$, we obtain $\phi_{N+1}=\phi_N'+\phi_N^\K$, that satisfies \cref{intro:eq:ansatzN} with $N$ replaced by $N+1$, i.e. $f_N\in\Aext{b}^{3,N+3,N+1,N}(\Mcomp)$.
	This improvement suffices to close the iteration scheme.
	
	\paragraph{Weighted $T$-energy estimate}
	This part of the proof essentially uses that $\ern$ describes an extremal black hole.
	The proof is also made easier by the existence of a causal Killing field $T$, which is present in the extremal Reissner-Nordstrom case.
	An important feature of $\Box_{\ern}$, respectively $\ern$, next to \cref{intro:eq:mapping} is the associated regularity gain at top order.
	Let us explain.
	We focus the discussion around the corner $\K\cap i^+$, but the same holds in all of $\Mcomp$, with appropriate modification.
	Let us denote by $\omega\in T^*(\sphere)$ one forms on the unit  sphere $\sphere$ .
	We can write 
	\begin{equation}
		\ern=r^2\Big(-D\frac{\dd t^2}{r^2}+D^{-1}\frac{\dd r^2}{r^2}+\slashed{g}\Big),
	\end{equation}
	so up to a factor $r^2$, we see that $\ern$ is a nondegenerate quadratic form in $\dd t/r,\dd r/r,\omega$ all the way to $\partial\Mcomp$.
	Following the geometric ideas of \cite{hintz_microlocal_2023-1}, we expect to be able to obtain a gain of regularity corresponding to the dual space $r\partial_t,r\partial_r,\Omega_{i}$.
	Let us denote the similarly obtained collection of finite number of vectorfield on $\Mcomp$ by $\Ve$, given explicitly in \cref{lin:eq:Ve}.
	
	Via only a $T=\partial_t$ energy estimate combined with some $r^p\partial_v$ estimates à la \cite{dafermos_new_2010}, we prove (see \cref{lin:prop:main}) for some $q_\scri,q_+,q_\K,q_\F$ 
	\begin{equation}\label{intro:eq:energy_top}
		\norm{\Ve\phi}_{\rho_\scri^{q_\scri}\rho_+^{q_+}\rho_\K^{q_\K}\rho_\F^{q_\F}L^2(\M)}\lesssim\norm{\Box_{\ern}\phi}_{\rho_\scri^{q_\scri}\rho_+^{q_++2}\rho_\K^{q_\K+1}\rho_\F^{q_\F}L^2(\M)},
	\end{equation}
	provided that $\Box_{\ern}\phi$ is compactly supported in $\M$.
	The simple proof of this result is due to $g_{\ern}(T,T)\leq0$, and crucially uses extremality.
	We can commute with (approximate) symmetries of $\Box_{\ern}$ to obtain \cref{intro:eq:energy_top} with $\Vb^k$ derivatives on both sides, as long as $k< q-10$.
	The requirement on $(k,N)$ is a manifestation, that higher commuted quantities experience an enhanced redshift/blueshift, and a remnant of this effect still exists in the extremal case.
	See \cref{lin:rem:blueshift} for more.

	Noting \cref{intro:eq:mapping}, we see that the estimate \cref{intro:eq:energy_top} is optimal, up to the $\rho_\K^{+1}$ factor, in terms of decay.
	The presence of this loss is due to trapping, or equivalently, due to the non-existence of an integrated local energy estimate without degeneration, see \cref{lin:rem:Morawetz} for more discussion.
	Connected to this, we already emphasise that we \emph{do not} use a Morawetz estimate for this part of the proof!
	
	\paragraph{Acknowledgement:}
	The authors would like to thank Mihalis Dafermos for useful sugesstions regarding the project. The second author would like to thank the Princeton Gravity Initiative, where much of this work was completed.
	The second author acknowledges the support of the SNSF starting grant TMSGI2 226018. 
	
	\paragraph{Overview:}
	The rest of the paper is structured as follows.
	We first introduce the geometric and analytic prerequisites in \cref{sec:setup}.
	Then, in \cref{sec:map} we record some important mapping properties of the operators associated to \cref{intro:eq:scattering}.
	\cref{sec:currents} contains standard divergence computations for energy currents.
	We use the energy currents in \cref{sec:an} to construct the approximate solution, and in \cref{sec:lin} to prove a linear estimate on polynomially weighted function spaces in $\M$.
	In turn, these linear estimates are used in \cref{sec:scat} to construction the exact solution of \cref{intro:eq:scattering} settling down to the approximate solution.
	We introduce the multi-black hole spacetimes in \cref{sec:multi}, and extend the previous result to this new class of manifolds.
	Finally, in \cref{sec:poly}, we prove the polyhomogeneous ansatz creation and in \cref{an:sec:subextremal} we show how to construct an ansatz in the subextremal case.

	\section{Geometric and analytic prerequisites}\label{sec:setup}
	In this section, we introduce the extremal Reissner Nordstrom ($\M$) family of spacetimes following closely \cite{angelopoulos_late-time_2020}.
	Then, we introduce a novel perspective on $\M$, by introducing a compactification $\Mcomp$ motivated by those appearing in \cite{hintz_sharp_2022,hintz_gluing_2023}.
	We also connect this to the conformal symmetry of $\M$ as studied by Couch-Torrence \cite{couch_conformal_1984}.
	
	\subsection{Spacetimes}\label{not:sec:spacetime}
	Let us first introduce the spacetime manifold in which \cref{intro:thm:rough} takes place, along with its associated Lorentzian metric.
	Let $\M_\mathfrak{t}'\subset\R_{v>0}\times\R_r\times \sphere$, for $\mathfrak{t} \in \mathbb{R}$, be a manifold defined for $v > 0$ and $ r\in (M-\mathfrak{t} v^{-1},\infty)$.
	We introduce
	\begin{equation}\label{not:eq:ERN_in_rv}
		\ern:=-D\dd v^2+2\dd v\dd r+r^2\slashed{g},\qquad \ern^{-1}=D\partial_r^2+2\partial_r\partial_v +r^{-2}\slashed{g}.
	\end{equation}
	where $D:=(1-\frac{M}{r})^2$ for some $M\in\R_{>0}$ and $\slashed{g} = d\theta^2 + \sin^2 \theta d \varphi^2$ is the round metric on $\sphere$.
	We  write the spatial coordinate $x=r\omega\subset \R^3_x$ where $(r,\omega)\in \R_r\times \sphere_\omega$. 
	We  introduce the standard functions 
	\begin{equation*}
		t:=v-r_*,\quad u:=v-2r_*,
	\end{equation*}
	corresponding to asymptotically flat time, outgoing null coordinate and the tortoise function $r_*$ satisfying $\frac{ \partial r_*}{\partial r} =D^{-1}$, defined separately in $r>M$ and $r<M$ as
	\begin{equation*}
		r_*(r):=r-M+\frac{M^2}{M-r}+2M\log\Big(\frac{\abs{r-M}}{M}\Big) .
	\end{equation*}
	
	We denote the volume form by $\mu=\sqrt{-\det(\ern )} dvdrd\theta d\varphi$.
	Let us introduce the smooth time function $t_*=v+h(r)$ where $h\in C^\infty(\R)$ satisfies
	\begin{equation}\label{not:eq:t*}
		h(r)=\begin{cases}
			-(r-M)& r<10M\\
			-2r_* +\frac{1}{\log r}& r>20M
		\end{cases}
		\implies \dd t_*=\begin{cases}
			\dd v-\dd r& r<10M\\
			\dd u-\frac{\dd r}{r\log^2r }& r>20M
		\end{cases}
	\end{equation}
	and extended smoothly in between so that $\dd t_*$ is timelike for $\ern$.

	In the $(v,r,\theta, \varphi)$ coordinate system, we define $T=\partial_v$  which is a future causal Killing field, the rotation vectorfields
	\begin{equation*}
		\begin{split}
			& \Omega_1 = \sin \varphi \partial_{\theta} + \cot \theta \cos \varphi \partial_{\varphi} , \\ & 
			\Omega_2 = - \cos \varphi \partial_{\theta} + \cot \theta \sin \varphi \partial_{\varphi} , \\ & 
			\Omega_3 = \partial_{\varphi} ,
		\end{split}
	\end{equation*}
	and $Y=-\partial_r$ which is future directed null vectorfield.
	
	Clearly \cref{not:eq:ERN_in_rv} shows, that the spacetime $(\M'_{\mathfrak{t}},\ern)$ can be extended to  $v\in\R$ and $r>0$, however in this paper we will mostly be concerned with $\M_\mathfrak{t}:=\M_\mathfrak{t}'\cap\{t_*>1\}$. 
	The standard depiction of the spacetime $\M_\mathfrak{t}$ is via a conformal map shown on the Penrose diagram of \cref{fig:penrose}.
	Moreover, it will be useful to split the spacetime into an interior, $\M_{\mathfrak{t},\mathrm{in}}:=\M_{\mathfrak{t}}\cap\{r\leq M\}$ and an exterior $\M=\M_0$. 
	
	For the sake of the reader, let us record in different coordinate systems the explicit form of $\ern$:
	\begin{subequations}\label{not:eq:ern_differents_reps}
		\begin{align}
			\ern&=-D\dd t^2+D^{-1}\dd r^2+r^2\slashed{g},& \ern^{-1}&=-D^{-1}\partial_t^2+D^{-1}\partial_r^2+r^{-2}\slashed{g}^{-1},\\
			\ern&=-D\dd u^2-2\dd u\dd r+r^2\slashed{g},& \ern^{-1}&=D\partial_r^2-2\partial_r\partial_u+r^{-2}\slashed{g}^{-1}\label{not:eq:ern_ru},
		\end{align}
		\begin{align}
			& \ern=-D\dd t_*^2+2(h'D+1)\dd t_*\dd r-h'(Dh'+2)\dd r^2+r^2\slashed{g}, \label{g_t*}\\ & 
			\ern^{-1}=-D\partial_{r}^2 +2(h'D+1)\partial_r \partial_{t_*}-h'(Dh'+2)\partial_{t_*}^2 +r^{-2}\slashed{g}^{-1} \label{ginv_t*}.
		\end{align}
	\end{subequations}
	and the corresponding d'Alembertian $\Box_{\ern}$
	\begin{nalign}\label{not:eq:wave_explicit}
		r\Box_{\ern}\phi=&2\partial_v\partial_r(r\phi)+\partial_r D\partial_r(r\phi)-r^{-1}D' r\phi+r^{-2}\slashed{\Delta}r\phi\\
		=&-2\partial_u\partial_r(r\phi)+\partial_r D\partial_r(r\phi)-r^{-1}D' r\phi+r^{-2}\slashed{\Delta}r\phi.
	\end{nalign}
	For the rest of the paper, it will be useful to introduce the following regions of spacetime
	\begin{equation}
		\mathcal{D}_{s_1}^{s_2}:=\{t_*\in(s_1,s_2)\},\quad \mathcal{D}_{s_1}:=\mathcal{D}_{s_1}^\infty,\quad \mathcal{D}^{s_2}:=\mathcal{D}^{s_2}_{1},\quad \Sigma_\tau=\{t_*=\tau\},\quad \B_{\mathfrak{t}}=\{r=M-t_*^{-1}\mathfrak{t}\}.
	\end{equation}

	\emph{Conformal symmetry.}
	A well-known property of the exterior of $\ern$ is that it possesses a conformal symmetry mapping the horizon $\mathcal{H}$ and null infinity $\mathcal{I}$ to one another. It was introduced by Couch-Torrence \cite{couch_conformal_1984} and is given by
	\begin{equation}
		\Phi_*(\ern)=\Omega^2\ern,\quad \Omega=\frac{M}{r-M},\quad \Phi:(t,r,\omega)\mapsto \left(t,M+\frac{M^2}{r-M},\omega\right).
	\end{equation}
	Note that $\Phi^*(r_*)=-r_*$, and $\Phi^2=\mathrm{Id}$.
	Using that $\ern$ has vanishing Ricci scalar, we get that for any $\phi$ solving the covariant wave equation, $\Box_{\ern}\phi=0$, $\tilde\phi=\Omega \Phi^*(\phi)$ also solves $\Box_{\ern}\tilde\phi=0$.
	More generally
	\begin{nalign}\label{not:eq:conformal_wave}
		\Box_{\ern}\phi=f&\implies \Box_{\ern} \Omega\Phi^*(\phi)=\Omega^3\Phi^*(f)\\
		\Box_{\ern}\phi&=\Phi^*\big(\Omega^{-3}\Box_{\ern}\Omega\Phi^*(\phi)\big).
	\end{nalign}
	
	\paragraph{Minkowski space.}
	
	We also introduce Minkowski space $(\R^{1+3}_{t,x},\eta)$, where $\eta=-\dd t^2+\dd x^2$.
	It will be convenient to introduce the analogue of \cref{not:eq:t*}, and set
	$t_*=t+h(r)$, where $h\in C^\infty(\R)$ satisfies
	\begin{equation}
		h(r)=\begin{cases}
			-r & r>20\\
			0 & r<10,
		\end{cases}
	\end{equation}
	and $\eta(\dd t_*,\dd t_*)\leq0$.
	The two different $t_*$ should not be confused, as they are defined on different manifolds. 
	
	\subsection{Analytic compactification}\label{not:sec:compactification}
	
	The extensive literature on solutions to the Einstein equations tells us that the confomal diagrammatic representation \cref{fig:penrose} of $\M_\mathfrak{t}$ is useful  when studying solutions to wave equations, as it contains the causal structure of $\M_\mathfrak{t}$.
	However, for some analytical purposes, it will be useful for us to use an alternative visualisation and corresponding compactification of $\M_{\mathfrak{t}}$.
	To this end, let us introduce the functions
	\begin{equation}\label{not:eq:boundary_defining}
		\rho_{\hor}=t_*(r-M),\quad \rho_\F=\frac{1+(r-M)t_*}{rt_*},\quad \rho_\K=\frac{r}{t_*\rho_\F},\quad\rho_{+}=\frac{r+t_*}{r t_*},\quad \rho_{\scri}=t_*/r.
	\end{equation}
	
	These functions are going to be the smooth boundary defining functions, i.e. $\bullet=\{\rho_\bullet=0\}$ where $\bullet\in\{\hor,\F,\K,\ip,\scri\}$, for the following boundaries:
	\begin{itemize}
		\item $\scri$, future null infinity, the endpoint of outgoing null geodesics;
		\item $\ip$, timelike infinity, the endpoint of outgoing hyperbolic massive geodesics; 
		\item $\K$, spacially compact timelike infinity, the endpoint of curves that are at a constant distance away from the black hole;
		\item $\F$, near horizon geometry and image of $\ip$ under the map $\Phi$;
		\item $\hor$, the future event horizon of the black hole and the image of $\scri$ under the map $\Phi$.
	\end{itemize}
	
	\begin{remark}[Artificial boundaries]
		Note that the surfaces $\B_{\mathfrak{t}} =\{r-M=t_*^{-1}\mathfrak{t}\}$ and $\Sigma_1=\{t_*=1\}$ are not part of $\M_{\mathfrak{t}}$ by definition.
		When partially compactifying $\M_{\mathfrak{t}}$, it will still not be closed near these regions.
		Alternatively, one can regard $\B_{\mathfrak{k}},\Sigma_1$ as part of $\M_{\mathfrak{t}}$, but remember that the b-differential structure is oblivious to these boundaries.
		Therefore, for a manifold with boundaries we may label some of the boundary components as \emph{artificial}, which means that they are not limiting hypersurfaces (see \cite[Definition 2.4]{hintz_stability_2020}).
	\end{remark}

	\begin{definition}\label{not:def:Mcomp}
		For $\mathfrak{t}=0$, we define the manifold with coners $\Mcomp$ as a compactification of $\M$ such that $\rho_\bullet$ with $\bullet\in\{\scri,+,\K,\F,\hor\}$ all extend to smooth boundary defining functions for their corresponding zero sets.\footnote{This construction can be performed via blow-ups of defining functions given that one start with a different comaptified manifold. See \cite{hintz_sharp_2022} for a similar construction.}
		(We call $\Sigma_1$ an artifical boundary.)
		
		For $\mathfrak{t}>0$, we define $\McompIn$ the manifold with corners as a compactification of $\M_{\mathfrak{t}}\setminus\M_0$ where $\rho_\F,\rho_\hor$ are extended smoothly to 0.
		(We call $\Sigma_1,\B_{\mathfrak{t}}$ artificial boundaries.)
		
		Finally we define $\Mcomp_\mathfrak{t}$ as the compactification of $\M_{\mathfrak{t}}$ with $\rho_\F,\rho_\K,\rho_+,\rho_\scri$ as in \cref{not:eq:boundary_defining} and $\rho_\F=(r-M)/r$.
		(We call $\Sigma_1,\B_{\mathfrak{t}}$ artificial boundaries.)
	\end{definition}
	
	$\Mcomp$ is depicted in \cref{fig:compactification}.
	The reason for introducing the above compactification is to free ourselves of the explicit choice made in \cref{not:eq:boundary_defining}.
	In particular, note that the manifold $\Mcomp$ does not depend on what smooth boundary defining functions we use, and we will make use of this freedom many times in computations, see already \cref{not:rem:examples}.

	\subsection{Function spaces}\label{not:sec:function_spaces}
	In this section, we introduce conormal spaces  for $\Mcomp$ following the general setup of \cite{grieser_basics_2001}.
	
	The solutions to \cref{intro:eq:scattering} that we are interested in will turn out to be regular with respect to vector fields that are tangent to the boundary of $\Mcomp$.
	In $(t_* , r , \theta , \varphi )$ coordinates, these are spanned by
	\begin{equation}\label{not:eq:Vb}
		\Vb=\Vb(\Mcomp):=\{(r-M)\partial_{r},t_* \partial_{t_*}, \Omega_i, i \in \{1,2,3\}\},
	\end{equation}
	over $\C^{\infty}(\Mcomp)$.
	
	Moreover, we require that the solutions considered to satisfy the regularity requirement of \cref{intro:thm:rough}, that is, we want the solutions to be regular with respect to $\partial_r$ near $\hor$.
	Via the conformal map $\Phi$, this corresponds to regularity with respect to $r^2\partial_r$ near $\scri$, also commonly called conformal regularity.
	The construction more generally is as follows:
	\begin{definition}[$\b$-operators]\label{not:def:b-ops}
		Given a compact manifold with corners $X$ defined by charts onto $\R^{n,2}_+:=[0,\infty)^2_{\rho_1,\rho_2}\times \R^n_y$, we write $\Diffb^1(X)$ for the space of vector fields that are tangent to the boundary of $X$ and let $\Vb(X)$ be a finite collection of these spanning $\Diffb^1(X)$ over $C^\infty(X)$. 
		Let $\rho_i$ for $i\in\{1,...,m\}$ be defining functions for each of the boundaries of $X$.
		We write $\Diffb^k(X)$ for a $k$-fold composition of $\Diffb^1(X)$. Note that if $X$ has artificial boundaries, $\Diffb^1(X)$ need \emph{not} be tangential towards these. 
	\end{definition}
	
	\begin{remark}[Examples]\label{not:rem:examples}
		We provide a collection of computations for the vectorfield $\partial_v|_r$ to show how to deduce its weight as a member of $\Diffb$. 
		Throughout, we use the liberty to pick different boundary defining functions when working in a neighbourhood of each of the corner of $\Mcomp$.
		\begin{nalign}
			\partial_v|_r&=\rho_{\F}(\rho_\hor\partial_{\rho_\hor}-\rho_\F\partial_{\rho_\F}) ,&&\text{near }\hor\cap\F,&&\text{with }\rho_\F=v^{-1}, \rho_\hor=(r-M)v; \\
			\partial_v|_r&= -\rho_\K\rho_\F(\rho_\K\partial_{\rho_\K}),&&\text{near }\F\cap\K,&&\text{with }\rho_\K=\rho_\F^{-1}v^{-1}, \rho_\F=(r-M); \\
			\partial_v|_r&= -\rho_\K\rho_+(\rho_\K\partial_{\rho_\K}) ,&&\text{near }\K\cap i^+,&&\text{with }\rho_+=r^{-1}, \rho_\K=r/t;\\
			\partial_v|_r&=\rho_+\left((1-\rho_\scri)\rho_{\scri}\partial_{\rho_\scri}-\rho_+\partial_{\rho_+}\right) ,&&\text{near }i^+\cap \scri,&&\text{with }\rho_\scri=u/t, \rho_+=1/u;\\
			&\implies\partial_v|_r\in\rho_\F\rho_\K\rho_+\Diffb(\Mcomp).
		\end{nalign}
		Similarly, we also note the following weights
		\begin{equation}
			\partial_r|_{t_*}\in \rho_\hor^{-1}\rho_\F^{-1}\rho_+\rho_\scri\Diffb(\Mcomp),\,\Omega_i\in \Diffb(\Mcomp).
		\end{equation}
	\end{remark}
	
	Next, we introduce the function spaces that describe the expected behavior in $\Mcomp$.
	
	\begin{definition}[Conormal functions]\label{not:def:conormal}
		For $\vec{a}\in\R^5$, write $\rho^{\vec{a}}=\prod_\bullet \rho_\bullet^{a_\bullet}$ with $\bullet\in\{\scri,\ip,\K,\F,\hor\}$.
		We define the \emph{conormal space} of functions for $s\in\N$ and $\vec{a}\in\R^5$ as
		\begin{nalign}
			\A{b}^{s;(a_\scri,a_{\ip},a_\K,a_\F,a_\hor)}(\Mcomp):=\{f\in \rho^{\vec{a}}L^\infty(\Mcomp): \Vb^\alpha f\in\rho^{\vec{a}}L^\infty(\Mcomp) \, \forall \abs{\alpha}\leq s \},
		\end{nalign}
		where we used the notation $\Vb^{\alpha}f$ to stand for vectorfields from $\Vb$ acting on $f$ according to the multiindex $\alpha$.
		We also write $\A{b}^{\vec{a}}(\Mcomp):=\A{b}^{\infty;\vec{a}}(\Mcomp)$.
		
		We define $\Aext{b}^{s;(a_\scri,a_{\ip},a_\K,a_\F)}(\Mcomp)=\A{b}^{s;(a_\scri,a_{\ip},a_\K,a_\F)}(\Mcomp\setminus\hor)$ to be the function space, with $\hor$ an artificial boundary. 
		
		More generally, given a manifold with boundary $X$ defined by charts onto $\R^{n,2}_+:=[0,\infty)^2_{\rho_1,\rho_2}\times \R^N_y$, boundary defining functions $\rho_i$, regularity index $s\in\N$, weights $a_i$, and a function $f:X\to\R$, we say that $f\in\A{b}^{s;\vec{a}}(X)$ is conormal with regularity $s$ and weight $\vec{a}$ if $\rho^{\vec{a}}\Diffb^s(X) f\in L^\infty(X)$.
		
		Finally, we define $\A{b}^{s;\vec{a}-}(X)=\bigcup_{\epsilon>0}\A{b}^{s;\vec{a}-\epsilon}$ and $\A{b}^{s;\vec{a}+}(X)=\bigcap_{\epsilon>0}\A{b}^{s;\vec{a}+\epsilon}$.
	\end{definition}
	Note that we naturally have $\scri,\hor\cong \sphere\times\overline{\R^+}$; $\K\cong\Rcomp$, and we write $\A{b}^q (\mathcal{I}),\A{b}^q (\mathcal{H})$, with $\scri,\mathcal{H}$ compactified with respect to $\rho_+|_{\scri}\sim1/u$ and $\rho_\F|_{\hor}\sim1/v$ respectively.
	
	\begin{notation}\label{not:not:indices}
		We will always use the convention that the indices are written from the far region towards the interior in order.
		This applies to $\Mcomp$, but also to the boundary components of $\Mcomp$, for instance, we write $\A{b}^{a_\scri,a_\K}(\ip)$ for the space of functions on $i^+\subset\partial \Mcomp$ that have $(1-\abs{x}/t)^{a_\scri}$ decay towards $\scri\cap \ip$ and $(\abs{x}/t)^{a_\K}$ decay towards $K\cap\ip$.
	\end{notation}

	We introduce some more manifolds on which part of the analysis takes place.
	\begin{definition}
		We write
		\begin{itemize}
			\item  $\Rcomp$ for the compactification of Euclidean space with the origin removed, $\R^3_\times:=\R^3\setminus\{0\}$, where we use $r$ and $r^{-1}$ as boundary defining function for the two boundary components at $\{r = 0 \}$ and infinity respectively;
			\item  $\MMink$ for the compactification of Minkowski space, $\R^{3+1}$, with boundary defining functions $\rho_+=\frac{t}{(t-\abs{x})\jpns{x}}$ for $\ip$, $\rho_ \K=\jpns{x}/t$ for spatially compact infinity, and $\rho_\scri=1-\rho_\K$ for null infinity.
			\item $\MMinkempty$ for the compactification of Minkowski space, $\R^{3+1}$, with boundary defining functions $\rho_+=(t-\jpns{x})^{-1}$ for $\ip$ and $\rho_\scri=1-\jpns{x}/t$ for null infinity.
		\end{itemize}
	\end{definition}

	\subsubsection{Energy spaces}
	
	We also introduce some $L^2$ based analogues of the $\A{b}$ spaces for ease of notation, though sometimes we will write out the corresponding definitions more explicitly wherever we deem helpful for understanding.
	
	It will be convenient to introduce a compact notation for the norm of commuted quantities.
	\begin{notation}(Multi-index)
		Let $X$ be a manifold with corners.
		For $\mathcal{V}=\{\Gamma_1,\Gamma_2,...,\Gamma_n\}$ a finite collection of operators and $Y$ a norm on $C^\infty(X)$,  we write
		\begin{equation}
			\norm{\mathcal{V}^k\phi}_{Y}:=\sum_{\abs{\alpha}\leq k}\norm{\Gamma^\alpha\phi}_Y,
		\end{equation}
		where $\alpha=(\alpha_1,...,\alpha_n)$ is a multi index and $\Gamma\in\mathcal{V}$.
	\end{notation}
	
	Note that the volume form $\mu$ can also be written as $\mu=\rho_{\scri}^{-3}\rho_{\ip}^{-4}\rho_\K^{-1}\rho_\F^{0}\rho_{\hor}^{1} \mu_\b $ where $\mu_\b$ is a non-degenerate b-density, i.e. near each boundary $\mu_b$ is of the form $\frac{\dd \rho_\bullet }{\rho_\bullet} f(\zeta,\rho_\bullet)$ where $f\in C^\infty(\Mcomp)$  and $f(\zeta,0)$ is a smooth non vanishing function on the boundary $\bullet\in\{\scri,\ip,\K,\F,\hor\}$.
	We highlight the weights between $\mu,\mu_\b$, as these are responsible for the weights appearing in the weighted Sobolev embeddings listed in \cref{not:eq:sobolev}.
	We define the norms for $f\in C^\infty_c(\M)$ and $s\in\N$, $\vec{a}\in\R^5$
	\begin{nalign}
		\norm{f}_{L^2(\Mcomp)}^2 :=\int_{\M} f^2 \, \mu ;\quad 
		\norm{f}_{\Hb^{s;(a_\scri,a_{\ip},a_\K,a_\F,a_\hor)}(\Mcomp)}=\norm{\rho^{-\vec{a}} \Vb^sf }_{L^2(\Mcomp)},
	\end{nalign}
	where $\rho^{\vec{a}}=\prod_\bullet \rho_\bullet^{a_\bullet}$. 
	We define the corresponding function spaces $\Hb^{s;\vec{a}}(\Mcomp)$ as completion of  $C^\infty_c(\M)$ with respect to the above norms.
	We also define $\Hb^{s;\vec{a}+}:=\cap_{\epsilon>0} \Hb^{s;\vec{a}+\epsilon}$, $\Hb^{s;\vec{a}-}:=\cup_{\epsilon>0} \Hb^{s;\vec{a}-\epsilon}$.
	We similarly define on $\K$ the Sobolev spaces with respect to the measure $\mu_\K$ given from $\mu=\dd t\mu_\K$.
	
	More generally
	\begin{definition}
		
		Given a density $\mu_X$ on a manifold with corners $X$, we write $\Hb^{s;\vec{a}}(X)$ for the completion of $C_c^\infty(\mathring{X})$ under the norm $\norm{f}_{\Hb^{s;\vec{a}}(X)}:=\norm{\prod_{i}\rho_i^{-a_i} \Vb^s(X) f}_{L^2(X,\mu_X)}$.
		
		For $U\subset X$ such that $U\cap \{\rho_i=0\}=\emptyset$ for $i\in A\subset\{1,...,m\}$, we will omit the indices $i\in A$ of $a_i$ in $\Hb^{s;\vec{a}}(U)$.
	\end{definition}
	
	Using Sobolev embedding, we already get that for $\vec{a}_{\mathrm{ER}}=(3/2,2,1/2,0,-1/2)$
	\begin{nalign}\label{not:eq:sobolev}
		\Hb^{3;-\vec{a}_{\mathrm{ER}}}(\Mcomp)\subset L^\infty(\Mcomp)\subset \Hb^{0;-\vec{a}_{\mathrm{ER}}-}(\Mcomp),\\
		\implies \Hb^{s+3;\vec{a}-\vec{a}_{\mathrm{ER}}}(\Mcomp)\subset \A{b}^{s;\vec{a}}(\Mcomp)\subset \Hb^{s;\vec{a}-\vec{a}_{\mathrm{ER}}-}(\Mcomp).
	\end{nalign}
	In the interior we will also need the following improvement for $a>-1/2$
	\begin{equation}\label{not:eq:sobolev2}
		\{f\in \Hb^{3;(a_{\F},a+1/2)}(\McompIn):t_*^{-1}Yf\in \Hb^{3;(a_{\F},a+1/2)}(\McompIn)\}\subset \rho_{\F}^{a_{\F}}L^\infty((\McompIn)).
	\end{equation}

	\begin{notation}[Symbols]
		We include a list of symbols used in the paper, some of which are only introduced later:
		
		\begin{itemize}[noitemsep]
			\item We use $\chi$ for cutoff function, with lower index denoting to what regions it localises, for instance writing $\chi(r)_{>1}$ for a smooth cutoff function satisfying $\supp\chi'(r)_{>1}\supset\{r>1/2\}$ and $\chi(r)_{>1}|_{r>1}=1-\chi(r)_{>1}|_{r<1/2}=1$;
			\item $\T$  denotes the energy momentum tensor;
			\item $\Mcomp,\Mcomp_\mathfrak{t}$ are compactified manifolds with a metric with the interior being isometric to subsets of extremal Reissner-Nordstrom.
			$\scri,\ip,\K,\F,\hor$ are boundary components.
			\item $\Vb,\Ve$ are finite collection of vectorfields which generate $\Diffb,\Diffe$ algebra of vectorfields over $C^\infty(\Mcomp)$.
			\item $\A{},\Aext{}$  ($\Hb$) are $L^\infty$ ($L^2$) based conormal spaces.
		\end{itemize}
	\end{notation}
	
	\section{Mapping properties of the wave operator}\label{sec:map}
	
	In this section, we record some of the mapping properties of $\Box_{\ern}$ on $\A{b}(\Mcomp)$ and the corresponding model operators.
	
	\begin{definition}[Model operators]\label{map:def:model}
		We define the following operators:
		\begin{subequations}\label{map:eq:model_op}
			\begin{align}
				\Normal{\K}&=\hat{\Box}_{\ern}(0):=\lim_{\sigma\to0}e^{\sigma t}\Box_{\ern} e^{-\sigma t}=r^{-2}\partial_r(Dr^2\partial_r\phi)+r^{-2}\slashed{\Delta}\phi,\\
				\Normal{\ip}&:=\Box_{\eta}=-2\partial_{u}\partial_r-2r^{-1}\partial_{u}+\Delta_r \\
				\Normal{\F}&:=\Box_{g_{\mathrm{near}}} =2\partial_{v}\partial_r+\partial_r (r-M)^2\partial_r+M^{2}\slashed{\Delta}
			\end{align}
		\end{subequations}
		where $\Delta_r$ denotes the Euclidean Laplacian  with respect to the $r$ radial coordinate and
		\begin{equation}
			g_{\mathrm{near}}:=-(r-M)^2\dd v^2+2\dd v\dd r+M^2\slashed{g},\quad g_{\mathrm{near}}^{-1}=(r-M)^2\partial_r^2+2\partial_v\partial_r+M^2\slashed{g}^{-1}.
		\end{equation}
	\end{definition}
	
	\begin{remark}[Conformal symmetry]\label{map:rem:isometry}
		A consequence of \cref{not:eq:conformal_wave} is that we can also write
		\begin{equation}\label{map:eq:isometry}
			\Normal{\F}=
			(r-M)^{-3}( -2\partial_{t_*}\partial_{\abs{y}}-2r^{-1}\partial_{t_*}+\Delta_{y} )(r-M)
			=(r-M)^{-3}(\Box_{\eta,y} )(r-M)
		\end{equation}
		where $y=(r-M)^{-1}\frac{x}{\abs{x}}$ and $\Box_{\eta,y}$ is the wave operator in $y,t_*$ coordinates in Minkowski space. 
	\end{remark}

	We next show, that these operators indeed capture the leading order behaviour of $\Box_{\ern}$.
	
	\begin{lemma}\label{map:lemma:normal_ops}
		For the wave operator on extremal Reissner--Nordstr\"{o}m it holds that
		\begin{equation}\label{map:eq:Box}
			\Box_{\ern}:\A{b}^{p_{\scri},p_+,p_\K,p_\F,p_\hor}(\Mcomp)\to\A{b}^{p_{\scri}+1,p_++2,p_\K,p_\F,p_\hor-1}(\Mcomp)
		\end{equation}
		and for $\bullet\in\{\K,\ip,\F\}$ the model operators defined in \cref{map:def:model} it holds that
		\begin{equation}\label{map:eq:Box-normal}
			\Box_{\ern}-\Normal{\bullet}:\A{b}^{p_{\scri},p_+,p_\K,p_\F,p_\hor}(\Mcomp)\to \rho_\bullet \A{b}^{p_{\scri}+1,p_++2,p_\K,p_\F,p_\hor-1}(\Mcomp).
		\end{equation}
		The same holds replacing everywhere $\A{b}$ for $\Aext{b}$ and removing the $\rho_\hor$ weights.
	\end{lemma}
	\begin{proof}
		These are direct computations.
		We give some details for the sake of completeness.
		Near the corner $\scri\cap \ip$, using \cref{not:eq:wave_explicit} can write
		\begin{equation}
			\Box_{\ern}=\Box_\eta+\rho_{\ip}^3\rho_\scri^2\Diffb^2(\Mcomp).
		\end{equation} 
		This proves the result in a neighbourhood of the corner.
		
		Near the corner $\ip\cap K$, we have 
		\begin{equation}
			\Box_{\ern}=-\partial_t|_{r}^2+\hat{\Box}_{\ern}(0)+\rho_\K^1\rho_{\ip}^3\Diffb^2(\Mcomp).
		\end{equation}
		Thus proving the result in a neighbourhood of  $\ip\cap K$.
		
		Near the corner $\K\cap\F$, we use \cref{not:eq:wave_explicit} to see that 
		\begin{equation}
			\Box_{\ern}=\Box_{g_{\mathrm{near}}}+\rho_\hor\rho_\F\Diffb^2(\Mcomp).
		\end{equation}
		This already yields the result.
	\end{proof}
	
	Let us also note that we may express the near horizon model operator with respect to $\mathfrak{r}=(r-M)/v$ and $\log v$ coordinates as
	\begin{equation}\label{map:eq:g_near_similarity_coord}
		\Box_{g_{\mathrm{near}}}=2\partial_{\mathfrak{r}}\partial_{\log v}+\partial_{\mathfrak{r}}(\mathfrak{r}^2+2\mathfrak{r})\partial_{\mathfrak{r}}+M^{-2}\slashed{\Delta}
	\end{equation}
	
	In \cref{sec:an}, we study invertibility properties of the model operators and use them to construct an approximate solution required to prove \cref{intro:thm:rough}.	
	\section{Current computations}\label{sec:currents}
	In this section, we present certain vector field computations used for energy estimates. 
	These results are well-known, and we include them here to align with our notation and provide clarity.
	The estimates are then derived by using the diverge theorem for the currents in this section.
	
	The energy--momentum tensor related to the wave operator is given by 
	$$\T_{\mu\nu}[\phi]:=\nabla_\mu\phi\nabla_\nu\phi-\frac{g_{\mu\nu}}{2}(\nabla_\sigma\phi\nabla^\sigma\phi). $$ 
	The two important properties of $\T$ that we use are: i) the explicit computation $(\Div\T)_\mu=\Box\phi \cdot \nabla_\mu$; ii) $\T$ is non-negative when contracted with future causal vectors, $\T[X,Y]\geq0$ for all $X,Y$ such that $g(X,X),g(Y,Y),g(X,Y)<0$.
	
	We begin by computing the contractions, with respect to the energy-momentum tensor, of the vector fields $\dd t^\#=\ern^{-1}(\dd t,\cdot)$ and $\dd t_*^\#$ with $T$.
	\begin{lemma}[Fluxes]
		We have the following coercivity of the energy on fixed $t_*$ and $t$ hypersurfaces in $\ern$:
		\begin{nalign}\label{curr:eq:fluxT}
			-\T[\dd t_*^{\#},T]&\sim\frac{(T\phi)^2}{r\log^2r}+D(\partial_r|_{t_*}\phi)^2+r^{-2}\abs{\slashed{\nabla}\phi} , \\
			-\T[\dd t^\#,T]&\sim\frac{(T\phi)^2}{D}+D(\partial_r|_{t_*}\phi)^2+r^{-2}\abs{\slashed{\nabla}\phi} ,
		\end{nalign}
		We also have for $r<2M$ that
		\begin{nalign}\label{curr:eq:fluxY}
			&-\T[\dd t^\#,Y]\sim (Y\phi)^2+D^{-1}\abs{\slashed{\nabla}\phi}^2, \\
			&-\T[\dd t_*^\#,Y]\sim (Y\phi)^2+\abs{\slashed{\nabla}\phi}^2.
		\end{nalign}
		
	\end{lemma}
	
	\begin{proof}
		We know that all terms are non-negative due to the causal nature of the vector fields.
		The rest is a simple computation:
		
		a)
		We focus on $\dd t_*$ first.
		In the region $r\in(2M,10M)$, \cref{curr:eq:fluxT} follows trivially, noting that both $T$ and $\dd t_*$ are strictly timelike.
		For $r>10M$ we have from \cref{not:eq:t*} that $-\dd t_*^\#=(1+D/(r\log^2r))\partial_r|_u+1/(r\log^2r)T$.
		The far regions follows by contracting $\T$.
		Similarly, for $r<2M$, we have $-\dd t_*^\#=(1+D)T-\partial_r|_v$.
		The near region again follows.
		
		We compute $-\dd t^\#=D^{-1}T$. 
		The result follows by contracting with $\T$.
		
		b) Follows again by contraction.
	\end{proof}
	
	Next, we study the bulk terms of $t$ weighted currents.
	These will be applied for backwards estimates.
	\begin{lemma}[Exterior]\label{curr:lem:ext}
		Let $N-\delta>2$ and $\delta>0$.
		We have the following divergence computation on $\ern$ for $\Box_{\ern}\phi=f$ in $\M$. 
		\begin{enumerate}
			\item Let $J^T_{N,\delta}=\frac{1}{N}t_*^{N-\delta}t^{\delta}  T\cdot\T$.
			Let $\mathfrak{F}=f^2\delta^{-1}t_*^{N-\delta}t^{\delta+1}D$. 
			Then
			\begin{equation}\label{curr:eq:T}
				-\Div J_{N,\delta}^T+\mathfrak{F}\gtrsim t_*^{N-1-\delta}t^\delta\Big(\frac{\delta t_*}{NDt} (T\phi)^2 +D(\partial_r|_{t_*}\phi)^2+r^{-2}\abs{\slashed{\nabla}\phi}^2\Big).
			\end{equation}
			\item Let  $\chi(r)$ be a localiser onto $\{r>10M\}$ and $J_{>,N}^1=\frac{1}{N}t_*^{N-1-\delta}t^\delta\chi r\partial_r|_u\cdot \T$.
			
			There exists $C(p,\delta)$ sufficiently large such that the following holds:
			for $\mathfrak{F}=f^2t_*^{N-\delta}t^\delta(C\delta^{-1}tD+N^{-1}r\chi)$
			\begin{equation}\label{curr:eq:p_far}
				-\Div( J_{>,N}^1+CJ^T_{N,\delta})+\mathfrak{F}\gtrsim t_*^{N-2-\delta} rt^\delta (\partial_r|_{t_*}\phi)^2\chi
			\end{equation}
			\item Let $\chi(r)$ be a localiser onto $\{r<2M\}$ and $J_{<,N}^1=\frac{1}{N}t_*^{N-1-\delta}t^{\delta}\chi (r-M)Y\cdot \T$.
			There exists $C(\delta)$ sufficiently large such that the following holds:
			for $\mathfrak{F}=f^2t_*^{N-\delta}t^\delta(\delta^{-1}t^{1}D+N^{-1}(r-M))$
			\begin{equation}\label{curr:eq:p_near1}
				-\Div( J^1_{<,N}+C J_{N,\delta}^T)+\mathfrak{F}\gtrsim t_*^{N-2-\delta}t^\delta\chi (r-M)(\partial_r\phi)^2.
			\end{equation}
		\end{enumerate}
		Moreover, all the currents above are future directed causal.
	\end{lemma}
	\begin{remark}[Forward estimate]
		The same currents can also be used for forward estimate by requiring that $N-\delta<0,\delta<0$.
		In this case, the currents become past directed causal, due to the explicit $N$ factor, and the sign in the boundary terms flips.
	\end{remark}
	\begin{remark}[$r^p$ hierarchy]
		For the reader familiar with the $r^p$ hierarchy framework of Dafermos-Rodnianski \cite{dafermos_new_2010} examplified in \cite{dafermos_quasilinear_2022} can note that \cref{curr:eq:p_far,curr:eq:p_near1} are an analogue of the $p=1$ estimates.
		Unlike in those works, we do not work with a twisted energy momentum tensor --alternatively with $r\phi$-- and recover the 0-th order control via Poincare type lemma in \cref{curr:lemma:poincare}.
	\end{remark}
	
	\begin{remark}[Fluxes vs weights]\label{curr:rem:flux_vs_weights}
		Let us already note that the reason we include $t,t_*$ weights in the estimates is purely for presentation.
		In particular, we will rarely make us of flux terms that the energy estimates control, in stark contrast to for instance \cite{dafermos_quasilinear_2022}.
		Let us give an example: not including a $t$ weight, and using the flux contribution to $J^T$ over incoming null hypersurfaces in a neighbourhood of $\scri$ yields the same bulk control over $T\phi$ as the divergence term \cref{curr:eq:T}.
		Similarly, the $t_*$ weight may be repackaged as an appropriate dyadic assumption in the terminology of \cite{dafermos_quasilinear_2022}.
	\end{remark}

	\begin{proof}
		a)
		We compute the deformation tensor
		\begin{equation}\label{curr:eq:T_def}
			\pi^{t_*^{N-\delta}t^\delta T}=Dt_*^{N-\delta}t^{\delta}(\delta t^{-1} \dd t\dd t+Nt_*^{-1}\dd t\dd t_*)
		\end{equation}
		We simply use $\Div J^T_{N,\delta}=t_*^{N-1-\delta}t^\delta(\dd t_*^{\#}+\frac{\delta t_*}{Nt}\dd t)\times T \cdot \T$, so that \cref{curr:eq:T} follows from \cref{curr:eq:fluxT}.
		The terms from the inhomogeneity are of the form
		$\abs{f N^{-1}t_*^{N-\delta}t^\delta T\phi}$, and can be bounded by the Young's inequality.
		
		b) For this part, it will be useful to work with $(r,u)$ coordinates and recall $\ern$ from \cref{not:eq:ern_ru}.
		
		We compute  the deformation tensors
		\begin{nalign}
			\pi^{\partial_r|_u}=\mathcal{L}_{\partial_r|_u}\ern=-D'\dd u^2+2r\slashed{g}\implies \pi^{r\partial_r|_u}=2r^2\slashed{g}-rD'\dd u^2-\dd r\dd u\\
			\pi^{t_*^{N-1}r\partial_r|_v}=t_*^N\pi^{r\partial_r|_v}+(N-1)t_*^{N-1}r\Big(\dd u^2-\frac{\dd u\dd r}{r\log^2r}\Big).
		\end{nalign}
		We note that for all but the $\slashed{g}$ term we may use \cref{curr:eq:T_def} to make them timelike:
		\begin{multline}
			2C \dd t\dd t_*-rD'\dd u^2-\dd r\dd u=C\dd t\dd u+C\dd t(\dd u-2\frac{\dd r}{r\log^2 r})-rD'\dd u^2-\dd r\dd u\\
			=\dd u \left(C\dd t-rD'\dd u-\dd r\right)+C\dd t(\dd u-2\frac{\dd r}{r\log^2 r}).
		\end{multline}
		This is clearly a sum of terms $\omega_1\otimes\omega_2$ for  future causal 1-forms $\omega_1,\omega_2$ and hence gives positive contribution when contracted with $\T$.
		Therefore, we obtain for $C$ sufficiently large in $r>10M$
		\begin{equation}\label{curr:eq:proof1}
			-\T\cdot\pi ^{t_*^{N-\delta}t^\delta CT+t_*^{N-1-\delta}t^\delta r\partial_r|_v}+t_*^{N-\delta-1}t^\delta\abs{\T\cdot r^2\slashed{g}}\gtrsim \frac{C}{2}\textrm{RHS} \cref{curr:eq:T}+RHS \cref{curr:eq:p_far}.
		\end{equation}
		We proceed to bound the term on the left using
		\begin{equation}
			\abs{\T\cdot r^2\slashed{g}}\lesssim r^{-2}\abs{\slashed{\nabla}\phi}^2+D\abs{\partial_r|_u\phi}^2+\abs{\partial_r\phi\partial_u\phi}.
		\end{equation}
		The first two terms on the right are bounded by $RHS\cref{curr:eq:T}$. For the last, we need to use in the region $r>t/2$
		\begin{equation}
			\abs{\partial_r\phi\partial_u\phi}\leq \frac{r}{t_*}\epsilon\abs{\partial_r\phi}^2+\frac{t_*}{r}\epsilon^{-1}\abs{\partial_u\phi}^2,
		\end{equation}
		and thus absorb it on the right hand side of \cref{curr:eq:proof1}.
		
		There is an additional error term from the inhomogeneity of the form $N^{-1}\abs{f}t_*^{N-\delta}t^\delta(\abs{T\phi}+\abs{t_*^{-1}\chi r^1\partial_r\phi}+\abs{t_*^{-1}\chi \phi})$.
		These terms are bounded by the Young's inequality to give \cref{curr:eq:p_far}.
		
		c) For this part, it is useful to work with $(v,r)$ coordinates and recall $\ern$ from \cref{not:eq:ERN_in_rv}.
		We note that $g(\partial_r|_v,\cdot )=\dd v$ and $\dd t_*=\dd v-\dd r$. 
		Let $\zeta=r-M$.
		We compute the deformation tensor
		\begin{multline}\label{curr:eq:proof3}
			\pi^{\partial_r|_v}=\mathcal{L}_{\partial_r|v}\ern=-D'\dd v^2+2r\slashed{g} \\
			\implies \pi^{\zeta t_*^N\partial_r|_v}=\zeta t_*^N(-D'\dd v^2 +2r\slashed{g})+t_*^N\dd v\dd r+\zeta Nt_*^{N-1} \dd v(\dd v-\dd r).
		\end{multline}
		To bound the terms other than $\slashed{g}$, we use \cref{curr:eq:T_def} and write for $C$ sufficiently large
		\begin{multline}
			CD\dd t\dd t_*-\zeta D'\dd v^2 +\dd v\dd r=C(D\dd v-\dd r)(\dd v-\dd r)+(D-\zeta D')\dd v^2+\dd v(\dd r-D\dd v)\\
			=C(D\dd v-\dd r)((1-C^{-1})\dd v-\dd r)+(D-\zeta D')\dd v^2\\
			=\omega +\frac{C}{2}\left(\frac{3D}{4}\dd v-\dd r+\frac{D}{4}\dd v\right)\left(\frac{\dd v}{2}-\dd r+\frac{\dd v}{2}\right)+(D-\zeta D')\dd v^2\\
			=\omega+ C\frac{D}{16}\dd v^2+(D-\zeta D')\dd v^2=\omega 
		\end{multline}
		where $\omega$ denotes a sum of product of future causal 1-forms changing from line to line.
		Therefore we obtain
		\begin{equation}
			\pi^{t_*^{N-\delta}t^\delta(CT-\zeta t^{-1} Y) }\cdot\T+\zeta t_*^{N-1-\delta}t^\delta\abs{\slashed{g}\cdot \T}\gtrsim CRHS\cref{curr:eq:T}+RHS\cref{curr:eq:p_near1}.
		\end{equation}
		We bound the second term on the left as
		\begin{equation}
			\zeta \T\cdot\slashed{g} \lesssim \zeta \abs{\slashed{\nabla}\phi}^2+\zeta D\abs{\partial_r\phi}^2+\zeta \abs{\partial_r\phi\partial_v\phi}.
		\end{equation}
		All the terms are bounded by RHS\cref{curr:eq:T}.
		Including the inhomogeneities again follows by Young's inequality.
	\end{proof}
	We note that by integrating the right hand side of \cref{curr:eq:T,curr:eq:p_far,curr:eq:p_near1} and ignoring $\delta$ and $N$ factors we obtain
	\begin{equation}
		\int\mu t_*^{N-\delta-1}t^{\delta} \Big(\frac{t_*}{tD} (T\phi)^2+\frac{D^{1/2}r}{t_*}(\partial_r|_{t_*}\phi)^2+\abs{\slashed{\nabla}\phi}^2r^{-2}\Big)\sim \norm{\mathring{\Ve}\phi}_{\Hb^{0;(\delta-1)/2,(N-3)/2,(N-1)/2,(N-1)/2,\delta/2}(\Mcomp)}
	\end{equation}
	with regularity $\mathring{\Ve}=\{(r-M)\partial_r|_{t_*},\rho_+^{-1}\rho_\F^{-1}T,\rho_\scri^{1/2}\rho_\hor^{1/2}\slashed{\nabla}\}$.
	Collecting the vectorfield in $\mathring{\Ve}$ as above also yields that we can immediately recover the 0-th order term from Poincare type estimate:
	
	\begin{lemma}[Poincare]\label{curr:lemma:poincare}
		Fix $a_+,a_\K,a_\F$ satisfying $a_\K>\max(a_+-3/2,a_\F+1/2)$ and $a_\F>0$, $a_+>-2$.
		For $\phi\in\Hb^{0;a_\scri,a_+,a_\K,a_\F,a_\hor}(\Mcomp)$, it holds that
		\begin{equation}
			\norm{\phi}_{\Hb^{0;a_\scri,a_+,a_\K,a_\F,a_\hor}(\Mcomp)}\lesssim\norm{\mathring{\Ve}\phi}_{\Hb^{0;a_\scri,a_+,a_\K,a_\F,a_\hor}(\Mcomp)}.
		\end{equation}
	\end{lemma}
	\begin{proof}
		We prove the result first in the far exterior ($i^+$) and near horizon ($\F$) regions, then propagate this control to the spatially compact part ($\K$).
		
		\emph{Far exterior:}
		In an open neighbourhood $U$ of the  corner $\scri\cap i^+$, we use coordinates $\rho_+,\rho_\scri,\omega$ for $\omega$ a function on the sphere.
		We observe that $\mathring{\Ve}$ yields control of $\rho_+\partial_{\rho_+}$.
		We also note the following inequality for $a<-1$
		\begin{equation}
			\int_{[0,1)_{\rho_+}}\dd \rho_+ \phi^2\rho_+^a\lesssim _a	\int_{[0,1)_{\rho_+}} \dd \rho_+ (\rho_+\partial_{\rho_+}\phi)^2\rho_+^a,
		\end{equation}
		where the constant only degenerates as $a\to -1$.
		Hence, $\norm{\phi}_{\Hb^{0;a_\scri,a_+}(U)}\lesssim_{a_+} \norm{\mathring{\Ve}\phi}_{\Hb^{0;a_\scri,a_+}(U)}$ for $a_+>-2$.
		This estimate can be extended to the region $\{\abs{x}/t>\delta\}$.
		
		\emph{Near horizon:}
		Using coordinates $\rho_{\F},\rho_{\K}$ and $x/\abs{x}$, we note that $\mathring{\Ve}$ controls the $\rho_{F}\partial_{\rho_{F}}$ derivative, and the previous argument goes through.
		
		\emph{Spatially compact:}
		We use the estimate for $M<r_1<r_2$ and $a>-3$
		\begin{equation}\label{curr:eq:proof_poincare}
			\int_{K\cap\{r\in(r_1,r_2)\}}r^a\phi^2\mu_\K\lesssim_{r_1} \int_{K\cap\{r\in(r_1,r_2)\}}r^a(r\partial_r\phi)^2\mu_\K+\int_{K\cap\{r\in(r_2/2,r_2)\}}r^a\phi^2\mu_\K
		\end{equation}
		Importantly, the implicit constant does not depend on $r_2$.
		Hence, we can use the estimate parametrically in $t$ in the region $U=\{\abs{x}\in(2M,t\delta)\}$ to obtain
		\begin{equation}
			\norm{\phi}_{\Hb^{0;a_+,a_\K}(U)}\lesssim \norm{r\partial_r\phi}_{\Hb^{0;a_+,a_\K}(U)}+\norm{\phi}_{\Hb^{0;a_+}(\{\abs{x}\in(t\delta/2,t\delta)\})}
		\end{equation}
		where $a_\K-a_+>-3/2$.
		
		The analogous estimate of \cref{curr:eq:proof_poincare} in the near horizon region takes the form 
		\begin{multline}
			\int_{K\cap\{r\in(M+r_1,r_2)\}}(r-M)^a\phi^2\mu_\K\lesssim_{r_2} \int_{K\cap\{r\in(M+r_1,r_2)\}}(r-M)^a\big((r-M)\partial_r\phi\big)^2\mu_\K\\
			+\int_{K\cap\{r\in(M+r_1/4,M+r_1/2)\}}(r-M)^a\phi^2\mu_\K
		\end{multline}
		for $r_1>0$, $r_2>2M+r_1$ and $a<-1$.
		The implicit constant does not depend on $r_1$.
		Hence, in the region $U=\{\abs{x}\in (M+\delta s^{-1},2M)\}$ for $a_\F-a_\K<-1/2$ it holds that
		\begin{equation}
			\norm{\phi}_{\Hb^{0;a_\K,a_\F}(U)}\lesssim\norm{(r-M)\partial_r\phi}_{\Hb^{0;a_\K,a_\F}(U)}+\norm{\phi}_{\Hb^{0;a_\F}(\{t(r-M)\in(\delta/2,\delta)\})}.
		\end{equation}
		
	\end{proof}

	We will also need to use energy estimates in the interior of the black hole regions.
	\begin{lemma}[Interior]\label{curr:lem:int}
		Consider $\Box_{\ern}\phi=f$ in the region $M-r\in(0,  t_*)$.
		Let  $J^T_{N,\delta}=\frac{1}{N}t_*^{N+\delta}\abs{t}^{-\delta}T\cdot\T$ and $J_<^1=\frac{1}{N}t_*^{N-1+\delta}\abs{t}^{-\delta}(r-M) Y\cdot \T$.
		Let  $\mathfrak{F}=f^2t_*^{N-\delta}\abs{t}^\delta(\delta^{-1}\abs{t}^{1}D+N^{-1}(r-M))$.
		There exist $C$, such that the following  holds 
		\begin{equation}\label{curr:eq:interior_div}
			-\Div (J^1_<+CJ_{N,\delta}^T)\gtrsim t_*^{N-1+\delta}\abs{t}^{-\delta}\big(D^{1/2}t_*^{-1}(\partial_r\phi)^2+\frac{\delta t_*}{NDt}(T\phi)^2+ \abs{\slashed{\nabla}\phi}^2\big).
		\end{equation}
	\end{lemma}
	\begin{proof}
		In the region of interest, $t<0$.
		We note that $-\dd t_*,T$ are future directed causal and so is $-\dd \abs{t}^{-1}$.
		\cref{curr:eq:interior_div} follows from the same computation as for \cref{curr:lem:ext}.
		In particular, not that in \cref{curr:eq:proof1} both $\dd v^2$ terms have the same sign, thus we only need to use $J^T$ current to absorb the other terms.
		The resulting inhomogeneity is 
		$t_*^{N+\delta}\abs{t}^{-\delta}\abs{f}(\abs{T\phi}+\abs{D^{1/2}t_*^{-1}\partial_r\phi})$, which can be bounded by the Young's inequality using the control provided by the divergence terms.
	\end{proof}

	In the interior, we can recover 0-th order terms similarly:
	\begin{lemma}[Hardy-type inequality]\label{curr:lemma:hardy}
		For $\phi\in\Hb^{0;q,a_\hor}(\McompIn)$ with $q>0$ it holds that
		\begin{equation}
			\norm{\phi}_{\Hb^{0;q,a_\hor}(\McompIn)}\lesssim_q\norm{\mathring{\Ve}\phi}_{\Hb^{0;q,a_\hor}(\McompIn)}.
		\end{equation}
	\end{lemma}
	\begin{proof}
		The proof is already contained in \cref{curr:lemma:poincare}.
	\end{proof}
	
	Finally, we record a Morawetz estimate, which will only be used in the construction of the approximate solution in \cref{sec:an}.
	In fact, as explained in \cref{an:rem:Morawetz}, even for this step it is not necessary. 
	\begin{lemma}[Morawetz current in Minkowski space]\label{curr:lem:morawetz}
		Let $N-\delta<0$ and $\delta<0$.
		In $(\MMink,\eta)$, define $J^T_{N,\delta}=\frac{1}{N}t_*^{N-\delta}t^{\delta}  T\cdot\T$  together with $J^X_N=\frac{1}{N}t_*^N\big(f(r)\partial_r\cdot\T-\phi^2\dd \chi/2+\chi\dd\phi^2/2\big)$ where $f=1-\jpns{r}^{-1}$ and $\chi=r^{-1}f$.
		Then, there exists $C$ sufficiently large such that $CJ^T+J^X$ is causal up to 0-th order terms and
		\begin{multline}\label{curr:eq:mor_minkowski}
			\Div(CJ_{N,\delta}^T+J_N^X)+t_*^{N-1-\delta}t_*^\delta\jpns{r}^{-2} \phi^2\\
			\gtrsim  \min(Ct_*^{-1},r^{-2})t_*^{N-\delta}t_*^\delta\Big(\delta t_*t_*^{-1} (T\phi)^2 +(\partial_r|_{t_*}\phi)^2+r^{-2}\abs{\slashed{\nabla}\phi}+\phi^2/\jpns{r}^2\Big).
		\end{multline}
		
	\end{lemma}
	
	\begin{proof}
		For $N=0$, define $J^T_{0,\delta}=t_*^{0-\delta}t^{\delta}  T\cdot\T$  together with $J^X_0=t_*^0\big(f(r)\partial_r\cdot\T-\phi^2\dd \chi/2+\chi\dd\phi^2/2\big)$ without the $1/N$ prefactors.
		
		We compute the deformation tensor in $(t,r)$ coordinates
		\begin{equation}
			\pi^{f\partial_r}=f\pi^{\partial_r}+\dd f\times \dd r=fr\slashed{g}+f'\dd r^2
		\end{equation}
		and the corresponding divergence is thus
		\begin{multline}
			\Div J^X_0=\pi^{f\partial_r}\cdot \T-\frac{1}{2}\phi^2\Box_\eta\chi+\chi\partial\phi\cdot\partial\phi+\big(\chi\phi+f\partial_r(\phi)\big)\Box_\eta\phi=\frac{f'}{2}\Big((T\phi)^2+(\partial_r\phi)^2\Big)\\
			-\Big(\frac{f}{2r^2}\Big)'\abs{\slashed{\nabla}\phi}^2-\frac{f''\phi^2}{2r}+\big(\chi\phi+f\partial_r(\phi)\big)\Box_\eta\phi.
		\end{multline}
		With $f=1-\jpns{r}^{-1}$ we obtain coercive estimates for $\partial\phi$.
		
		The causal property follows from the fact that $J^T$ controls all derivatives uniformly.
		
		Finally, the result for $J^X_N+J^T_{N,\delta}$ follows from the $N=0$ case together with the causal property and bounding the extra 0-th order term.
	\end{proof}
	
	\section{Ansatz}\label{sec:an}
	In this section, we prove the first part of \cref{intro:thm:sem}.
	Recall that $\mathcal{N}[\phi]:=\ern(\dd\phi,\dd \phi)+\phi^3$.
	We construct an approximate solution, that attains the required data at null infinity \emph{and} remains smooth up to the horizon:
	
	\begin{theorem}\label{an:thm:ansatz}
		Given $q\in\R_{>3/2}$ and $\psi_\scri\in\A{b}^{q-1}(\scri)$ there exists a function $\phi\in\A{b}^{1,(q,q,q)-}(\Mcomp_\mathfrak{t})$, that we will call the \emph{ansatz}, satisfying
		\begin{equation}\label{an:eq:ansatz}
			\Box_{\ern}\phi-\mathcal{N}[\phi]\in\A{b}^{3,\infty,\infty,\infty}(\Mcomp_\mathfrak{t}),\qquad r\phi|_{\scri}=\psi_\scri,\qquad \phi|_{\hor}=\phi_\hor\in\A{b}^{q}(\hor).
		\end{equation}
	\end{theorem}

	\begin{cor}[Non-zero Aretakis constant]\label{an:cor:Aretakis}
		Let $\psi_\scri,q$ be as in \cref{an:thm:ansatz}, and set
		\begin{equation}\label{an:eq:phi_p=1}
			\phi_1=\chi_{<2M}(r)t_*^{-1}-\chi_{<2M}(r)\frac{r-M}{(r-M)t_*+2}\in\A{b}^{\infty,\infty,2,1}(\Mcomp_{\mathfrak{t}}).
		\end{equation}
		For any $c\in\R$ there exists $\phi\in c\phi_1+\A{b}^{1,(q,q,q)-}(\Mcomp_\mathfrak{t})$ and $\phi_\hor\in ct_*^{-1}+\A{b}^{q-}(\hor)$ solving \cref{an:eq:ansatz}.
	\end{cor}
	
	We prove the theorem in \cref{an:sec:iteration} after having studied the invertibility of each model operator from \cref{map:lemma:normal_ops}.
	
	\subsection{Solvability on the spatially compact face $\K$}\label{an:sec:K-inverse}
	
	In this first section we  study the solvability of the model problem at $\K\cong\Rcomp$.
	That is, for fixed $f\in\A{b}^{3,0}(\K)$, i.e. (see \cref{not:not:indices}) $f\sim r^{-3}$ as $r\to\infty$ and $f\sim 1$ as $r\to M$, we are looking for solutions to 
	\begin{equation}\label{an:eq:K_eq}
		\Normal{\K}\phi=f.
	\end{equation}
	Note that this has already been studied in \cite[Section 10]{angelopoulos_non-degenerate_2020},  but we include a proof for completeness and adjustment to the notation here.

	The main result is the following:
	\begin{prop}[Invertibility of $\Normal{\K}$]\label{an:prop:K}
		Fix $a_+\in(0,1)$, $a_\F\in(-1,0)$.
		For $f\in\A{b}^{a_++2,a_\F}(\K)$ there exists a unique $\phi\in\A{b}^{(a_+,a_\F)-}(\K)$ solving $\Normal{\K} \phi=f$.
	\end{prop}
	
	\begin{proof}
		Note that \cref{an:eq:K_eq} is equivalent to the study of $\Box_{\ern}\phi=f$ for time independent $\phi$ and $f$.
		The result here is a slight adaptation of \cite[Theorem 3.1]{hintz_lectures_2023}.
		We first prove a priori estimates for the operator $\Normal{\K}$ for smooth compactly supported functions and extend these by density.
		Then, we show that the adjoint has no kernel and thus obtain invertibility.
		
		\emph{Step 1, existence:}
		First, we note that the $\partial_t$ energy estimate for $\T$ yields a coercive control for solutions of \cref{an:eq:K_eq}: integrating by parts $\phi\Normal{\K}\phi$ with $\phi$ compactly supported gives
		\begin{subequations}
			\begin{align}
				\int_\K \phi \big(r^{-2}\partial_r ( D r^2 \partial_r\phi ) + r^{-2}\slashed{\Delta}\phi\big)\mu_\K=-\int_\K  D(\partial_r\phi)^2+\abs{\slashed{\nabla}\phi}^2/r^2\mu_\K\\
				\implies
				\int_\K D(\partial_{r}\phi)^2+\abs{\slashed{\nabla}\phi}^2/r^2+(\phi/r)^{2}\mu_\K\lesssim 	\int_\K f^2r^2\mu_\K,\label{an:eq:elliptic_T}
			\end{align}
		\end{subequations}
		where we used Hardy inequality to obtain the zero-th order term and remember that $\mu_\K=r^2\dd r\dd\abs{\slashed{g}}$.
		We may express \cref{an:eq:elliptic_T} as $\norm{\phi}_{\Hb^{1;-1,0}(\K)}\lesssim\norm{f}_{\Hb^{0;1,0}(\K)}$.
		Integrating $r^{2}(\Normal{K}\phi)^2$ improves this estimate to
		\begin{equation}
			\int_\K r^{-2}(\partial_r (  D r^2 \partial_r\phi )+\slashed{\Delta}\phi)^2\mu_\K\lesssim \int_\K f^2r^2\mu_\K+\int_\K D(\partial_{r}\phi)^2+\abs{\slashed{\nabla}\phi}^2/r^2+(\phi/r)^{2}\mu_\K
		\end{equation}
		implying $\norm{\phi}_{\Hb^{2;-1,0}(\K)}\lesssim\norm{f}_{\Hb^{0;1,0}(\K)}$ after 2 further integration by parts on the left hand side.
		
		Using the self-adjointness of $\Normal{\K}$, this already yields that a solution $\phi\in\Hb^{2;-1,0}(\K)$ exists provided that $f\in \Hb^{0;1,0}(\K)\subset \A{b}^{0;5/2,-1/2}(\K)$:
		indeed, we may compute what is the kernel $\mathrm{Ker}(\Normal{\K}|_{X})$ restricted to the adjoint space of the image, $X=(\Hb^{0;1,0}(\K))^*=\Hb^{0;-1,0}(\K)$.
		For $\phi_{0}\in\mathrm{Ker}(\Normal{\K}|_X)$, by elliptic regularity, we have $\phi_{0}\in \Hb^{\infty;-1,0}(\K)$, and so, we may use the same integration by parts as in \cref{an:eq:elliptic_T} and keeping the boundary terms of \cref{an:eq:elliptic_T}, we get
		\begin{equation}
			\int_{\K\cap(r_1,r_2) } D(\partial_{r}\phi_{0})^2+\abs{\slashed{\nabla}\phi_{0}}^2/r^2+(\phi_{0}/r)^{2}\lesssim \int_{\K\cap\{(r_1/2,2r_1)\cup(r_2/2,2r_2)\}} D(\partial_{r}\phi_{0})^2+\abs{\slashed{\nabla}\phi_{0}}^2/r^2+(\phi_{0}/r)^{2}
		\end{equation}
		Taking $r_1\to M$ and $r_2\to \infty$, we get that the right hand side vanishes, and so $\phi_{0}\equiv0$. In particular, the kernel is empty and $\Normal{\K}$ is invertible.

		\emph{Step 2, weights:}
		Let's first consider $a_+<1/2,a_\F<-1/2$.
		We observe, that \cref{an:eq:elliptic_T} also yields a coercive estimate when using $r^{-\alpha}(1-M/r)^\beta$ with $\alpha,\beta\in(0,1)$.
		Thus we get $\norm{\phi}_{\Hb^{1;a_+-3/2,a_\F+1/2}(\K)}\lesssim\norm{f}_{\Hb^{0;a_++1/2,a_\F+1/2}(\K)}$ for $a_+\in(0,1/2)$ and $a_\F\in(-1,-1/2)$.
		
		For $a_+>1/2,a_\F>-1/2$ we use $r^p\partial_r$ estimates:
		Fix $p=(1,2a_+)$.
		Let $\chi$ be a cutoff localising to $r>10M$  and consider the energy estimate arising from the multiplier $V=r^p\chi\partial_r$ for $\Box_{\ern}$.
		That is, we integrate by part the expression $V(\phi)\Normal{\K}\phi$ to get
		\begin{equation}
			\int_\K \chi pr^{p-1}(\partial_r\phi)^2+(p-2)r^{p-3}\abs{\slashed{\nabla}\phi}^2\mu_\K\lesssim \mathrm{LHS}\cref{an:eq:elliptic_T}+\int_\K f^2 r^{p+1}\mu_\K.
		\end{equation}
		This already yields $\phi\in\Hb^{1;a_+-3/2,0}(\K)$.
		Using an $(r-M)^p\partial_r$  multiplier near the horizon with $p\in(1+2a_\F,1)$, we get $\phi\in\Hb^{1;a_+-3/2,1/2+a_\F}(\K)$.

		Finally, integrating $D^{p_\F}r^{p_+}(\Box_{\ern}\phi)^2$ with $p_+\in(2,1+2a_+)$, $p_\F\in(a_\F,0)$ similarly yields that $\phi\in\Hb^{2;a_+-3/2;1/2+a_\F}(\K)$.
		
		Given these a priori estimates, the existence follows by duality as before.
		
		\emph{Step 3, higher regularity:}
		Let us note that for $\Vb(\K)=\{(r-M)\partial_r,\Omega_{ij} \}$, it holds that
		\begin{equation}
			[\Normal{\K},\Vb]=\rho_+^2\Diff^2_\b(\K).
		\end{equation}    
		Therefore commuting with $\Vb$, we obtain that $\phi\in\Hb^{\infty;a_+-3/2-,1/2+a_\F}(\K)$.
		Indeed, by induction we get
		\begin{equation}        \Normal{\K}\Vb^k\phi=\rho_+^2\Diffb^{k+1}(\K)\phi+\Vb^kf\implies\phi\in\Hb^{2+k;a_+-3/2,1/2+a_\F}(\K).
		\end{equation}
		Via \cref{not:eq:sobolev}, this already yields $\phi_2\in\A{b}^{(a_+,a_\F)-}(\K)$. 
		
	\end{proof}
	
	\begin{cor}\label{an:cor:K}
		Fix $p\in\R$ and $p_+\in(p,p+1)$, $p_\F\in(p-1,p)$.
		Let $U_c=\{\rho_\K\in(-c,c)\}\cap\Mcomp$ and $f\in\A{b}^{p_++2,p,p_\F}(U_{1/4})$  with $\supp f\subset U_{1/4}$.
		Then, there exists $\phi\in\A{b}^{(p+1,p,p)-}(U_{1/4})$ with $\supp\phi\subset U_{1/4}$ such that 
		\begin{equation}
			\Box_{\ern}\phi-f\in\A{b}^{(p_i+2,p+1,p_\F)-}(U_{1/4}).
		\end{equation}
	\end{cor}
	\begin{proof}
		Let us use coordinates $t_*,x$ and consider $f\in\A{b}^p(\R_{t_*},\A{b}^{p_++2-p,p_\F-p}(\K))$ as functions in $\A{b}(\K)$ parametrised by $t_*$.
		Inverting $\Box_{\ern}$ parametrically in $t_*$, we write
		\begin{equation}
			\phi(t_*,x):=\Normal{\K}^{-1}f(t_*,x)\in \A{b}^p(\R_{t_*},\A{b}^{p_+-p-,p_\F-p-}(\K))
		\end{equation}
		where the inverse is given by \cref{an:prop:K}.
		Using \cref{map:lemma:normal_ops} and localising $\phi$ to $U_{1/4}$ yields the result.
	\end{proof}

	The above proposition is closely related to the elliptic estimates that are used to recover all derivatives given that we can control the $T$ derivative, after using the equation.
	
	\begin{lemma}[$\K$ elliptic estimate]\label{ell:lemma:near}
		Let $\Box_{\ern}\phi=f$ in $\M$, and fix $r_1\in(0,1)$, $r_2\in(2M,\infty)$.
		Let $J^T=T\cdot\T$.
		Then, there exists $C$ independent of $r_1,r_2$ such that on $\Sigma_\tau=\{t=\tau\}$, we have
		\begin{equation}
			\int_{\Sigma_\tau\cap \{r\in(M+r_1,r_2)\}}r^2\abs{D\partial^2\phi}^2\mu_\K\leq C
			\int_{\Sigma_\tau\cap \{r\in(M+r_1/2,2r_2)\}} (J^T[\phi]+r^2D^{-1}J^T[T\phi])+r^2f^2\mu_\K.
		\end{equation}
		Similar higher order analogues for $s\in\N_{\geq0}$ and $\mathcal{V}=\{1,r\partial_r,\Omega\}$
		\begin{equation}
			\int_{\Sigma_\tau\cap \{r\in(M+r_1,r_2)\}}r^2\abs{D\partial^2\mathcal{V}^s\phi}^2\mu_\K\lesssim_{s} 
			\int_{\Sigma_\tau\cap \{r\in(M+r_1/2,2r_2)\}} r^2\mathcal{V}^s f^2+\sum_{s'\leq s+1}r^{2s'}D^{-s'}J^T[T^{s'}\phi]\mu_\K.
		\end{equation}
	\end{lemma}
	
	\begin{proof}
		Follows from squaring both sides of $\Normal{\K}\phi=D^{-1}\partial_t^2\phi+f$ and integrating, together with the flux on $\Sigma_\tau$ given by \cref{curr:eq:fluxT} and \cref{an:prop:K}.
	\end{proof}

	\begin{cor}[$i^+$ elliptic estimate]\label{ell:lemma:ext}
		Let $\Box_{\ern} \phi=f$ and fix $0<a<b<1$ together with $U_c=\{c^{-1}<r/t<c\}$.
		Then we have the far region estimate
		\begin{equation}
			\norm{\partial\phi}_{\Hb^{1;l}(U_a)}\lesssim\norm{T\phi}_{\Hb^{1;l}(U_a)}+\norm{\phi}_{\Hb^{1;l-1}(U_a)}+\norm{f}_{\Hb^{0;l+1}(U_a)},
		\end{equation}
		and higher order analogue for $s\in\N_{\geq1}$
		\begin{equation}
			\norm{\partial\phi}_{\Hb^{s;l}(U_a)}\lesssim_{a,s}\norm{\{1,tT\}^s\phi}_{\Hb^{1;l-1}(U_a)}+\norm{f}_{\Hb^{s-1;l+1}(U_a)}.
		\end{equation}
	\end{cor}
	\begin{proof}
		This is a special case of \cref{ell:lemma:near} using that the inequalities have $r_1,r_2$ independent constants. 
	\end{proof}
	
	\subsection{Solvability on timelike infinity}\label{an:sec:i+-inverse}
	In this section, we study the invertibility of $\Normal{\ip}$, i.e. the Minkowski wave operator on functions with polynomial decay.
	
	\begin{prop}[Conormal inverse of $\Normal{\ip}$]\label{an:prop:i+_conormal}
		Fix $q_\K\in(q-1,q)$ and $q>3/2$.
		Let $f\in\A{b}^{3,q+2,q_\K}(\MMink)$ and $\psi^\scri\in\A{b}^{q-1}(\scri)$.
		Then, the unique scattering solution to 
		\begin{equation}
			\Box_\eta\phi=f,\qquad r\phi|_\scri=\psi^\scri,
		\end{equation}
		satisfies $\phi\in\A{b}^{1,(q,q_\K)-}(\MMink)$.
	\end{prop}
	\begin{remark}[Propagation of weight from $\ip$ to $\scri$]
		Provided that $f\in\A{b}^{\infty,q+2,q_\K}(\MMink)$, one can prove via energy estimates that $\phi\in\A{b}^{q,q,q_\K}(\MMink)$, and (already in spherical symmetry) this is almost sharp.
		That is, the solution $\phi$ will vanish to order $q$ towards $\scri$ and will be conormal, and not smooth with respect to $r^2\partial_r$ in a neighbourhood of $\scri$.
		This in particular shows, how conformal smoothness may be lost by imposing $\psi^\scri=0$.
	\end{remark}
	\begin{proof}
		\emph{Step 1, smoothing:}
		It suffices to consider $f\in\A{b}^{3,q+2,q+2}(\MMink)$. Indeed, let us solve parametrically in $t$ the ansatz $\Delta\phi_1=f\chi(r/t)$, where $\chi$ is localising to $r/t<1/2$.
		Using that $f\in\A{b}^{3,q+2,q_\K}(\MMink)$ we have $\chi f\in\A{b}^{q_\K}(\Rcomp;\A{b}^{q+2-q_\K}(\K))$, thus we need to invert $\Delta$ on functions with decay of rate $r^{-2-q+q_\K}$.
		We define	
		$$\phi_1=\chi \Delta^{-1}\chi f \in\A{b}^{\infty,q,q_\K}(\MMink). $$
		Thus, we compute that $\Box(\phi_1\chi)-f\in\A{b}^{q_\K}(\Rcomp;\A{b}^{q+2-q_\K}(\K))\subset \A{b}^{3,q+2,q_\K+2}(\MMink)$.
		Repeating this construction yields the simplification.
		
		We  note that writing $\phi_0=\phi-r^{-1}\chi(r/t)\psi^\scri$ for a cutoff function $\chi$ supported on $r/t>1/2$, we get that $\Box_\eta\phi_0=f+\A{b}^{3,q+2,\infty}(\MMink)$, so it suffices to consider the case $\psi^\scri=0$.

		\emph{Step 2, scattering:}
		Let us assume that $\supp f\subset \{t_*<T_f\}$ for some $T_f\gg1$.
		Provided that we can prove the result uniformly in $T_f$, we can remove this assumption on $\supp f$ by a limiting argument.
		Let $\mathcal{D}_s=\MMink\cap {t_*\geq 0}$.
		
		Consider the current $J^T_{N,\delta},J^X_N$ from \cref{curr:lem:morawetz} for $N>0,\delta>0$.
		Using \cref{curr:eq:mor_minkowski} and \cref{curr:lemma:poincare}, we already have	
		\begin{equation}\label{an:eq:i+_proof}
			\begin{split}
				\int_{\mathcal{D}_s} \max(1,t_*/r^3)t_*^{N-1-\delta}t^\delta & \left((\partial_{r}|_{t_*}\phi)^2+\delta(t_*/t)(T\phi)^2+\frac{\abs{\slashed{\nabla}\phi}^2}{r^2}+\frac{\phi^2}{/\jpns{r}^2}\right) \, \mu \\ \lesssim_\delta & \int_{\mathcal{D}_s}f^2 t^{\delta+1}t_*^{N-\delta}\min(1,r^3/t_*) \, \mu .
			\end{split}
		\end{equation}
		The right hand side is bounded for $\delta<2$ and $N+\delta+1<2q$.
		This already yields for $j=1$
		\begin{equation}\label{an:eq:proof1}
			\phi\in\Hb^{j;\delta/2-1,N-3+\delta,N-1-\delta}(\MMink).
		\end{equation}
		
		Commuting with $T$, we can prove one order faster decay rate for $T\phi$, since $Tf\in\A{b}^{3,q+3,q+3}(\MMink)$. 
		After commuting furthermore with scaling $(r\partial_r+t\partial_t)$, rotations $(x_i\partial_j-x_j\partial_i)$, we have \cref{an:eq:proof1} for all $j$.
		Using \cref{not:eq:sobolev} yields the result.
	\end{proof}
	
	\begin{remark}[Emitting the Morawetz estimate]\label{an:rem:Morawetz}
		In the above proof, we may first smoothen out $f$ so that $f\in\A{b}^{3,q+2,\infty}(\MMink)\subset \A{b}^{3,q+2}(\MMinkempty)$, and use only the $J^T$ vectorfield together with commutations with Lorentz boosts to obtain the result.
		We presented a proof using a Morawetz estimate as in the case of finite regularity scattering data $\psi_\scri$ it yields fewer losses of derivatives to construct $\phi_{app}$.
	\end{remark}
	
	\begin{cor}\label{an:cor:i_+_conormal}
		Let $f\in\A{b}^{3,q+2,q_\K}(U_3)$ with $U_c=\Mcomp\cap\{r>cM\}$, $q_\K\in(q-1,q)$ with $q>1$.
		Then, there exists $\phi\in\A{b}^{1,(q,q_\K)-}(U_3)$ such that
		\begin{equation}
			\Box_{\ern}\phi-f\in\A{b}^{3,(q+3,q_\K)-}(U_3),\quad r\phi|_{\scri}=0.
		\end{equation}
	\end{cor}
	\begin{proof}
		This follows from \cref{an:prop:i+_conormal} together with \cref{map:lemma:normal_ops} and an application of a cutoff function localised to the far region.
	\end{proof}
	
	\subsection{Solvability with conformal smoothness}\label{an:sec:i+-inverse-smooth}
	In this section and the next, we concentrate on the normal operator $\Normal{\F}$.
	Via \cref{map:rem:isometry}, understanding solutions that are smooth across $\hor$ is equivalent to studying conformally smooth solutions near $\scri$ for $\Normal{\ip}$.
	However, as discussed in \cref{intro:rem:conf_vs_con}, conformal smoothness is not a physical requirement, while smoothness across $\hor$ certainly is.
	We will study $\Normal{\F}$ directly in the next section, but we include the conformally smooth solutions to $\Normal{\ip}$ for the reader's convenience and for comparison.
	
	In the next result, we enforce not only conormal smoothness near $\scri$ (regularity with respect to $r\partial_r|_u$), but also conformal smoothness (regularity with respect to $u^{-1}r^2\partial_r|_{u}$).
	\emph{For this section only}, let us use the notation
	\begin{equation}
		\Aext{b}^{3,q_+,q_\K}(\MMink):=\{f\in \A{b}^{3,q_+,q_\K}(\MMink): (\chi_{r>1}(r)u^{-1}r^2\partial_r|_{u})^k\in \A{b}^{3,q_+,q_\K}(\MMink)\}
	\end{equation}
	
	\begin{prop}\label{an:prop:i+_smooth_conormal}
		Fix an inhomogeneity $f\in\Aext{b}^{3,q+2,q_\K}(\MMink)$ where $q\in\R$ and $q_\K\in(q-1,q)$.
		Then, there exists $\psi^\scri\in \A{b}^{q-1-}(\scri)$ such that the scattering solution 
		\begin{equation}
			\Box_\eta\phi=f,\qquad r\phi|_\scri=\psi^\scri
		\end{equation}
		satisfies $\phi\in\Aext{b}^{1,q-,q_\K-}(\MMink)$.
	\end{prop}
	
	\begin{remark}[Non-uniqueness]
		We remark, that $\psi^\scri$ in the above proposition is \emph{not} unique.
		Indeed, given any homogeneous solution $\Box_\eta ( t^{-q'} \cdot \varphi_{q'}(x/t) ) =0$ for $q\geq q'$ such that $\varphi_q$ is smooth  near $\partial B_{x/t}$ we obtain that $\phi+t^{-q'} \cdot \varphi_{q'}(x/t)$ is also an admissible solution.
		For instance, within the spherically symmetric class, we can take $t^{-q-1}r\phi_{q+1}=v^{-q}-u^{-q}$.
	\end{remark}
	
	\begin{proof}
		\emph{Smoothing:}
		As in the proof of \cref{an:prop:i+_conormal}, it suffices to consider the case $f\in\Aext{b}^{3,q+2,q+4}(\MMink)$.
		
		\emph{Commutation with conformal scaling:}
		For the rest of the proof, assume that $q>0$, for otherwise the proof applies with $\Phi$ below replaced with $\phi$.
		We use the conformal Killing vectorfield $\mathcal{T}=u^2\partial_u+v^2\partial_v\in(\rho_\K\rho_+)^{-1}\Diffb(\MMink)$, satisfying the commutation
		\begin{equation}\label{an:eq:conformal_commutation}
			[\mathcal{T},u^2v^2r\Box_\eta r^{-1}]=0\implies \Box_\eta \left( r^{-1}\mathcal{T}^k (r\phi ) \right) =f_k:=\frac{\mathcal{T}^k ( u^2v^2r f )}{ u^2v^2r}\in\A{b}^{3,q+2-k,q+4-k}(\MMink).
		\end{equation}
		Let's set $F=f_{1+\floor{q}}$ and $q'=q+1-\floor{q}\in[1,2]$.
		We will construct $\phi$ as the $\mathcal{T}^{-1}$ integral of $\Phi$, where $\Phi$ satisfies 
		\begin{equation}
			\Box \Phi=F,\qquad \Phi|_{t=0}=T\Phi|_{t=0}=0.
		\end{equation}
		
		Let's fix the currents $J^T_{N,\delta},J^X_N$ as in \cref{curr:lem:morawetz} with $N<0,\delta<0$.
		Applying the divergence theorem to $J^T_{N,\delta},J^X_N$ and using \cref{curr:lemma:poincare} yields
		\begin{equation}\label{an:eq:proof_smooth1}
			\begin{split}
				\int_{\mathcal{D}_s} \max(1,t_*/r^3)t_*^{N-1-\delta}t^\delta & \left((\partial_{r}|_{t_*}\Phi)^2+\delta(t_*/t)(T\Phi)^2+\frac{\abs{\slashed{\nabla}\Phi}^2}{r^2}+\frac{\Phi^2}{/\jpns{r}^2} \right) \, \mu \\ \lesssim_\delta & \int_{\mathcal{D}_s}F^2 t^{\delta+1}t_*^N\min(1,r^3/t_*) \, \mu .
			\end{split}
		\end{equation}
		The right hand side is bounded for $\delta<2$ and $N+\delta<2q'-5/2$.
		Therefore, we obtain $\Phi\in \mathcal{X}^0$ where $\mathcal{X}^s=\Hb^{s;\frac{\delta-1}{2},\frac{N-3}{2},\frac{N}{2}}(\MMink)$.

		\emph{Higher regularity:}
		We can commute the equation with $\partial_x,\partial_t$ to obtain $\{\partial_t,\partial_x\}^k\Phi\in \mathcal{X}^0$ for any $k$.
		
		We compute
		\begin{equation}\label{an:eq:rdr_commutation}
			[r\partial_r|_{t*},r^2r \Box_\eta r^{-1}]=-2r\partial_{t_*}r\partial_r+\chi \Diff^2\implies \Box_\eta r\partial_r\Phi-\frac{1}{r}T(r\partial_r)\Phi=\chi\Diff^2\Phi+\Diffb^1 F
		\end{equation}
		where $\chi$ is compactly supported.
		The additional term on left hand side yields extra damping for $\Phi$.
		In particular, we can apply $J^T_{N,\delta},J^X_N$ to conclude that $r\partial_r\Phi\in\mathcal{X}^0$.
		Using further commutations, we get $(r\partial_r)^k\Phi\in\mathcal{X}^0$.
		Finally, using commutations with the scaling and rotation vectorfields, we get $\Phi\in\mathcal{X}^\infty$.
		Via \cref{not:eq:sobolev} we obtain $\Phi\in\A{b}^{1/2,q'-2,q'-2}(\MMink)$.
		
		Let us note that even though in the above construction, we had $q'\geq0$ the estimates for $\Box\Phi=F$ also apply for $q'<0$.
		
		\emph{Conformal smoothness:}
		We use \cref{an:eq:conformal_commutation} to prove estimates for $r^{-1}\mathcal{T}^kr\Phi$.
		Since applying $\mathcal{T}$ vectorfields decreases the decay rate, the previous estimates apply and  it follows that $\mathcal{T}^k\Phi\in\A{b}^{1/2,q'-k-2,q'-k-2}(\MMink)$. 
		Hence, we obtain $\Phi\in\Aext{b}^{1,q'-2,q'-2}(\MMink)$.
		
		\emph{Integration:}
		Let us work first in a neighbhourhood of $\scri$ with coordinates $\rho_+,\rho_\scri$, and extend $f$ (correspondingly $F$) and $\Phi$ smoothly across $\{\rho_\scri=0\}$ with support in $\rho_\scri<-1/2.$
		Then, we may obtain $f$ by integrating $F$ via \cref{map:lemma:T-1} below, and \emph{define} $\phi$ as the corresponding $\mathcal{T}^{-k}\Phi$ integral of $\Phi$ where the integration starts at $\{\rho_\scri=-1\}$.
		Using the mapping property of $\mathcal{T}^{-1}$ in \cref{map:eq:T-1} yields the result.
	\end{proof}
	
	\begin{remark}[Avoiding $r\partial_r|_{t_*}$ commutations]\label{an:rem:rpr}
		In the previous proof, we may use the same improvement as in \cref{an:rem:Morawetz} to only study $f\in\Aext{b}^{3,q+2}(\MMinkempty)$.
		In this case, we can commute with Lorentz boosts as well as scaling to yield infinite $\Hb$ regularity, and thus there is no need to commute with $r\partial_r$.
	\end{remark}
	
	\begin{cor}\label{an:cor:i_+_smoth_conormal}
		Let $U_c=\{r<cM\}\cap\Mcomp$ and fix $f\in\Aext{b}^{q_\K,q-1}(U_2)$ with $\supp f\subset U_2$, $q_\K\in(q-1,q)$ and $q\in\R$.
		Then, there exists $\phi\in\Aext{b}^{q_\K,q-1}(U_2)$ and $\supp\phi\subset U_2$ such that
		\begin{equation}
			\Box\phi-f\in\A{b}^{q_\K,q}(U_2).
		\end{equation}
	\end{cor}
	\begin{proof}
		This follows from \cref{an:prop:i+_smooth_conormal} after Couch-Torres isometry \cref{not:eq:conformal_wave}.
	\end{proof}
	
	To finish the proof of \cref{an:prop:i+_conormal}, we recall some standard ODE results for functions in $\A{b}(\Mcomp_\mathfrak{t})$. Similar results are also true on $(\MMink,\eta)$.
	
	We first give the form of the vector field $\mathcal{T}:=u^2\partial_u+v^2\partial_v=u^2\partial_u+2(u+r)r\partial_r=2(t^2+r^2)\partial_t+4tr\partial_r\in(\rho_\scri\rho_+\rho_\K\rho_F)^{-1}\Diffb(\Mcomp_{\mathfrak{t}})$, which is the timelike conformal Killing field of Minkowski in the far away region: 
	\begin{equation}\label{map:eq:conformal}
		\mathcal{T}=\begin{cases}
			\rho_+^{-1}(-\rho_+\partial_{\rho_+}-(1-\rho_\scri)\partial_{\rho_\scri})& \text{near }\scri\cap \ip\text{with }\rho_+=u^{-1},\rho_\scri=u/v\\
			\rho_\K^{-1}\rho_+^{-1}\big((1-\rho_\K^2)\rho_\K\partial_{\rho_\K}-2\rho_{+}\partial_{\rho_+}\big)&\text{ near }\ip\cap\K\text{ with} \rho_+=1/r,\rho_\K=r/t\\
			\rho_\K^{-1}\rho_F^{-1}\big((1-\rho_\K^2)\rho_\K\partial_{\rho_\K}-2\rho_{F}\partial_{\rho_F}\big)&\text{ near }\K\cap\F\text{ with} \rho_+=1/(r-M),\rho_\K=(r-M)/t\\
			\rho_\F^{-1}(-\rho_\F\partial_{\rho_\F}-(1-\rho_\hor)\partial_{\rho_\hor})& \text{near }\F\cap \hor\text{with }\rho_\F=v^{-1},\rho_\hor=v/u
		\end{cases}
	\end{equation}
	The future directed integral curve of $\mathcal{T}$ starting from $(u,v)$ is given by $\Gamma_{u,v}=\{(u',v'):1/u'-1/v'=1/u-1/v,v'\geq v\}$.
	Let us define the following time integration operator
	\begin{nalign}
		\mathcal{T}_{\downarrow}^{-1}f(u,v):=\int_{\Gamma_{u,v}} f(u',v')\, \dd v' ,
	\end{nalign}
	which is well defined for $f$ compactly supported in $\M_{\mathfrak{t}}$.
	
	The following is a standard result
	\begin{lemma}\label{map:lemma:T-1}
		Let $(u^{-1}r^2\partial_r|_u)^kf\in\A{b}^{0,q_+,q_\K,q_\F}(\Mcomp_{\mathfrak{t}})$ for all $k$.
		Then $\mathcal{T}_{\downarrow}^{-1}f$ is well defined and it satisfies 
		\begin{subequations}\label{map:eq:T-1}
			\begin{align}
				\chi_{r_*\notin(-1,1)}(u^{-1}r^2\partial_r|_u)^k\mathcal{T}_\downarrow^{-1} f&\in \A{b}^{0,q_++1,\min(q_\K+1,q_++2,q_\F+2),q_\F+1}(\Mcomp_{\mathfrak{t}}).\label{map:eq:Tbackward}
			\end{align}
		\end{subequations}
	\end{lemma}
	
	\begin{proof}
		This follows from the propagation results in \cite[Section 4]{kadar_scattering_2025}.
		In particular, it is sufficient to localise the transport equation near the individual boundary faces and near the corners, where the above reference applies.
		
		Note that the form of $\mathcal{T}$ given in \cref{map:eq:conformal} is only valid away from the special integral curve $u=v$, which is inside the region $r_*\in(-1,1)$.\footnote{The estimate \cref{map:eq:Tbackward} still holds in this region, but does not follow from the reference, and since we not need in the rest of the paper, we choose to localise.}
	\end{proof}
	
	\subsection{Solvability across the horizon}\label{an:sec:across_horizon}
	For future applications, it is helpful, to extend the scattering construction beyond the horizon of the black hole.
	This is the first section where we work in $\Mcomp_{\mathfrak{t}}$ and not $\Mcomp$.
	
	\begin{prop}\label{an:prop:F_interior}
		Let $U_c=\Mcomp_{\mathfrak{t}}\cap\{r<cM\}$.
		Fix  $f\in\A{b}^{q_\K,q-1}(U_{2})$ with $\supp f\subset U_{2}$ and $q\in\R$, $q_\K\in(q-1,q)$.
		Then, there exists $\phi\in\A{b}^{(q_\K,q-1)-}(U_{2})$ with $\supp\phi\subset U_{2}$ such that
		\begin{equation}\label{an:eq:through_horizon}
			\Box_{\ern}\phi-f\in\A{b}^{(q_\K,q)-}(U_{2}).
		\end{equation}
	\end{prop}
	
	\begin{proof}
		The proof is essentially the same as for \cref{an:prop:i+_smooth_conormal}, but we extend across the horizon.
		We prove \cref{an:eq:through_horizon} with $\Box_{g_{\mathrm{near}}}$, and substituting $\ern$ will yield the same error term using \cref{map:lemma:normal_ops}.
		
		\emph{Smoothing:} As in \cref{an:prop:i+_conormal}, it suffices to consider $f\in\A{b}^{q,q+2}(U_2)$.
		
		\emph{Conformal commuted:}
		Let $\zeta=r-M$.
		We use the vectorfield $\mathcal{T}_{\mathrm{near}}=v^2\partial_v-2(\zeta v+1)\partial_x$.
		Using \cref{an:eq:conformal_commutation}, we also get
		\begin{equation}\label{an:eq:conformal_near_commutation}
			[\mathcal{T}_{\mathrm{near}},v^2(v+2\zeta^{-1})^2\Box_{g_{\mathrm{near}}}]=0\implies \Box_{g_{\mathrm{near}}}\mathcal{T}^k\phi=f_k=\frac{\mathcal{T}^kv^2(v+2\zeta^{-1})^2f}{v^2(v+2\zeta^{-1})^2}.
		\end{equation}
		We set $q'=q-\ceil{q}\in(-5,-4]$ and $F:=f_{4+\ceil{q}}\in\A{b}^{q',q'+3}(U_2)$. 
		Let $\mathcal{B}=\{r=cM\}$ and consider the initial boundary value problem 
		\begin{equation}
			\Box_{g_\mathrm{near}} \Phi=F,\quad \Phi|_{t_*=1}=T\Phi|_{t_*=1}=0,\quad \Phi|_{\mathcal{B}}=0.
		\end{equation}
		We prove that $\Phi\in\A{b}^{q',q'-1}(U_2)$.
		
		\emph{Exterior:}
		We first work in $\{r>M\}$, and write $U_2^{\mathrm{e}}=U_2\cap\{r\geq M\}$.
		Fix $N<0$ and $\delta<0$.
		Using a $T$ energy estimate, we get from \cref{curr:lem:ext} that $J^T_{N,\delta}=t^\delta t_*^{N-\delta}T\cdot\T[\Phi]$ satisfies
		\begin{equation}
			\Div J^T_{N,\delta}+t^\delta t_*^{N-\delta}\abs{F(T\phi)} \sim t_*^{N-1-\delta}t^\delta \Big(D(\partial_r\Phi)^2+D'\Phi^2+\abs{\slashed{\nabla}\Phi}^2+(T\Phi)^2\big(\delta t_*/(Dt)+1\big)\Big)
		\end{equation}
		and $J^T_{N,\delta}\cdot \dd r|_\mathcal{B}=0$  vanishes by the Dirichlet condition $\Phi|_{\mathcal{B}}=0$.
		We also use the current $J_N^2[\phi]=-v^{N-2}\partial_r\cdot \T[\phi]$, which satisfies $J^2_N[\phi]\cdot \dd r|_{\mathcal{B}}>0$ and
		\begin{equation}
			\Div J^2_N[\phi]+v^{N-2}\abs{F\partial_r\Phi}\sim v^{N-3}(2-N+D')(\partial_r \Phi)^2.
		\end{equation}
		An application of the divergence theorem yields
		\begin{equation}
			\int_{\mathcal{D}_s}t^\delta t_*^{N-1-\delta}\Big( \rho_\F^2(\partial_r\Phi)^2+(T\Phi)^2t_*/(Dt)+\abs{\slashed{\nabla}\Phi}^2+D'\Phi^2
			\Big)\lesssim_{\delta}\int_{\mathcal{D}_s} t^\delta t_*^{N-\delta}F^2(Dt+1).
		\end{equation}
		The right hand side is bounded for $N<2q'-1$ and $\delta<0$, implying that 
		\begin{equation}\label{an:eq:across_proof1}
			(\rho_\F\partial_r,\rho_\F^{-1}\rho_\hor^{-1/2} T,\slashed{\nabla},1)\Phi\in t^{0+}\Hb^{0;q'-1,q'-1,0}(U_2^{\mathrm{e}}).
		\end{equation}
		
		\emph{Commutation:}
		We still restrict to  $\{r>M\}$.
		Similar to \cref{an:eq:rdr_commutation}, we have
		\begin{equation}
			[(r-M)\partial_r,\Box_{g_\mathrm{near}}]=-\frac{1}{r-M}\partial_v (r-M)\partial_r\implies (\Box_{g_\mathrm{near}}-\frac{1}{r-M}T)(r-M)\partial_r\Phi=(r-M)\partial_r F.
		\end{equation}
		The extra damping term already implies that $(r-M)\partial_r\Phi$ is in the same function space as $\Phi$ from \cref{an:eq:across_proof1}.
		Noting that scaling ($vT-(r-M)\partial_r$) and rotations both commute with $\Box_{g_\mathrm{near}}$, by induction, we get for all $k\in\N$
		\begin{equation}\label{an:eq:across_proof2}
			\{(r-M)\partial_r,vT,x_i\partial_j-x_j\partial_i,1\}^k\Phi\in t^{0+}\Hb^{0;q'-1,q'-1,0}(U_2^{\mathrm{e}}).
		\end{equation}
		Finally, we can use \cref{an:eq:conformal_near_commutation} to further obtain
		$\mathcal{T}^k\Phi\in t^{0+}\Hb^{0;q'-1-k,q'-1-k,0}(U_2^{\mathrm{e}})$.
		Using \cref{not:eq:sobolev}, we get $\Phi\in t^{0+}\A{b}^{\infty;q'-1/2,q'-1,0}(U_2^\mathrm{e})$.
		
		\emph{Flux:}
		A posteriori, we can apply a $T$ energy estimate without $t^\delta$ weight to conclude that
		\begin{equation}
			\int_{\hor}J^T_N[\{v^{-1}\partial_r,vT,x_i\partial_j-x_j\partial_i,1\}^k\Phi]\lesssim_{\alpha,F} 1,
		\end{equation}
		for $N<2q'-1$, and in particular $\Phi|_{\hor}\in\A{b}^{q'-1}(\hor)$.
		Similarly $\mathcal{T}^k\Phi|_{\hor}\in\A{b}^{q'-k-1}(\hor)$.
		
		\emph{Interior:}
		Finally, we can use a $T$ energy estimate to propagate the same bound in the interior.
		Write $\Phi(v,r,x/r)=\Phi|_{\hor}(v,x/r)\chi(v(M-r))+\bar{\Phi}$ for a cutoff function localising near $\hor$ with $\chi'$ disjoint from the artificial boundary of $\overline{\M}_\mathfrak{t}$ and $\hor$.
		It suffices to prove that $\bar{\Phi}\in\Aext{b}^{q'-1}(\McompIn)$.
		
		Let $J^T_{N,\delta}=\abs{t}^\delta t_*^{N-\delta}T\cdot\T[\Phi]$ for $\delta>0$, and note that $\dd \abs{t}=-\dd t$ is future directed causal, just like $\dd t_*^{-1}$.
		A straightforward computation yields
		\begin{equation}
			\Div J^T_{N,\delta}+t^\delta t_*^{N-1-\delta}F^2\gtrsim t_*^{N-1-\delta}t^\delta \Big(D(\partial_r\Phi)^2+\abs{\slashed{\nabla}\Phi}^2+(T\Phi)^2\big(\delta t_*/(Dt)+1\big)\Big).
		\end{equation}
		We also have the deformation tensor $\pi^{-\frac{1}{v^{N+2}}\partial_r}=v^{-2}(D'v-N+2)\dd v^2$.
		In the interior, we can bound $\abs{D'v}\leq\mathfrak{t}\geq 2-N$, and so we obtain for $J^2_N=-\frac{1}{v^{N+2}}\partial_r\cdot\T[\Phi]$ with $N<\min(-1,2q'-1)$
		\begin{equation}
			\Div J^2_N\sim v^{-(N+3)}(\partial_r\Phi)^2.
		\end{equation}
		In particular, we obtain with a Hardy inequality (\cref{curr:lemma:hardy}) that $\Phi\in\Hb^{0;q'-1/2,\delta}(\McompIn)$.
		
		The same commutations as in the exterior region apply and we conclude that $\Phi\in\A{b}^{q',q'-1}(U_2)$.
		
		\emph{Integration:}
		Extending $\Phi$ (correspondingly $F$) to the region $t_*(M-r)\in(\mathfrak{t},2\mathfrak{t})$ smoothly and supported in $t_*(M-r)<3\mathfrak{t}/2$, we can obtain $\phi$ by integrating along the integral curves of $\mathcal{T}$ via \cref{map:lemma:T-1}.
	\end{proof}
	
	\begin{remark}[Log redshift]\label{an:rem:log_redshift}
		We also observe that for $\tilde{Y}=v^{-1}\partial_r$ the following commutation holds\footnote{it is possible to exchange this part by commuting with $\mathcal{T}$}
		\begin{equation}\label{an:eq:across_proof3}
			[\tilde{Y},\Box_{g_\mathrm{near}}]=2\big(\partial_r (r-M)+v^{-1}\partial_r\big)\tilde{Y}\implies \big(\Box_{g_\mathrm{near}}+2\tilde{Y}\big)\tilde{Y}\Phi=\tilde{Y}F-\partial_r(r-M)\tilde{Y}\Phi.
		\end{equation}
		Let us comment on the structure of \cref{an:eq:across_proof3}.
		The appearance of the extra $2k\tilde{Y}$ linear term for $\tilde{Y}^k\Phi$ is due to an \emph{enhanced log-redshift} that the commuted quantities satisfy.
		Indeed, note that the classical redshift on sub-extremal black holes is due to the leading $\partial_{t_*}\partial_r+\kappa\partial_r$ part of $\Box_g$.
		This leading part shows that high frequency part of the solution decays at $e^{-t_*\kappa}$ rate near the horizon.
		In our case, $\tilde{Y}^k\Phi$ satisfies an equation with leading part $\partial_{t_*}\partial_r+\frac{2k}{t_*}\partial_r$.
		Using $\log t_*$ as a coordinates, this is of the same form as in the subextremal case.
	\end{remark}
	
	\subsection{Nonlinear mapping}\label{sec:an:nonlinear}
	In this section, we study the nonlinear part of \cref{an:thm:ansatz}, and prove that under the assumed decay rate it is completely perturbative.
	
	\begin{lemma}\label{an:lemma:nonlin}
		For $\phi\in\Aext{b}^{a_\scri,a_+,a_\K,a_\F}(\Mcomp)$,
		\begin{nalign}\label{an:eq:nonlin_mapping}
			\partial_r|_{t_*}\phi&\in \Aext{b}^{a_\scri+1,a_++1,a_\K,a_\F-1}(\Mcomp),
			&T\phi&\in \Aext{b}^{a_\scri,a_++1,a_\K+1,a_\F+1}(\Mcomp),\\
			\phi^3&\in \Aext{b}^{3a_\scri,3a_+,3a_\K,3a_\F}(\Mcomp), &\ern(\dd\phi,\dd\phi)&\in \Aext{b}^{2a_\scri+1,2a_++2,2a_\F+1,2a_\F}(\Mcomp)
		\end{nalign}
	\end{lemma}
	\begin{proof}
		This is purely computational.
	\end{proof}
	
	\begin{lemma}[Linearised operator of $\Box_{\ern}+\mathcal{N}$]\label{an:lemma:nonlin_pert}
		Let $q>3/2$ and $q_\F\geq1$. 
		Let $\bar{\phi}\in\Aext{b}^{1,q,q,q_\F}(\multiComp)$.
		For $q'>q$, $q_\F'>q_\F$ and $\tilde{\phi}\in\Aext{b}^{1,q',q',q_\F'}(\multiComp)$, we have
		\begin{equation}\label{an:eq:nonlinear_gain}
			\Box_{\ern}(\bar{\phi}+\tilde{\phi})-\mathcal{N}[\bar\phi+\tilde\phi]=\Box_{\ern}\bar{\phi}-\mathcal{N}[\bar\phi]+\Box_{\ern}\tilde{\phi}+\Aext{b}^{3,q'+3,q'+1,q'+1}(\multiComp).
		\end{equation}
	\end{lemma}
	\begin{proof}
		This is a direct computation via \cref{an:eq:nonlin_mapping,map:lemma:normal_ops}.
	\end{proof}
	
	\begin{remark}\label{an:rem:optimality}
		We note that in \cref{an:eq:nonlinear_gain} 1 extra order of decay is gained compared to \cref{map:eq:Box}.
		We could take $q,q'$ smaller in \cref{an:lemma:nonlin_pert}, provided that we relax the gain to some $\epsilon>0$.
		Similarly, we can allow for a larger range of non-linearities as long as $\epsilon>0$.
		In this case, the iteration scheme in the proof of \cref{an:thm:ansatz} would need to take this into account and make the proof less clear.
		As the point is not to optimise  the decay rates $q$, we pursued no such improvements.
		See \cite[Definition 5.4]{kadar_scattering_2025}, for a similar classification of allowed decay and nonlinearities.
	\end{remark}
	
	\subsection{Iteration scheme}\label{an:sec:iteration}
	We are ready to put together the invertibility statements above:
	\begin{proof}[Proof of \cref{an:thm:ansatz} and \cref{an:cor:Aretakis}]
		We iteratively solve error terms at the faces $\F,\K,i^+$.
		
		\emph{Step 0)}
		To this end, let us already write $r\phi_0(u,r,\omega)=\psi_\scri(u,\omega)\chi(u/r)$ where $\chi$ is a cutoff function localising to $u/r<1$.
		Then, we already have that $\Box_{\ern}\phi_0\in\A{b}^{3,q+2,\infty,\infty,\infty}(\Mcomp)$ and $r\phi_0|_\scri=\psi_\scri$. 
		For the rest of the proof, we will take 0 data along $\scri$.
		
		\emph{Step 1)}
		Assume that we already found $\phi\in\A{b}^{1,q,q,q}(\Mcomp_\mathfrak{t})$ such that 
		\begin{equation}\label{eq:an:proof1}
			f=\Box_{\ern}\phi-\mathcal{N}[\phi]\in\Aext{b}^{3,p_++2,p_\K,p_\F-1}(\Mcomp),\qquad r\phi|_{\scri}=\psi_\scri,
		\end{equation}
		with $p_\K\geq p_{+}$, $p_\K\geq p_\F$, and $p_\F,p_+\geq q$.
		Let $\chi_{\ip}, \chi_\F$ localise to $\{r>10M\}$ and $\{r<2M\}$ respectively.
		Let $f_{\bullet}=\chi_\bullet f$ for $\bullet\in\{\ip,F\}$.
		From \cref{an:cor:i_+_conormal,an:prop:F_interior}, we know that there exists $\phi'\in\Aext{b}^{1,p_+,\min(p_+,p_\F),p_\F}(\Mcomp)$ such that 
		\begin{equation}
			\Box_{\ern}\phi'-f_\F-f_{\ip} \in \Aext{b}^{3,p_++3,\min(p_+,p_\F),p_\F}(\Mcomp).
		\end{equation}
		Using \cref{an:lemma:nonlin,an:lemma:nonlin_pert} it already follows that 
		\begin{equation}
			\mathcal{N}[\phi+\phi']-\mathcal{N}[\phi]\in \Aext{b}^{3,p_++3,\min(p_+,p_\F),p_\F}(\Mcomp),\quad r\phi'|_\scri=0.
		\end{equation}
		
		\emph{Step 2)}
		Assume that we already found $\phi$ such that  \cref{eq:an:proof1} holds with $p_\K\leq p_+$ and $p_\K\leq p_\F$.
		Using \cref{an:cor:K} it already follows that there exists $\phi'\in\Aext{b}^{\infty,p_\K+1,p_\K,p_\K}(\Mcomp)$ such that
		\begin{equation}
			\Box_{\ern}\phi'-f\in\Aext{b}^{3,\min(p_\K+3,p_++2),p_\K+1,(p_\K,p_\F-1)}(\Mcomp).
		\end{equation}
		Using \cref{an:lemma:nonlin,an:lemma:nonlin_pert} it already follows that 
		\begin{equation}\label{an:eq:iteration2}
			\mathcal{N}[\phi+\phi']-\mathcal{N}[\phi]\in\Aext{b}^{3,\min(p_\K+3,p_++2),p_\K+1,(p_\K,p_\F-1)}(\Mcomp).
		\end{equation}
		
		\emph{Step 3)}
		Iterating the previous steps clearly yield $\phi_N$ for any $N$ such that \cref{eq:an:proof1} holds with $p_+=p_\K=p_\F=N$.
		Here, we crucially use the fact that as indicated in \cref{an:eq:iteration2}, there is a gap of size 1 between $p_\K+3,p_++2$ at the case of equality between the variables $p_\K,p_+$.
		Borell summation proves the result.

		\emph{Proof of \cref{an:cor:Aretakis}}
		We simply note that at the first time we invert the model operator at $F$, without loss of generality, we can add an arbitrary amount of $\phi_1$, since $\Box_{\ern}\phi_1\in\A{b}^{\infty,\infty,2,2}(\Mcomp_{\mathfrak{t}})$.
	\end{proof}

	\subsection{Sharpness of ansatz and matching conditions}\label{an:sec:sharp}
	In this section, we revisit the creation of the ansatz close to the horizon and describe in some sense how sharp \cref{an:thm:ansatz} is.
	Let us focus for now on the linear wave equation $\Box_{\ern}\phi=0$.
	
	\begin{proof}[Proof of \cref{intro:cor:non-unique}]
		Clearly, we can write $\phi_p=\bar\phi_p+\tilde{\phi}_p$, where, by assumption we have
		\begin{equation}
			\Box_{\ern}\tilde{\phi}_p\in\Aext{b}^{\infty,\infty,p,p+1}(\Mcomp).
		\end{equation}
		We may apply \cref{an:thm:ansatz} to conclude that that there exists $\tilde{\phi}_p\in\Aext{b}^{1,p,p,p+1}(\Mcomp)$ such that
		\begin{equation}
			\Box_{\ern}(\bar\phi_p+\tilde{\phi}_p)=\Aext{b}^{3,\infty,\infty,\infty}(\Mcomp).
		\end{equation}
		The result follows from \cref{an:thm:ansatz}.
	\end{proof}
	
	Let us show that the set of $\bar\phi_p$ as required in \cref{intro:cor:non-unique} is not empty.
	For $p=1$, this is already given in the explicit formula \cref{an:eq:phi_p=1}.
	We focus on spherically symmetric solutions, but it is not difficult to show similar results outside symmetry.
	Using the ansatz $\bar\phi=v^{-p}\bar\phi(\mathfrak{r})$, where $\mathfrak{r}=(r-M)/v$, we get via \cref{map:eq:g_near_similarity_coord} that
	\begin{equation}\label{eq:an:non-unique}
		\Box_{g_{\mathrm{near}}}\bar\phi=v^{-p}\big(\partial_\mathfrak{r}(\mathfrak{r}^2+2\mathfrak{r}-2p)\big)\partial_{\mathfrak{r}}\bar\phi_p=0.
	\end{equation}
	The explicit solution that decays as $\mathfrak{r}\to\infty$ is given by $\partial_\mathfrak{r}\bar\phi=\mathfrak{r}^{p-1}(\mathfrak{r}+2)^{-(1+p)}$.
	This is only smooth near $\mathfrak{r}=0$, when $p\in\N$, and for this choice, it yields an admissible function in \cref{intro:cor:non-unique}.
	
	The argument in the proof of \cref{intro:cor:non-unique} also applies in the case that $\bar\phi_p$ is merely conormal near $\hor$, instead smooth.
	Therefore, taking the solution of \cref{eq:an:non-unique} for $p\notin\N$, we can construct solutions $\phi$ that settle down to a prescribed radiation field $\psi_\scri\in\A{b}^{q-1}(\scri)$, satisfy \cref{an:eq:ansatz}, but are only of conormal regularity $\phi\in\Aext{b}^{1,q,q,q}(\Mcomp)+\A{b}^{1,p,p,p,p}(\Mcomp)$ and $\phi\notin C^{p+1}(\M)$ near the horizon.

	\subsection{Subextremal black holes}\label{an:sec:subextremal}
	In this section, we show a version of \cref{an:thm:ansatz} to Schwarzchild black holes.
	This is significantly easier, as there is no need to make a compactification that includes the boundary $\F$.
	We emphasise once again, that such an approximate solution is not very useful in the subextremal case, as the corresponding errors cannot be removed by a simple energy estimate.
	
	We introduce separate notation for this section alone.
	\begin{definition}
		Let $\rho_\F=1$, and $\rho_\K,\rho_+,\rho_\scri$ as in \cref{not:eq:boundary_defining}.
		We define $\Mcomp_{\mathrm{s}}$ ($\mathrm{s}$ for subextremal) as a compactification of $\M_{\s}=\R_{t_*}\times\R^3_{\abs{x}>M/2}$, where $\rho_\K,\rho_+,\rho_\scri$ are extended to take value 0 and be boundary defining functions.
		The boundary $\{\abs{x}=M/2\}$ is an artificial boundary.
	\end{definition}
	Let us recall, that in line with \cref{not:def:b-ops,not:def:conormal}, we have on $\K$ the following space of functions
	\begin{equation}
		\A{b}^{q}(\K)=\{\abs{x}^{q}f\in L^\infty: \abs{x}^q\{\abs{x}\partial_x,1\}^\alpha f\in L^\infty,\, \forall \alpha\},
	\end{equation}
	where the vector fields do not degenerate at the inner boundary as that is an artificial one.
	
	We will consider $\M$ equipped with the Schwarzchild metric
	\begin{equation}
		g_{\mathrm{S}}=-(1-2M/r)\dd t_*^2+2\dd t_*\dd r+r^2\slashed{g},
	\end{equation}
	We define the zero energy operator $\widehat{\Box}_{g_\mathrm{S}}(0)=\lim_{\sigma\to0}e^{-t_*\sigma}\Box_{g_{\mathrm{S}}}e^{it_*\sigma}$.
	Let us recall from \cite[Theorem 6.1]{hafner_linear_2021} the relevant invertibility statement:
	\begin{lemma}\label{an:lemma:sub_zero}
		The operator
		\begin{equation}
			\widehat{\Box}_{g_\mathrm{S}}(0):\A{b}^l(\K)\mapsto\A{b}^{l+2}(\K)
		\end{equation}
		is invertible for $l\in(0,1)$.
	\end{lemma}
	
	Using \cref{an:lemma:sub_zero}, we can repeat the proof of \cref{an:thm:ansatz} to obtain:
	\begin{lemma}\label{an:lemma:subextremal}
		Let $\psi_\scri,q$ be as in \cref{an:thm:ansatz} and let $\mathcal{N}=\phi^3+g_{\mathrm{S}}(\dd\phi,\dd\phi)$.
		Then, there exists an ansatz $\phi\in\A{b}^{1,(q,q)-}(\Mcomp_{\mathrm{S}})$ satisfying
		\begin{equation}
			\Box_{g_{\mathrm{S}}}\phi-\mathcal{N}[\phi]\in\A{b}^{3,\infty,\infty}(\Mcomp_{\mathcal{S}}),\quad r\phi|_{\scri}=\psi_\scri.
		\end{equation}
	\end{lemma}
	
	\begin{remark}[Admissible metrics]
		The result \cref{an:lemma:sub_zero} is known in the full subextremal Kerr case due to \cite{shlapentokh-rothman_quantitative_2015}, and presented in the function spaces of the present paper in \cite[Theorem 4.7]{andersson_mode_2024}.
		Indeed, the zero energy invertibility is a much less delicate and should hold for a much larger class of metrics.
		We do not pursue this direction in the present paper.
	\end{remark}
	
	\section{Linear estimates}\label{sec:lin}
	In this section, we consider the linear wave equations $\Box_{\ern}\phi=f$ and apply the currents from \cref{sec:currents} to obtain estimates with optimal decay rates at $\F,\ip$ and a loss of decay at $\K$.
	The main result is the following:
	\begin{prop}\label{lin:prop:main}
		\emph{Exterior:} Fix $q>3/2$, $k+k'\leq q$, $s\in\N$.
		The solution $\phi$ of the scattering problem 
		\begin{equation}
			\Box_{\ern}\phi=f,\qquad \phi|_{\hor}=r\phi|_{\scri}=0
		\end{equation}
		where $f$ is supported away from $\partial\Mcomp$, satisfies the estimate
		\begin{equation}\label{lin:eq:main}
			\norm{\Ve\Vc^{k'} \phi}_{\Hb^{s;-1/4,q-3/2,q-1/2,q-1/2,k+3/4}(\Mcomp)}\lesssim_q\norm{\Vc^{k'}f}_{\Hb^{s;3/4,q+1/2,q+1/2,q-1/2,k-1/4}(\Mcomp)}
		\end{equation}
		where  the regularity across the horizon is  $\Vc=\{1,v^{-1}Y\}$ and the regularity gain is 
		\begin{equation}\label{lin:eq:Ve}
			\Ve=\{(r-M)\partial_r|_{t_*},\rho_+^{-1}\rho_\F^{-1}T,\rho_\F^{1/2}\rho_\scri^{1/2}\slashed{\nabla},1\}.
		\end{equation}
		
		\emph{Interior:} Fix $q,k,s$ as above.
		Fix $f\in C^\infty(\McompIn)$.
		For the solution to
		\begin{equation}
			\Box_{\ern}\phi=f, \qquad \phi|_{\mathcal{B}_{\mathfrak{t}}}=T\phi|_{\mathcal{B}_{\mathfrak{t}}}=0
		\end{equation}
		in $\McompIn$, we have the following estimate
		\begin{equation}\label{lin:eq:interior}
			\norm{\Ve\Vc^k\phi}_{\Hb^{s;,q-3/2,1/2-1/10}(\McompIn)}\lesssim_q\norm{\Vc^kf}_{\Hb^{s;,q-3/2,-1/2-1/10}(\McompIn)}+\norm{\Vb^{s+k+1}\phi|_{\mathcal{B}}}_{\Hb^{0;q-3/2}(\mathcal{B})}.
		\end{equation}
	\end{prop}
	
	\begin{remark}\label{lin:rem:Morawetz}
		Using a Morawetz multiplier together with an appropriate commutation as done in \cite{holzegel_note_2020}, we could improve the solvability to $\Hb^{\cdot,q+2,q,q-1,\cdot}\ni f\to\phi \in\Hb^{\cdot,q,q,q-1,\cdot}$.
		See also \cite{dyatlov_spectral_2016,dyatlov_asymptotics_2015}.
		Even without a commutation this would be possible, but not at top order.
		Not doing this extra work is motivated by the approach of \cite{dafermos_quasilinear_2022,angelopoulos_nonlinear_2026}, and puts extra constraint on the minimum allowed value of $q$.
	\end{remark}

	\begin{remark}[Regularity and decay]\label{lin:rem:blueshift}
		We note that the amount of regularity that we can propagate across the horizon ($\Vc$) depends on the decay rate of $\Box_{\ern}\phi$.
		This can be understood via the enhanced blueshift as explained in \cref{an:rem:log_redshift}, even though in the proof, we use the commutator $\mathcal{T}$ from \cref{lin:lemma:near_horizon}, instead $\tilde{Y}=\frac{1}{v}Y$.
	\end{remark}
	
	Before discussing the proof, let us comment on how to control the initial data norm in \cref{lin:eq:interior}.
	\begin{lemma}[Initial data]\label{lin:lemma:data}
		Let $k,q$ be as in \cref{lin:prop:main}.
		Consider $\Box_{\ern}\phi=f+\mathcal{N}[\phi]$ together with $\phi|_{\mathcal{B}_\mathfrak{t}}=\phi_0$, $t_*T\phi_{\mathcal{B}_\mathfrak{t}}=\phi_1$.
		Then, the following estimate holds for $f,\phi_0,\phi_1$ smooth and compactly supported in $t_*$
		\begin{equation}
			\norm{\Vb^s\Vc^k\phi|_{\mathcal{B}_\mathfrak{t}}}_{\Hb^{1;q-3/2}(\mathcal{B}_\mathfrak{t})}\lesssim_{s,k} \norm{\Vb^{s+k}f|_{\mathcal{B}_\mathfrak{t}}}_{\Hb^{1;q-3/2}(\mathcal{B}_\mathfrak{t})}+\norm{\phi_0,\phi_1}_{\Hb^{s+k+1;q-3/2}(\mathcal{B}_\mathfrak{t})}.
		\end{equation}
	\end{lemma}
	\begin{proof}
		Using that $\Box_{\ern}=-D^{-1}T^2+\partial_r r^2D\partial_r+r^2\slashed{\Delta}\in\Diffb^2$ in a neighbourhood of $\mathcal{B}_\mathfrak{t}$, the result follows after commuting the equation with $\Vb$, and using via \cref{lin:lemma:algebra,lin:lemma:nonlin} that the nonlinearities always decay sufficiently fast.
	\end{proof}
	
	For the proof, we will utilise the current computation from \cref{curr:lem:ext,curr:lem:int} together with the following commutations:
	
	\begin{lemma}[Near horizon commutators]\label{lin:lemma:near_horizon}
		Let $\zeta=r-M$.
		Let $\mathcal{T}=v^2\partial_v-2(1+v\zeta)\partial_\zeta\in\rho_\F^{-1}\Diffb$, and $S=(r-M)\partial_r-v\partial_v\in\Diffb$
		Then in $r<2M$ 
		\begin{equation}\label{lin:eq:commute_SK}
			[\Box_{\ern},S]=\rho_\F\Diffb^2,\quad v^{-2}(v+2\zeta^{-1})^{-2}[v^2(v+2\zeta^{-1})^2\Box_{\ern},\mathcal{T}]=\Diffb^2
		\end{equation}
	\end{lemma}
	\begin{proof}
		From \cref{map:lemma:normal_ops}, we write $\Box_{\ern}=\Box_{g_\mathrm{near}}+\rho_\F\Diffb^2$.
		Using \cref{an:eq:conformal_near_commutation}
		\begin{equation}
			[\mathcal{T},v^2(v+2\zeta^{-1})^2\Box_{g_\mathrm{near}}]=0.
		\end{equation}
		
		Similarly, for $S$, we observe that $[\Box_{g_{\mathrm{near}}},S]=0$, and so the result follows.
	\end{proof}

	\begin{proof}[Proof of \cref{lin:prop:main}]
		We only explain how to prove the exterior, the interior following verbatim, after replacing \cref{curr:lem:ext} with \cref{curr:lem:int}.
		
		\emph{Step 1)}
		Let $J^1_{<,N},J^T_{N,\delta},J^1_{>,N},C$ be as in \cref{curr:lem:ext}, and define the cutoff functions $\chi_<=\chi((r-M)t_*)$, $\chi_>=\chi(t_*/r)$ where $\chi(x)$ localises to $x<1/2$.
		Let $J_N=(1+2C)J^T_{N,1/4}+J^1_{>,N}+J^1_{>,N}$.
		Set $N=2q-5/2-2k$ and $\delta=1/2+2k$.
		We proceed by induction and write the proof for $k=0$ for notational convenience.
		
		For $s=1$, we first apply \cref{curr:lem:ext}.
		This already yields the result.
		
		\emph{Step 2)}
		Assume the result is true for some $s-1\geq1$.
		Commuting $\Box_{\ern}$ with $T^s$, and applying \cref{curr:lem:ext} for $J_{N+2s}[T^s\phi]$, we obtain 
		\begin{equation}
			\norm{\Ve t_*^sT^s\phi}_{\Hb^{0;-1/4,q-3/2,q-1/2,q-1/2,3/4}(\Mcomp}\lesssim_f 1.
		\end{equation}
		Via \cref{ell:lemma:near}, this already yields \cref{lin:eq:main} away from $\hor,\scri$.
		
		Next, we commute with  $\chi_>S,\chi_<S,\Omega$.
		Since $[\Box_{\ern},\Omega]=0$, control of angular derivatives follow.
		For the rest, we use \cref{lin:lemma:near_horizon}.
		Applying \cref{curr:lem:ext} for $J_{N}[(\chi_<S)^s\phi]$, we obtain error terms
		\begin{equation}
			t_*^{N-1}\Big(\abs{Sf}^2+\abs{\chi'\Diff^2_b\phi}^2+\abs{\Diffb^2\phi}^2\rho_\F^2\Big).
		\end{equation}
		The first term is already controlled by the right hand side of \cref{lin:eq:main}.
		The second term is supported away from $\mathcal{B}_1$ and $\hor$, so we can us \cref{ell:lemma:ext} to control the error term from $\Div J^T_{N,\delta}[T^s\phi]$, and the $s-1$ part of the right hand side of \cref{lin:eq:main}.
		Finally, the last term is controlled provided that $\rho_\F$ is sufficiently small.
		
		The same argument applies near $\scri$.
		In particular, we obtain \cref{lin:eq:main} for $k=0$.
		
		\emph{Step 3)}
		For $k\geq1$, we apply \cref{lin:lemma:near_horizon} and \cref{curr:lem:ext}.
		The result follows the same as for $S$, but we note that $\mathcal{T}$ has one order less decay than $S$, so we can only use the current $J_{N-2k+2s}[(\chi_<\mathcal{T})^kT^s\phi]$.
		The restriction on $k\leq\floor{q}$ comes from $N-2k>0$.
	\end{proof}
	
	\subsection{Algebra property}
	Finally, we make some short nonlinear computations with the spaces $\Hb(\Mcomp)$.
	These are $L^2$ analogues of the computations from \cref{sec:an:nonlinear}.
	
	\begin{lemma}[Algebra property]\label{lin:lemma:algebra}
		Fix $\vec{a}=\vec{a}^1+\vec{a}^2+\vec{a}_{ER}$, where $\vec{a}_{ER}=(3/2,2,1/2,0,-1/2)$ and $s\geq4$, $k\geq0$
		\begin{equation}
			\norm{\Vc^k\phi_1\phi_2}_{\Hb^{s;\vec{a}}(\Mcomp)}\lesssim\norm{\Vc^k\phi_1}_{\Hb^{s;\vec{a}^1 }(\Mcomp)}\norm{\Vc^k\phi_2}_{\Hb^{s;\vec{a}^2}(\Mcomp)}.
		\end{equation}
		In the interior we similarly have  for $s\geq4,k\geq1$ and  $a_{\F}=a^{1}_{\F}+a^{2}_{\F}$ $a_{\hor}=\min(a_\hor^1,a^2_\hor)$with $a^1_{\hor},a^2_{\hor}\in(-1/2,1/2)$
		\begin{equation}\label{lin:eq:interior_product}
			\norm{\Vc^k\phi_1\phi_2}_{\Hb^{s;\vec{a}}(\McompIn)}\lesssim\norm{\Vc^k\phi_1}_{\Hb^{s;\vec{a}^1 }(\McompIn)}\norm{\Vc^k\phi_2}_{\Hb^{s;\vec{a}^2}(\McompIn)}.
		\end{equation}
	\end{lemma}
	\begin{proof}
		\emph{Exterior:}
		For $k=0$, this is a consequence of \cref{not:eq:sobolev}, and a product estimate
		\begin{equation}
			\norm{\phi_1\phi_2}_{\Hb^{s;\vec{a}}}\lesssim \norm{\phi_1}_{\Hb^{s;\vec{a}^1}}\norm{w_2\Vb^{s-3}\phi_2}_{L^\infty}+\norm{\phi_2}_{\Hb^{s;\vec{a}^2}}\norm{w_1\Vb^{s-3}\phi_1}_{L^\infty}
			\lesssim \norm{\phi_1}_{\Hb^{s;\vec{a}^1 }(\Mcomp)}\norm{\phi_2}_{\Hb^{s;\vec{a}^2}(\Mcomp)},
		\end{equation}
		where $w_\bullet=\rho_{\scri}^{a_\scri^\bullet+3/2}\rho_{+}^{a_+^\bullet+2}\rho_{\K}^{a_\K^\bullet+1/2}\rho_{\F}^{a_\F^\bullet}\rho_{\hor}^{a_\hor^\bullet-1/2}$ for $\bullet\in\{1,2\}$.
		The extra $\Vc$ regularity is propagated by Leibniz rule.
		
		\emph{Interior:}
		We use the improvement provided by \cref{not:eq:sobolev2} to overcome the loss towards $\hor$ and improve $a_\hor^1+a_\hor^2$ to $\min(a^1_\hor,a^2_\hor)$.
	\end{proof}

	\begin{lemma}[Nonlinear estimate]\label{lin:lemma:nonlin}
		Fix decay rates $q_0>1$, $q>3/2$, and regularity indices $s\geq 4$, $k\geq0$.
		Let $\phi_0\in\Aext{b}^{1,q_0,q_0,1}(\Mcomp_{\mathfrak{t}})$, and let
		\begin{equation}
			\norm{\Vc^k\Ve \phi}_{\Hb^{s;-1/4,q-3/2,q-1/2,q-1/2,3/4}(\Mcomp)} \leq\epsilon\leq 1.
		\end{equation}
		Then
		\begin{equation}
			\norm{t_*^{1/2}\Vc^k(\mathcal{N}[\phi+\phi_0]-\mathcal{N}[\phi_0])}_{\Hb^{s;1,q+1/2,q+1/2,q-1/2,0}(\Mcomp)}\lesssim\epsilon.
		\end{equation}
		In the interior, we assume $s\geq4, k\geq1,q_0>1,q>3/2$ and
		\begin{equation}
			\norm{\Vc^k\Ve \phi}_{\Hb^{s;q-1/2,1/2-1/10}(\McompIn)} \leq\epsilon\leq 1.
		\end{equation}
		Then 
		\begin{equation}
			\norm{t_*^{1/2}\Vc^k(\mathcal{N}[\phi+\phi_0]-\mathcal{N}[\phi_0])}_{\Hb^{s;q-1/2,-1/2-1/10}(\McompIn)}\lesssim\epsilon.
		\end{equation}
	\end{lemma}
	\begin{proof}
		Let us note that we already have from \cref{not:eq:sobolev} that $\Ve\phi\in\A{b}^{0;5/4,q+1/2,q,q-1/2,1/4}(\Mcomp)$.
		The proof is purely computational from \cref{lin:lemma:algebra,not:eq:sobolev,not:eq:sobolev2}.
		
		Exterior
		For instance a linear term is bounded as
		\begin{multline}
			\norm{g_{\ern}(\dd\phi_0,\dd\phi)}_{\Hb^{s;1,q+1/2,q+1/2,q-1/2,0}(\Mcomp)}\lesssim\norm{\{\rho_\F\partial_r|_{t_*},\rho_\F^{-1}\rho_\scri T,\rho_+\rho_\scri\Omega_{ij}\}\phi}_{\Hb^{s;0,q-q_0,q+1-q_0,q-q_0,0}(\Mcomp)}\\
			\lesssim\norm{\Ve\phi}_{\Hb^{s;-1/2,q-1-q_0,q+1-q_0,q-q_0,0}(\Mcomp)}\lesssim\epsilon.
		\end{multline}
		While we can bound the nonlinear terms for instance as
		\begin{nalign}
			&\norm{D(\partial_r|_{t_*}\phi^2)}_{\Hb^{s;1,q+1/2,q+1/2,q-1/2,0}(\Mcomp)}\lesssim\norm{\Ve \phi}_{\Hb^{s;-1/4,q-3/2,q-1/2,q-1/2,3/4}(\Mcomp)}
			\norm{\Vb \phi}_{\Hb^{s-3;-9/4,-2,1/2,0,-1/4}(\Mcomp)};\\
			&\begin{multlined}
				\norm{(h'D+1) \abs{\partial_{t_*}\phi\partial_r\phi}}_{\Hb^{s;1,q+1/2,q+1/2,q-1/2,0}(\Mcomp)}\lesssim\norm{(r-M)\partial_r\phi}_{\Hb^{s;-1/4,q-3/2,q-1/2,q-1/2,3/4}(\Mcomp)}\\
				\cdot \norm{(r-M)^{-1}\partial_{t_*}\phi}_{\Hb^{s;-1/4,0,1/2,0,-1/4}(\Mcomp)};
			\end{multlined}\\
			&\norm{\rho_+\rho_\scri\abs{\partial_{t_*}\phi}^2}_{\Hb^{s;1,q+1/2,q+1/2,q-1/2,0}(\Mcomp)}\lesssim\norm{\partial_{t_*}\phi}_{\Hb^{s;-1/4,q-3/2,q-1/2,q-1/2,3/4}(\Mcomp)}\norm{\partial_{t_*}\phi}_{\Hb^{s-3;-5/4,-3,1/2,0,-1/4}(\Mcomp)};\\
			&\begin{multlined}
				\implies\norm{g_{\ern}(\dd\phi,\dd\phi)}_{\Hb^{s;1,q+1/2,q+1/2,q-1/2,0}(\Mcomp)}\lesssim\norm{D(\partial_r\phi)^2+(h'D+1) \abs{\partial_{t_*}\phi\partial_r\phi}}_{\Hb^{s;1,q+1/2,q+1/2,q-1/2,0}(\Mcomp)}	\\
				+\norm{(Dh'+2)\abs{\partial_{t_*}\phi}^2+r^{-2}\abs{\Omega\phi}^2}_{\Hb^{s;1,q+1/2,q+1/2,q-1/2,0}(\Mcomp)}.
			\end{multlined}
		\end{nalign}
		The rest of the nonlinearities are bounded similarly.
		
		For the interior, we use the improvement \cref{lin:eq:interior_product}. 
		Let us only discuss the nonlinearity that has $\partial_r$ derivatives, for this needs the improvement:
		\begin{multline}
			\norm{\Vc^k\partial_{t_*}\phi\partial_r\phi}_{\Hb^{s;q-1/2,-1/2-1/10}}\sim\norm{\Vc^k\partial_{t_*}\phi(r-M)\partial_r\phi}_{\Hb^{s;q+1/2,1/2-1/10}}\\
			\lesssim\norm{\Vc^k(r-M)\partial_r\phi}_{\Hb^{s;q-1/2,1/2-1/10}}\norm{\Vc^k\partial_{t_*}\phi}_{\Hb^{s;0,1/2-1/10}}\\
			\lesssim\norm{\Ve \phi}_{\Hb^{s;q-1/2,1/2-1/10}}\norm{t_*^{-1/2}\Ve\phi}_{\Hb^{s;q-1/2,1/2-1/10}} \lesssim
			\epsilon^2.
		\end{multline}
	\end{proof}
	
	\section{Scattering theory for a semilinear equation}\label{sec:scat}
	
	In this section, we correct the approximate solution given by \cref{an:thm:ansatz} to an actual solution of \cref{intro:eq:scattering}.
	For this, we prove a slight modification of a result of \cite{angelopoulos_non-degenerate_2020} by constructing scattering solutions in limited regularity spaces.
	
	\begin{theorem}\label{scat:thm:main}
		Fix $s\geq20,q_0>1,q>2$ and $\phi_{app}\in\A{b}^{s;1,q_0,q_0,q_0-1}(\Mcomp_{\m,\mathfrak{t}})$ with  $\Box_{\ern}\phi_{app}-\mathcal{N}[\phi_{app}]\in\A{b}^{s;3,q+5/2,q+1,q-1/2}(\Mcomp_{\m,\mathfrak{t}})$.
		Then, there exists a unique $\phi\in\A{b}^{k;1,q+1/2,q,q-1/2}(\Mcomp_{\m,\mathfrak{t}})$, solution to 
		\begin{equation}\label{scat:eq:scat}
			\Box_{\ern}(\phi+\phi_{app})-\mathcal{N}[\phi_{app}+\phi]=0,\quad \phi|_{\mathcal{B}_\mathfrak{t}}=T\phi|_{\mathcal{B}_\mathfrak{t}}=0,\quad r\phi|_{\scri}=0.
		\end{equation}
		where $k\leq \min(\floor{q-1},(s-15)/2)$.
	\end{theorem}
	The proof is split into two parts.
	First --in \cref{scat:lemma:interior}-- we solve the problem in the interior ($\McompIn$) and obtain estimates for the solution over $\hor$.
	Then --in \cref{scat:prop:main}-- we solve in the exterior ($\Mcomp$) and thus finish the result.
	
	\subsection{Interior solution}
	
	The main result of this section is the existence of nonlinear solutions in the interior.
	
	\begin{prop}\label{scat:lemma:interior}
		Let $s\geq 10$, $1\leq k\leq\floor{q-1}$.
		Let $\phi_{app}\in \Aext{b}^{s+k;q_0-1/2}(\McompIn)$ with $q_0>1$ such that 
		$f_0=\Box_{\ern}\phi_{app}-\mathcal{N}[\phi_{app}]\in\Aext{b}^{s+k;(q-1/2)+}(\McompIn)$,  for $q>2$.
		For $T_\mathrm{f}\gg1$, there exists a unique scattering solution in $\D_{T_\mathrm{f}}=\{\McompIn\cap t_*>T_\mathrm{f}\}$ to
		\begin{equation}\label{scat:eq:scatIn}
			\Box_{\ern}(\phi+\phi_{app})-\mathcal{N}[\phi_{app}+\phi]=0,\quad \phi|_{\mathcal{B}_\mathfrak{t}}=T\phi|_{\mathcal{B}_\mathfrak{t}}=0.
		\end{equation}
		satisfying $\phi\in\Aext{b}^{s-3;q-1/2}(\McompIn)+\A{b}^{s-3;\infty,q-1/2,k-1/10}(\McompIn)$ and $	(\Vb\cup\Vc)^k\phi|_{\hor}\in\A{b}^{s-3;q}(\hor)$.
	\end{prop}

	\begin{proof}[Proof of \cref{scat:lemma:interior}]
		\emph{Existence:}
		We consider the truncated problem
		\begin{equation}
			\Box_{\ern}\phi_1=\mathcal{N}[\phi_{app}+\phi_1]-\mathcal{N}[\phi_{app}]+\chi(t_*/T_1)(\mathcal{N}[\phi_{app}]-\Box_{\ern}\phi_{app})
		\end{equation}
		For some $T_1>T_\mathrm{f}$, and look for a solution with $\phi=0$ for $t_*>T_1$.
		From local existence for wave equations, we know that a solution exists in the region $t_*\in(T_\mathfrak{f},T_1)$ for some $T_\mathrm{f}<T_1$, and that $H^4$ norm of $\phi$ can be used as a breakdown criteria.
		
		We prove that for $T_{\mathrm{f}}$ sufficiently large and some fixed $s\geq4,\, k\geq 1$, the norm
		\begin{equation}
			\norm{\Vc\Ve^k\phi}_{\Hb^{s,q}(\D_{T_\mathrm{f}})}
		\end{equation}
		is uniformly controlled.
		Applying the linear estimate from \cref{lin:prop:main} together with the nonlinear bound \cref{lin:lemma:nonlin}, we obtain
		\begin{equation}
			\norm{\Vc^k\Ve \phi}_{\Hb^{s,q}(\D_{T_\mathrm{f}})}\lesssim_{s,k} T_{\mathrm{f}}^{-1/2} \norm{\Vc^k\Ve \phi}_{\Hb^{s,q}(\D_{T_\mathrm{f}})}+\norm{\Vc^kf_0}_{\Hb^{s,q}(\D_{T_\mathrm{f}})}++\norm{\Vb^{s+k+1}\phi|_{\mathcal{B}}}_{\Hb^{0;q-3/2}(\mathcal{B})},
		\end{equation}
		provided that the left hand side is smaller than 1.
		Using \cref{lin:lemma:data} we obtain that the right most term is controlled by the assumption on $f$.
		For $T_{\mathrm{f}}$ sufficiently large, we can absorb the first term on the left to the right and obtain uniform estimates.
		Existence follows by a limiting argument:
		Denoting by $\phi_{n}$ the above found solutions for $T^1=2^n$, we can find a sub sequence in $\Hb^{s-1,q-\epsilon}(\Mcomp)$ for any $\epsilon>0$ such that the convergence $\phi_n$ converges to a limit.
		
		\emph{Uniqueness:}
		Let $\tilde{\phi},\bar{\phi}\in\Aext{b}^{4;q-1/2}(\McompIn)$ be two solutions to \cref{scat:eq:scatIn}, and set $\phi=\tilde{\phi}-\bar{\phi}$.
		Then, $\phi$ satisfies a linear equation with no data and no forcing.
		We apply the energy estimate from the proof of \cref{lin:prop:main} with $1<q'<q$ in a bounded region $\D_{T_\mathrm{f}}^{T_1}$, and keep track of the boundary term at $T_1$ to get 
		\begin{equation}
			\norm{\Ve\phi}_{\Hb^{q'-1/2,1/2-1/10}(\D^{T_1}_{T_2})}\lesssim\int_{t_*=T_1}J^T_{2q'-5/2}[\phi].
		\end{equation}
		The boundary term goes to zero by assumption on $\tilde{\phi},\bar{\phi}$, and thus $\phi=0$.

		\emph{Regularity:}
		Via \cref{not:eq:sobolev2}, we obtain that $\Vc^{k}\phi\in\A{b}^{s-3;q-1/2,-1/10}(\McompIn)$, since we can take $s$ arbitrarily large.
		In particular, $\partial_{\rho_\hor}^k\phi|\A{b}^{s-3;q-1/2,-1/10}(\McompIn)$.
		Integrating $k$ times yields the result.

		\emph{Flux:}
		Finally, having found a solution, we can simply apply a $J^T_{N,0}[\phi]$ current, and bound all nonlinearities with the apriori knowledge that $\norm{\Vc^k\Ve \phi}_{\Hb^{s,q}(D_{T_\mathrm{f}})}$ is bounded.
		
	\end{proof}
	
	\subsection{Exterior solution}
	\begin{prop}\label{scat:prop:main}
		Let $q>2$, $q_0>3/2$, $s>10$, $0\leq k<\floor{q-1/2}$.

		Let $\psi_\scri=0$, $\phi|_{\hor}\in \A{b}^{s+2k+2;(q-1/2)+}(\hor)$, $\phi_{app}\in\Aext{b}^{s+2k;1,q_0,q_0,q_0-1}(\Mcomp)$ and $\Box_{\ern}\phi_{app}-\mathcal{N}[\phi_{app}]=f\in\Aext{b}^{s+2k;(3,q+5/2,q+1,q-1/2)+}(\Mcomp)$.
		Then, for $T_\mathrm{f}$ sufficiently large, there exists a unique scattering solution to 
		\begin{equation}\label{scat:eq:ext}
			\Box_{\ern}(\phi+\phi_{app})-\mathcal{N}[\phi_{app}+\phi]=0,\qquad r\phi|_\scri=\psi_\scri,\qquad \phi|_\hor=\phi_\hor
		\end{equation}
		satisfying $\phi\in\Aext{b}^{s;1,q+1/2,q,q-1/2}(\Mcomp)+\A{b}^{s-3;1,q+1/2,q,q-1/2,k+1/4}(\Mcomp)$.
	\end{prop}
	
	The first step of the proof is to put \cref{scat:eq:ext} into a form, so that \cref{lin:prop:main} is applicable.
	
	\begin{lemma}[Peeling]\label{scat:lemma:peeling}
		Let $\psi_\scri,\phi_\hor,f,\phi_{app},k$ be as in \cref{scat:prop:main}.
		Then, for any $j\leq k$ there exists $\phi\in \Aext{b}^{s+2(k-j);1,q+1/2,q,q-1/2}(\Mcomp)$ satisfying $r\phi|_{\scri}=\psi_\scri$, $\phi|_{\hor}=\phi_\hor$
		\begin{equation}\label{scat:eq:peeling}
			\Box_{\ern}(\phi+\phi_{app})-\mathcal{N}[\phi_{app}+\phi]\in\rho_\hor^j\Aext{b}^{s+2(k-j);3,q+5/2,q+1,q-1/2}(\Mcomp).
		\end{equation}
	\end{lemma}
	\begin{proof}
		We write $\phi_{\leq j}=\phi_{0}+\phi_1+...+\phi_j$, such that \cref{scat:eq:peeling} holds for $\phi_{\leq j}$.
		
		We set $\phi_0=\chi(t_*(r-M))\phi_\hor$, where $\chi_\hor$ to be a cut-off supported in an open neighbourhood of a $\hor$.
		
		Assume by induction that $\phi_{\leq j}$ is already constructed such that \cref{scat:eq:peeling} holds, with $f_j$ denoting the right hand side.
		For $\phi_{j+1}\in \rho_\hor^{j+1}\Aext{b}^{s+2(k-j)+2;\infty,\infty,\infty,q-1/2}(\Mcomp)$,  we can compute in $(r,t_*)$ coordinates that
		\begin{equation}
			\Box_{\ern}(\phi_{j+1})-\mathcal{N}[\phi_{app}+\phi_{j+1}]=2\partial_v\partial_r\phi_j-\partial_r\phi_j\partial_v\phi_{app}+\rho_{\hor}^{j+1}\Aext{b}^{s+2(k-j);\infty,\infty,\infty,q-1/2}(\Mcomp).
		\end{equation}
		We can integrate the leading order part of this equation for $\partial_r\phi_j$ in $v$  with $f_{j}$ on the right hand side to obtain 
		$\partial_r\phi_{j+1}\in\rho_{\hor}^{j}\Aext{b}^{s+2(k-j);\infty,\infty,\infty,q-1/2-1}(\Mcomp)$ satisfying in $(r-M)t_*<1$
		\begin{equation}
			\partial_v\partial_r\phi_{j+1}-\partial_r\phi_{j+1}\partial_v\phi_{app}=e^{\phi_{app}}\partial_ve^{-\phi_{app}}\partial_r\phi=f_{j}. 
		\end{equation} 
		Integrating in $r$ with $\phi_{j+1}|_\hor=0$ we obtain 		$\phi_{j+1}\in\rho_{\hor}^{j+1}\Aext{b}^{s+2(k-j);\infty,\infty,\infty,q-1/2-1}(\Mcomp)$.
		We compute using \cref{map:eq:model_op} and \cref{lin:lemma:nonlin}
		\begin{nalign}
			\Box_{\ern}\phi_{j+1}-f_j\in\rho_\hor^{j+1}\Aext{b}^{s+2(k-j);3,q+5/2,q+1,q-1/2}(\Mcomp);\\
			\mathcal{N}[\phi_{app}+\phi_{\leq,j+1} ]-\mathcal{N}[\phi_{app}+\phi_{\leq,j} ]\in \rho_\hor^{j+1}\Aext{b}^{s+2(k-j);3,q+5/2,q+1,q-1/2}(\Mcomp).
		\end{nalign}
		This yields the result.
	\end{proof}
	
	\begin{proof}[Proof of \cref{scat:prop:main}]
		\emph{Step 1)}
		We apply \cref{scat:lemma:peeling} to set $\psi_\scri=\phi_\hor=0$ and obtain $\rho_\hor^{k}$ with $k$ vanishing towards $\hor$.
		
		\emph{Step 2)}
		We study a cutoff problem with $f$ multiplied by $\chi(t_*/T_2)$ where $\chi$ is a smooth cutoff function satisfying $\chi|_{x<1}=1$ and $\chi|_{x>2}=0$.
		The result follows from taking a limit to obtain the solution $\phi$.
		Accordingly, for the rest of the proof, we assume that $f$ has compact in $t_*$ support and we look for a solution to \cref{scat:eq:ext} satisfying $\phi=0$ for $t_*\gg1$ and we prove uniform estimates independent of the support assumption.
		
		\emph{Step 3)}
		Next, we also apply a cutoff function, localising $f$ away from $\hor,\scri$, so we take $f_n=\chi((r-M)2^n)\chi(r^{-1}2^{n})f$.
		It follows from \cref{not:eq:sobolev} that
		\begin{equation}
			\norm{f_n}_{\Hb^{s;3/4,q+1/2,q+1/2,q-1/2,k-1/4}}\lesssim \norm{f}_{\Hb^{s;3/4,q+1/2,q+1/2,q-1/2,k-1/4}}.
		\end{equation}
		The right hand side is finite by assumption.
		Note that at this step, it is crucial that we apply the peeling in \emph{Step 1)}.
		
		We consider the solution $\phi_n$ to \cref{scat:eq:ext}, with $f_0$ replaced by $f_n$. 
		We apply \cref{lin:prop:main,lin:lemma:nonlin}, to obtain uniform control over $\phi_n$ as $n\to \infty$ in the norm
		\begin{equation}
			\norm{\Ve \phi}_{\Hb^{s;-1/4,q-3/2,q-1/2,q-1/2,k+3/4}(\Mcomp)}.
		\end{equation}
		Hence, we get a solution $\phi$ that solves \cref{scat:eq:ext} with $f$.
		
		Existence and uniqueness follow from the same argument as in \cref{scat:lemma:interior}.
	\end{proof}

	\section{Multi-black hole scattering}\label{sec:multi}
	To highlight the robustness of the scheme, we consider nonlinear scalar equations on a multi-black hole backgrounds.
	More concretely, we consider the same nonlinearity as before, on a metric that  \emph{does not} solve the Einstein equations, but it does describe $N$ extremal black holes moving away from one another on hyperbolic orbits.
	We construct scattering solutions in the late future, i.e. a small neighbourhood of $\ip$.
	
	\subsection{Geometry}
	
	Let us recall that in $(t_*,r)$ coordinates $\ern$ takes the form \cref{g_t*}.
	In the following, we will introduce a Lorentzian manifold $(\M_{\m,\mathfrak{t}},g_\m)$ with $\M_{\m,\mathfrak{t}}\subset\R^{1+3}$.
	We will use $(s,x)$ for the standard temporal and spatial coordinates on $\R^{1+3}$.
	For $z\in \mathring{B}$ and $\gamma_z:=(1-\abs{z}^2)^{-1/2}$, introduce the Lorentz boost map
	\begin{equation}
		\Phi_z:(s,x)\mapsto (t_z,x_z):=\big(\gamma_z(s-z\cdot x),x+(\gamma_z-1)(x\cdot\hat{z})\hat{z}-\gamma_zsz \big).
	\end{equation}
	Fix a finite collection $z_i\in A\subset\mathring{B}$, $\abs{A}<\infty$ corresponding to  velocities of the final black holes.
	Without loss of generality we take $0\in A$.
	We define a multi extremal black hole metric
	\begin{equation}\label{nonlinear:eq:multi_metric}
		g_{\mathrm{m}}:=\eta+\sum_{z\in A} \chi(x_{z}/t_{z})\Big(\Phi_{z}^* (\ern)-\eta\Big),
	\end{equation}
	where $\chi$ is a localising to a $\delta \ll1$ neighbourhood of $0$, such that for $s>T_\mathrm{f}\gg 1$ we have $\supp \chi(x_{z}/t_{z})$ all pairwise disjoint.
	Here, $\ern$ is expressed with respect to $(t_*,r)$ coordinates, and the pullback is defined accordingly.
	$T_\mathrm{f}$ stands for final time, and we will only construct solutions in the region $T_\mathrm{f}$ is large.
	In fact, the construction will only work in $s-\abs{x}> T_\mathrm{f}$, but it is well known how to perform the scattering construction in the remaining region, see for instance \cite{kadar_scattering_2025}.
	
	Let us note that on $\abs{x_z}=\text{const}$ it holds that $s=\gamma_z^{-1}t_z+\mathcal{O}(1)$, and thus $s^{-1}$ may be used locally as the boundary defining function for each component of the compact manifold $\Mcomp_{\m,\mathfrak{t}}$ that we are to define.

	\begin{remark}[Asymptotically flat time coordinates]
		In our work, we localise the perturbation of $\ern-\eta$ to be away from $\scri$.
		Provided that one wishes to keep the sum of $\ern$ black holes all the way to $\scri$ it is  necessary to choose coordinates $(s,x)$ in which $s-\abs{r}$ is an approximately null coordinate, so that this property will be preserved after applying Lorentz boosts $\Phi_z$.
		A possible such choice is to use on $(\M,\ern)$ the ingoing-flat coordinate $s$ satisfying $\ern^{-1}(\dd s,\dd s)\leq0$ and
		\begin{equation}
			s=v+\begin{cases}
				0&r<10M,\\
				-2r_*+r&r>20M.
			\end{cases}
		\end{equation}
		This choice implies that $s-r$ foliates $\scri$ in $r>20M$, while $s$ foliates $\hor$ in $r<10M$.
	\end{remark}
	
	\begin{remark}[Possible extensions]\label{multi:rem:extension}
		One could replace $g_\m$ with a metric that is not exactly equal to $\ern$ but only settles down to $\ern$ at a polynomial rate.
		More precisely, using the present techniques, it suffices if $g_\m$ converges in the near horizon ($\F$), compact ($\K$) and self similar ($i^+$) regions to the respective leading order nondegenerate quadratic Lorentzian metric at a rate $\A{b}^{\cdot,0+,1+,0+}(\multiComp)$:
		Denoting by $\dd x_\mu$ elements from $\{\dd x,\dd t\}$, by $\dd\omega$ one forms on the sphere, we would require for some quadratic forms $Q$ (changing from line to line)
		\begin{nalign}
			g_\m\equiv \eta & \mod \A{b}^{0+}(\{r/t\in(1/4,1/2)\})Q(\dd x_\nu,\dd x_\mu);\\
			g_\m\equiv \ern & \mod \A{b}^{1+}(\{r\in(2M,10M)\}) Q(\dd x_\mu,\dd x_\mu);\\
			g_\m\equiv g_{\mathrm{near}} & \mod \A{b}^{0+}(\{(r-M)v\in(-1,1)\}) Q\Big(\frac{\dd r}{(r-M)},(r-M)\dd v,\dd \omega\Big).
		\end{nalign} 
		The loss of $1+$ at the compact region is due to trapping and is well known how to improve on it, see \cref{lin:rem:Morawetz}.
		We pursue no such extension of our result in this introductory work.
	\end{remark}

	It is straightforward to check that $g_\m$ describes a Lorentzian metric on the manifold
	\begin{equation}
		\M_{\m,\mathfrak{t}}:=\{(s,x)\in\R^{1+3}: s-\abs{x}>T_f,\, \abs{x_{z}}-M>-\mathfrak{t}/t , \quad\forall z\in A\}.
	\end{equation}
	Moreover, by construction, each part of the spacetime $\M_{\m,\mathfrak{t}}\cap\{\abs{x_{z}}<\delta s_{z}\}$ is isometric to $(\M,\ern)$.
	One can also check, that $\M_\m$ describes a spacetime that has a future complete null infinity and a black hole region ($\hor_{g_\m}$) formed by $\abs{A}$ disconnected components.
	Finally, we mention that $\cup_z\hor_z=\{\abs{x_i}=M\}$ is \emph{not} the black hole horizon\footnote{Since, this global geometric facts plays no role here, we leave it to the reader to check explicitly.} for $g_\m$, but approximate them as $s\to\infty$.
	
	We define the following vectorfields in $\M_{\m,\mathfrak{t}}$
	\begin{equation}
		T=\partial_s|_x,\quad T_z=\partial_{t_z}|_{x_z}.
	\end{equation}
	
	Let us introduce boundary defining functions
	\begin{equation}
		\rho^\hor_z=s(\abs{x_z}-M),\quad\rho_z^\F=\frac{1+(\abs{x_{z}}-M)s}{\abs{x_z}s},\quad \rho^\K_z=\frac{\jpns{x_{z}}}{s\rho_z^\F},\quad \rho^{+}=\frac{\sum_{z\in A}(\rho_z\rho_z^\F)^{-1}}{s-\abs{x}},\quad \rho^\scri:=1-\abs{x}/s.
	\end{equation}
	We define $\multiComp$ to be the manifold with corners as a compactification of $\M_{{\m,0}}$ obtained by extending all the functions $\rho^\hor_z,\rho^\F_{z_i},\rho^\K_{z_i},\rho^+,\rho_\scri$ to 0 smoothly.
	We have the following boundary components:
	\begin{itemize}
		\item $\scri\cong \Rcomp\times\sphere$ and $\mathring{\ip}\cong B\setminus A$ similar to $\Mcomp$;
		\item $\K_{z}$, spacially compact timelike infinity, the endpoint of curves that are a constant distance away from the black hole with velocity $z_i$;
		\item $\F_{z}$,  the near horizon geometry of black hole with velocity $z_i$;
		\item $\hor_z$, the image of the Extremal Reissner-Nordstrom horizon under $\Phi_z$;
		\item boundary hypersurface $\Sigma_{T_f}:=\{s-\abs{x}=T_f\}$ is an artificial boundaries.
	\end{itemize} 
	
	We also introduce $\multiCompin$ as the compactification of $\M_{\m,\mathfrak{t}}\cap\{\abs{x_z}<M\}$ by extending $\rho^\hor_z,\rho^\F_z$ smoothly to 0, which has
	$\mathcal{B}_z=\{\abs{x_{z_z}}-M=-\delta/s\}$ as an additional artificial boundary.
	We note that $\multiCompin$ is isomorphic to $\McompIn$.
	Finally, we define $\Mcomp_{\m,\mathfrak{t}}$ as the compactification of $\M_{\m,\mathfrak{t}}$ with $\rho^\K_z,\rho^+,\rho^\scri$ and $\rho^{F\prime}_z=(\abs{x_z}-M)s$ extending to $0$ smoothly.
	See \cref{fig:multi} for a diagrammatic representation.
	
	Next, we extend the notation from \cref{sec:setup} to capture the extra regions associated to the geometry of $\multiComp$.
	Given $(a_\scri,a_+,a_{K,1},...,a_{K,N},a_{F,1},....,a_{F,N})=(a_\scri,a_+,\vec{a}_\K,\vec{a}_\F)=\vec{a}\in \R^{2+2N}$ we write $\A{b}^{\vec{a}}(\multiComp)$ for the conormal space of functions on $\multiComp$.
	Whenever $a_{F,1}=a_{F,2}=...=a_{F,N}$, we simply write $\overrightarrow{a_{F,1}}$ for $\vec{a}_{\F}$.
	
	Finally, let us introduce $T^z,\Omega^z_{ij},Y^z$ that are the vectorfields defined in \cref{not:sec:spacetime} pulled back via the diffeomorphism $\Phi_{z}$.
	\subsection{Main result}
	
	We study scattering solutions for a seminilinear wave equation on $\M_{\m,\mathfrak{t}}$.
	We are not interested in obtaining sharp result regarding the admissible form of nonlinearities nor to optimize the required decay at $\scri$.
	We consider the semilinear equation $\mathcal{N}[\phi]=g_{\mathrm{m}}^{-1}(\dd \phi,\dd\phi)+\phi^3$.
	
	\begin{theorem}\label{multi:thm:main}
		Let $q>3/2$ and let $\psi^\scri\in\A{b}^{q-1}(\scri)$.
		Then, there exists a smooth solution $\phi\in\A{b}^{1,(q,\vec{q},\vec{q})-}(\Mcomp_{\m,\mathfrak{t}})$ of the half-scattering problem
		\begin{equation}\label{multi:eq:scat}
			\Box_{g_\m}\phi=\mathcal{N[\phi]},\qquad r\phi|_{\scri}=\psi^\scri.
		\end{equation}
	\end{theorem}
	
	The proof of \cref{multi:thm:main} is divided into an ansatz construction and a correction part just as for \cref{intro:thm:rough}.
	These are contained in \cref{multi:prop:ansatz,multi:prop:nonlin} respectively.
	
	\begin{remark}[Finite regularity]
		Provided that $\psi^\scri\in\A{b}^{n;q-1}$ has finite regularity, the solutions that we obtain are also only of finite regularity $\phi\in\A{b}^{n';q-1}$, but we merely have $n\sim n'$, with a not implicit constant.
	\end{remark}
	
	\begin{remark}[Limited regularity for scattering solutions]
		We also mention, that as in the case of \cref{scat:prop:main}, the exact solutions that we construct with an inverse polynomially decaying data along $\cup_z \B_z$ and $\scri$ are smooth in $\M_{\m,\mathfrak{t}}$, except along the hypersurfaces $\hor_{z}$.
		As we already mentioned, we believe, that since $\cup_z\hor_{z}\neq \hor_{g_\m}$, there is a nonempty region between $\hor_{g_\m}$ and $\hor_{z}$, as indicated on \cref{fig:multi}.
		On the figure we have 3 region i) the region between $\mathcal{B}_z$ and $\hor_z$ ($\multiCompin$) denoting the region where the scattering solution only depends on the data at $\B_0$; ii) the region between $\hor_{g_\m}$ and $\hor_z$ where the solution depends on $\cup_z \B_z$; iii)  $\M_{\m,0}$ where the solution depends on all the data.
		Therefore, the appearance of limited regularity is \emph{not} connected to the black hole horizon, rather the domain of influence of individual black hole interiors.
	\end{remark}
	
	\begin{figure}[htbp]
		\centering
		\includegraphics[width=0.6\textwidth]{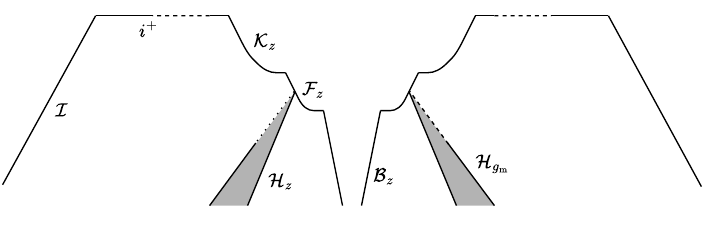}
		\caption{We drawn the compactification $\Mcomp_{\m,\mathfrak{t}}$ on which our scattering result takes place.
			We only indicated how the geometry looks around a single black hole, with $\hor_{g_\m},\hor_{\ern}$ denoting the black hole horizon of $g_\m$ and the image of the horizon of $\ern$ under the diffeomorphism $\Phi_z$. 
			Part of $\hor_{g_\m}$ is dotted as we do not compute the precise convergence rate of the two hypersurfaces and thus the asymptotic behaviour on the compactification $\Mcomp_{\m,\mathfrak{t}}$.}
		\label{fig:multi}
	\end{figure}
	\subsection{Ansatz}
	We begin by constructing an approximate solution to \cref{multi:eq:scat}.
	
	\begin{prop}\label{multi:prop:ansatz}
		Let $q>3/2$ and let $\psi^\scri\in\A{b}^{q-1}(\scri)$.
		Then, there exists a smooth ansatz $\phi_A\in\A{b}^{1,(q+1,\vec{q+1},\vec{q+1})-}(\Mcomp_{\m,\mathfrak{t}})$ of \cref{multi:eq:scat} satisfying
		\begin{equation}
			\Box_{g_\m}\phi_A-\mathcal{N[\phi_A]}\in \A{b}^{1,\infty,\vec{\infty},\vec{\infty}}(\Mcomp_{\m,\mathfrak{t}}),\qquad r\phi_A|_{\scri}=\psi^\scri.
		\end{equation}
	\end{prop}
	\begin{proof}
		Since the proof of \cref{an:thm:ansatz} is completely modular, it applies to the multi-black hole setting as well,  whenever we solve model problems on the disjoint neighbourhood of $\K_z$ and $\F_z$.
		Let us show how the model problem at $\ip$ is treated, and this will complete the proof.
		\begin{claim}[Improvement at $i^+$]
			Fix $s>3/2$, and $s'\in\big(\min(s-1,3/2),s\big)$.
			For $f\in\A{b}^{3,s+2,\vec{s}',\vec{\infty}}(\multiComp)$, there exists $\bar{\phi}\in\A{b}^{1,s,\vec{s}',\vec{\infty}}(\multiComp)$ such that 
			\begin{equation}
				\Box_{g_\m}\bar\phi-f\in\A{b}^{3,s+3,\overrightarrow{s}',\vec{\infty}}(\multiComp),\qquad r\phi|_{\scri}=0.
			\end{equation}
		\end{claim}
		We write $f_{z}=\chi(x_{z}/s_{z})f$ for $\chi$ as in \cref{nonlinear:eq:multi_metric}, and let $f_+=f-\sum_{z\in A} f_{z}$.
		Next, we apply \cref{an:prop:i+_conormal} for each $f_{\bullet}$ and localise the corresponding $\bar\phi_{\bullet}$ to the region $\cap_j\{\abs{x_{z_j}}>10M\}$.
		Thus, we get that $\Box_{g_\m}\bar\phi-f\in\A{b}^{3,s+3,\vec{s}',\vec{\infty}}(\multiComp)$.
		
		The nonlinear terms are still pertrubative following the same computation as in the proof of \cref{an:thm:ansatz}.
	\end{proof}
	
	\subsection{Linear estimate}
	
	It is useful at this point, to introduce the collection of vectorfields that we can uniformly bound at top order
	\begin{equation}
		\mathring{\Ve}:=\bigcup_{z\in A}\chi_z\{D_z^{1/2}\partial_{\abs{x_z}},\rho_{\F_z}^{-1}\rho_{+}^{-1}T_z,\rho_\hor^{1/2}\Omega_{i}^z\}
		\bigcup \chi_{\mathrm{far}} \{\rho_\scri^{1/2}\Omega_{i},r\partial_r|_{t_*},t_* T\},
	\end{equation}
	where $\chi_z$ localises to $\{\abs{x_z}/t_z<\delta\}$ and $\chi_{\mathrm{far}}$ localises to the complement region.
	We also set $\Ve:=\mathring{\Ve}\cup \{1\}$.
	
	We also introduce the commutators
	\begin{equation}
		\Vc:=\{s^{-1}\chi_z(\abs{x_z})\partial_{\abs{x_z}}|_{t_z}:z\in A\}
	\end{equation}
	where $\chi_z(x_z)$  localises to $\{\abs{x_z}<2M\}$.
	
	The main result of this section is the following linear estimate:
	\begin{prop}\label{multi:prop:linear}
		Fix $n\geq1$, and $f$ smooth supported away from $\partial\multiComp$.
		There exists $q\sim_A n$, such that for $k\leq n/2$ the following estimate holds
		\begin{equation}
			\norm{\Ve\Vc^k\phi}_{\Hb^{n;-1/4,q-3/2,\vec{q}-1/2,\vec{q}-1/2,\overrightarrow{3/4}}(\multiComp)}\lesssim_n\norm{\Vc^kf}_{\Hb^{n;3/4,q+1/2,\vec{q}+1/2,\vec{q}-1/2,\overrightarrow{-1/4}}(\multiComp)}
		\end{equation}
	\end{prop}
	
	The proof of \cref{multi:prop:nonlin} will essentially follow that of \cref{lin:prop:main}, with the main novelty, that we need a coercive multiplier replacing $t_*^{2N}T$.
	\begin{lemma}[Time functions]\label{multi:lemma:tau}
		There exist function $\tau,\tau_*$ with the properties that $\dd\tau,\dd\tau_*$ are causal in $\multiComp$ with uniform bound
		\begin{equation}\label{multi:eq:tau_uniform}
			g_m^{-1}(\dd\tau,\dd\tau),g_m^{-1}(\dd\tau_*,\dd\tau_*)<c<0\quad\text{ in } \multiComp\cap \{\abs{x_z}>3M\}\cap\{1-x/s>\delta\},
		\end{equation}
		together with the following properties
		\begin{equation}
			\tau_*=\begin{cases}
				s-\abs{x}& \text{in } \{1-\abs{x}/s<\delta/2\},\\
				\gamma_z^{-1}t_z & \text{in } \{\abs{x_z}/t_z<\delta/2\},
			\end{cases}\qquad
			\tau=\begin{cases}
				s& \text{in } \{1-\abs{x}/s<\delta\},\\
				\gamma_z^{-1}t_z & \text{in } \{\abs{x_z}/t_z<\delta/2\}.
			\end{cases}
		\end{equation}
	\end{lemma}
	\begin{proof}
		The construction is motivated by \cite{kadar_construction_2024}.
		We set
		\begin{subequations}\label{multi:eq:tau}
			\begin{empheq}[left={\tau= \empheqlbrace}]{alignat=3}
				&s && \text{if }\cap \{ x_{z}/s_{z}>\delta\},\label{multi:eq:tau_ext}\\
				&(1-\abs{z}^2)t_{z} && \text{if } \{ x_{i}/s_{i}<\delta/2\},\label{multi:eq:tau_int}\\
				&s-\chi_zz\cdot x \qquad&& \text{if } \{x_{z}/s_{z}\in(\delta/2,\delta)\}\label{multi:eq:tau_transition},
			\end{empheq}
		\end{subequations}
		where $\chi_z=\chi(\abs{x_z}/t_z)$ is a cutoff function to be determined.
		Clearly, $\tau$ is timelike outside of the transition zones in \cref{multi:eq:tau_transition}.
		We compute in this zone $\tau=\gamma_z(t_z+\chi z_z\cdot x_z)$ and thus 
		\begin{equation}
			\gamma_z^{-2}\eta(\dd\tau,\dd\tau)\leq-1+\chi'{}^2\frac{(z\cdot x_z \abs{x_z})^2}{t_z^4}+\eta\Big(\chi z\cdot \dd x_z+\chi'\frac{z\cdot x_z\dd\abs{x_z}}{t_z}\Big).
		\end{equation}
		We can bound the last term by $\abs{z}^2(1+\sup_\rho \rho\chi'(\rho))$.
		Choosing a cutoff such that $\sup_\rho \rho\chi'(\rho)$ is sufficiently small yields that $\dd\tau$ is strictly timelike in the transition region.
		
		For $\tau_*$ the same argument works, with $t_z$ behaviour in the near horizon region and an extra transition to a $\scri$ penetrating coordinate in the far region.
	\end{proof}
	
	\begin{lemma}[Vectorfield]\label{multi:lemma:tildeT}
		There exists a timelike vectorfield $\tilde{T}$
		with uniform bound as in \cref{multi:eq:tau_uniform}
		satisfying
		\begin{equation}
			\tilde{T}=\begin{cases}
				T & \text{in } \cap \{\abs{x_z}/t_z>2\delta\}\\
				\gamma_z^{-1} T_z & \text{in } \{\abs{x_z}/t_z<\delta/4\}.
			\end{cases}
		\end{equation}
		and for some quadratic form $Q$ the estimate
		\begin{equation}\label{multi:eq:pi_tildeT}
			\pi^{\tilde{T}}\in \A{b}^{\infty,2,\infty,\infty}(\multiComp) Q(\dd x_\mu,\dd x_\nu).
		\end{equation}
	\end{lemma}
	\begin{proof}
		We simply take $\tilde{T}=g_\m^{-1}(\dd \tau,\cdot)$.
		
	\end{proof}
	
	\begin{lemma}[N-body exterior current]\label{multi:lemma:T_multiplier}
		Let $\tau,\tau_*,\tilde{T}$ as above and consider $N-\delta,\delta>0$.
		There exists $N$ sufficiently large, depending on $A\subset\mathring{B}$, such that 
		$J^{\tilde{T}}_{N,\delta}=\frac{1}{N}\tau_*^{N-\delta}\tau^\delta\tilde{T}\cdot \T$ satisfies the coercivity
		\begin{equation}\label{multi:eq:V2}
			-\Div J^{\tilde{T}}_{N,\delta}+\mathfrak{F}\gtrsim\tau_*^{N-1-\delta}\tau^{\delta} \underbrace{\begin{cases}
					D_{z}(\partial_{\abs{x_{z}}}\phi)^2+\big(1+\frac{\delta\tau_*}{ND\tau}\big)(T_{z}\phi)^2+\abs{x_{z}}^{-2}\abs{\slashed{\nabla}_{z}\phi}^2 &  \abs{x_{z}}/s_{z}<\delta,\\
					(\partial_{\abs{x}}\phi)^2+\big(1+\frac{\delta\tau_*}{N\tau}\big)(T\phi)^2+\abs{r}^{-2}\abs{\slashed{\nabla}\phi}^2 & \text{else},
			\end{cases}}_{E_T[\phi]}
		\end{equation}
		where $\mathfrak{F}=\delta^{-2}\tau_*^N\rho_\K^{-2}\rho_+^{-2}\rho_\scri^{-1-\delta}\abs{\Box_{\ern}\phi}^2$.
		
		Let $\chi_z$ be a localiser onto $\{\abs{x_z}<2M\}$ and $\chi_{\mathrm{far}}$ localising onto $\{1-\abs{x}/s<\delta\}$.
		Let $J^1_{>,N}=\tau_*^{N-1-\delta}\tau^\delta\chi_{\mathrm{far}}\frac{1}{N}r\partial_r|_{\tau_*}\cdot\T$ .
		Let $J^{1,z}_{<,N}$ be as in \cref{curr:lem:ext} with respect to $t_z,r_z$ coordinates.
		There exist $C$ sufficiently large, such that
		$J_{N,\delta}=CJ^{\tilde{T}}+J^1_{>,N}+\sum_zJ^{1,z}_{<,N}$ satisfies
		\begin{equation}\label{multi:eq:master1}
			-\Div J_{N,\delta}+\mathfrak{F}\gtrsim \tau_*^{N-1-\delta}\tau^\delta E_T[\phi]+\tau_*^{N-2-\delta}\tau^\delta r\chi_{\mathrm{far}} (Y\phi)^2+\sum_z\chi_z \tau_*^{N-2-\delta}\tau^\delta D_z'(Y_z\phi)^2,
		\end{equation}
		where $\mathfrak{F}=\delta^{-2}\tau_*^N\rho_\K^{-2}\rho_+^{-2}\rho_\scri^{-1-\delta}\abs{\Box_{\ern}\phi}^2$.
	\end{lemma}
	
	\begin{proof}
		The one form $\dd \tau_*^{N-\delta}\tau^\delta$ is timelike, since both $\dd \tau_*,\dd \tau$ are.
		Therefore, we compute that
		\begin{equation}
			-N\Div J^{\tilde{T}}_{N,\delta}\geq (\dd(\tau_*^{N-\delta}\tau^\delta) \times \tilde{T})\cdot \T- \dd \tau_*^{N-\delta}\tau^\delta \abs{\pi^{\tilde{T}}\cdot\T}.
		\end{equation}
		From \cref{multi:eq:pi_tildeT} and the uniform timelike property \cref{multi:eq:tau_uniform} together with the coercivity \cref{curr:eq:fluxT}, it follows that for $N$ sufficiently large, the first term on the right is at least twice as large as the second.
		Hence, we may drop the negative term by introducing an implicit constant.
		
		The $J^{1,z}_{<,N}$  terms can be introduce into a current the same way as in \cref{curr:lem:ext}.
		For $J^1_{>,N}$, the error terms arising from the cutoff can be estimated using the uniform control provided by $J^T_{N,\delta}$ for all derivatives.
		This yields \cref{multi:eq:master1} with an error 
		\begin{equation}
			\mathfrak{F}'=\tau_*^N\abs{\Box_{g_\m}\phi}\big(\tau_*^{-\delta}\tau^\delta\abs{\tilde{T}\phi}+\abs{\chi_{\mathrm{far}}\tau_*^{-1}rY\phi}+\abs{\chi_z\tau_*^{-2}Y_z\phi}\big).
		\end{equation}
		Using Cauchy-Schwarz, and the control on the right hand side of \cref{multi:eq:master1}, we get the result.
	\end{proof}
	
	Next, we show how to recover $\Ve$ from $\mathring{\Ve}$ via a Poincare type estimate.
	
	\begin{lemma}[Poincare]\label{multi:lemma:poincare}
		Fix $a_+,a_\K,a_\F$ satisfying $a_\K>\max(a_+-3/2,a_\F+1/2)$ and $\min(a_\F,a_+)>-2$.
		For $\phi\in\A{b}^{a_\scri,a_+,\vec{a}_\K,\vec{a}_\F}(\multiComp)$, it holds that
		\begin{equation}
			\norm{\phi}_{\Hb^{0;a_\scri,a_+,a_\K,a_\F}(\multiComp)}\lesssim\norm{\mathring{\Ve}\phi}_{\Hb^{0;a_\scri,a_+,a_\K,a_\F}(\multiComp)}.
		\end{equation}
	\end{lemma}
	\begin{proof}
		This follows using the same proof as in \cref{curr:lemma:poincare}.
	\end{proof}
	
	Next, we introduce \emph{good} commutators.
	Since $g_\m$ lacks any global almost Killing vectorfields, to obtain coercive estimates for commuted quantities, we loose some decay at each commutation step.
	
	\begin{lemma}[N-body commutator]\label{multi:lemma:commutators}
		Let $\tilde{T}$ be as in \cref{multi:lemma:T_multiplier}.
		It holds that
		\begin{equation}\label{multi:eq:commutator}
			[\Box_{g_\m},\tilde{T}]=\chi\rho_+^3\Diffb^2(\multiComp).
		\end{equation}
		for some cutoff function $\chi$ localised in $\bigcap_z\{\abs{x_z}/t_z\in(\delta/2,\delta)\}$.
		Furthermore, for $n\geq1$, and $f=\Box_{g_\m}\phi$ compactly supported, there exists $N\sim_A n$ such that
		\begin{equation}\label{multi:eq:n_commuted}
			\int_{\D_s}-\Div J_{N,q}[\tilde{T}^n\phi]\mu\lesssim\int_{\D_s} \tau_*^{-2n}(\abs{\Diff_b^n f}^2+\bar{\chi}\abs{\Diffb^n\phi}^2)\tau_*^N\rho_\K^{-2}\rho_+^{-2}\rho_\scri^{-1-q}\mu,
		\end{equation}
		where $\bar{\chi}$ is a cutoff function localised to a neighbourhood of $i^+\setminus(K,\scri)$.
	\end{lemma}
	\begin{remark}
		Using a Morawetz estimate, we could obtain smaller implicit constant in $N\sim_A n$, but still order one losses for each commutation.
	\end{remark}
	\begin{proof}
		\cref{multi:eq:commutator} follows trivially, using the decay rate of $g_\m$, and that $\tilde{T}$ is Killing outside the support of the cutoff function.
		Let us more precisely write 
		$\mathrm{RHS}\cref{multi:eq:commutator}=C\chi\rho_+^3\Vb\Vb^*$, where we used unit size vectorfields, thus we introduced a constant $C$ depending on $A$.
		
		Using \cref{multi:eq:commutator} iteratively, we get for $n\geq1$
		\begin{equation}\label{multi:eq:proof_comm1}
			[\Box_{g_\m},\tilde{T}^n]=Cn\rho_+^3\bar\chi \Vb\Vb^* \tilde{T}^{n-1}+\rho_+^{n+2}\bar\chi\Diffb^{n-1}(\multiComp)\Vb^*.
		\end{equation}
		Applying the divergence computation from \cref{multi:lemma:T_multiplier}, we see that the right hand side of \cref{multi:eq:n_commuted} already controls the second term from \cref{multi:eq:proof_comm1}.
		For the first term is controlled by the elliptic estimate in \cref{ell:lemma:ext}:
		Set $U_1=\{\delta/2<\abs{x}/t<\delta\}$ and $U_2=\{\delta/4<\abs{x}/t<2\delta\}$.
		Then
		\begin{equation}
			\norm{Cn\rho_+^3\chi \Vb\Vb^* \tilde{T}^{n-1}\phi}_{\Hb^{0;l}(U_1)}\lesssim Cn\norm{\{1,sT\}\tilde{T}^{n-1}\phi}_{\Hb^{1;-4}(U_2)}+\norm{\Box_{g_\m}\tilde{T}^{n-1}\phi}_{\Hb^{0;-2}(U_2)}.
		\end{equation}
		Noting that the right hand side of \cref{multi:eq:V2} yields control over $T\phi$ nondegenerately in terms of $N$ in $U_2$, we see that for $N\sim n$ sufficiently large, the error term arising from the first term in \cref{multi:eq:proof_comm1} can be absorbed into the left hand side of \cref{multi:eq:n_commuted}.
	\end{proof}
	
	Finally, we have 
	
	\begin{proof}[\cref{multi:prop:linear}]
		The proof is essentially the same as for \cref{lin:prop:main}, with the main difference that we need to take $N$ in $J^T_{N,1/4}$ large enough to control the error terms in \cref{multi:eq:n_commuted}.
	\end{proof}

	\subsection{Scattering}
	Finally, we show via energy estimates that we can remove the fast decaying error.
	Note that since the regions $\multiCompin$ are isometric to $\McompIn$, we can already use \cref{scat:lemma:interior}, to obtain that the solution is regular up to $\hor_z$, with a \emph{necessarily} loss of regularity towards $\hor_z$, depending on the decay rate.
	\begin{prop}\label{multi:prop:nonlin}
		Let $s\geq10$, $q_0>3/2$ $\phi_{app}\in\Aext{b}^{1,q_0,\vec{q}_0,\vec{q}_0}(\multiComp)$.
		There exists $N\sim_{A}s$, such that the following holds:
		Let $k\leq s$.
		If $\Box_{g_\m}\phi_{app}-\mathcal{N}[\phi_{app}]\in\Aext{b}^{s+2k;3,N+5/2,\vec{N}+1,\vec{N}-1/2}(\multiComp)$ and $\phi|_{\hor_z}=\psi_{\hor_z}=0$, then 
		there exists $T_{\mathrm{f}}$ large enough such that the nonlinear scattering problem
		\begin{equation}
			\Box_{g_\m}(\phi+\phi_{app})-\mathcal{N}[\phi_A+\phi]=0,\quad r\phi|_{\scri}=0,\quad \phi|_{\hor_z}=\psi_{\hor_z}
		\end{equation}
		has a unique solution $\phi\in\Aext{b}^{s-3;1,N+1/2,\vec{N},\vec{N}-1/2}(\multiComp)+\A{b}^{s-3;1,N+1/2,\vec{N},\vec{N}-1/2,k+1/4}(\multiComp)$ within $s-\abs{x}>T_\mathrm{f}$.
	\end{prop}
	We already mention, that $N(s)$ in the above theorem is not sharp, but this is not of interest in this paper.

	\begin{proof}[Proof of \cref{multi:prop:nonlin}]
		As in the proof of \cref{scat:prop:main}, we assume that the error $f=\Box_{g_\m}\phi_{app}-\mathcal{N}[\phi_{app}]$ is supported in $s<T_2$ for some $T_2\gg1$, and the solution $\phi$ is constructed by a limiting argument.
		
		We fix $N$ as in \cref{multi:prop:linear}.
	\end{proof}

	\appendix
	
	\section{Polyhomogeneous ansatz}\label{sec:poly}
	The aim of this section is to put \cref{an:thm:ansatz} into the polyhomogeneous setting.
	We  introduce the space of functions appropriate for the partial expansions in \cref{poly:sec:function_spaces}, and then revisit each step of the proof of \cref{an:thm:ansatz} in the new functional framework.
	We prove
	\begin{prop}\label{poly:prop:an}
		Fix polyhomogeneous scattering data $\psi_\scri\in\A{}^{\E_\scri}(\scri)$ for some index set $\E_\scri$ with $\min(\E_\scri)>1$.
		Then, there exists index set $\E_\hor$ with $\min(\E_\hor)=\min(\E_\scri)-1$ and data on the horizon $\phi_\hor\in\A{}^{\E_\hor}(\hor)$ together with $\phi\in\Aext{phg}^{(1,0),2,2,2}(\Mcomp)$ such that 
		\begin{equation}
			\Box_{\ern}\phi-\mathcal{N}[\phi]\in\Aext{phg}^{3,\infty,\infty,\infty}(\Mcomp),\qquad r\phi|_{\scri}=\psi_\scri,\qquad \phi_{\hor}=\phi_\hor.
		\end{equation}
	\end{prop}
	\begin{proof}
		We follow the same induction as for \cref{an:thm:ansatz}, replacing \cref{an:cor:K,an:cor:i_+_conormal,an:prop:F_interior} with \cref{poly:cor:K,poly:corr:i+,poly:corr:i+_smooth} respectively.
	\end{proof}
	
	\subsection{Function spaces with expansion}\label{poly:sec:function_spaces}
	
	For a more precise understanding of the solution, it is convenient to use polyhomogeneous function spaces to capture functions with generalised Taylor expansions towards the respective boundaries:
	\begin{definition}[Index set]
		We call  $\E\subset \R_z\times\N_k$ an index set if it satisfies that
		\begin{itemize}
			\item $\E_c:=\E\cap\{(z,k)\in\E:z\leq c\}$ has finite cardinality for all $c\in\R$;
			\item $(z,k)\in\E\implies (z+1,k)\in\E$;
			\item $(z,k)\in\E$ and $k\geq1$ $\implies$ $(z,k-1)\in\E$.
		\end{itemize}
		We will use the notation $\min(\E)=\min\{z:(z,k)\in\E\}$ as well as
		\begin{itemize}
			\item $(z,k)\geq (z',k')$ if $z>z'$ or $z=z'$ and $k\leq k'$;
			\item $(z,k)> (z',k')$ if $z>z'$ or $z=z'$ and $k< k'$;
			\item $z\geq (z',k')$ if $z\geq z'$,
			\item $\E^1\cupdex\E^2=\{(z,k):(z,k_i)\in\E^i,\, k_1+k_2+1\geq k\}$,
			\item for $A\subset\R\times\N$ with $\abs{A}<\infty$, let $\overline{A}:=\bigcap\limits_{A\subset\E}\E$ for index sets $\E$. 
			We use the shorthand $\overline{(s,k)}:=\overline{\{(s,k)\}}$.
		\end{itemize}
	\end{definition}
	
	Next, we make a simplifying assumption to streamline some of the proofs , but lifting this extra requirement is only a notational modification.
	Unless otherwise stated, we will use:
	\begin{assumption}
		Assume that \emph{all} index sets are subsets of $\N\times\N$.
	\end{assumption}
	
	\begin{definition}[Polyhomogeneous functions]
		Fix a manifold with corners $X$, an open set $U\subset X$ containing at most one corner and corresponding coordinates $\rho_1,\rho_2,y\in[0,1)^2\times\R^n$ in $U$.
		We say that a function $f:X\to\R$ is polyhomogeneous in $U$ with index sets $\E_1,\E_2$ and write $f\in\A{}^{\E_1,\E_2}(U)$ if for all $c_1,c_2\in\R$
		\begin{equation}\label{not:eq:poly_def}
			\sum_{\substack{(z_1,k_1)\in\E^1_{c_1}\\ (z_2,k_2)\in\E^2_{c_2}}} (\rho_1\partial_{\rho_1}-z_1)^{k_1+1}(\rho_2\partial_{\rho_2}-z_2)^{k_2+1}f \in \A{b}^{c_1+,c_2+}(U).
		\end{equation}
		We say $f$ is polyhomogeneous on $X$ if it is polyhomogeneous near all of $\partial X$.
		
		We write $f\in\O^{\vec{a}}$ for some $a=((z_1,k_1),....,(z_m,k_m))$ if $f\in\A{}^{\vec{\E}}$ with $\min(\E_i)\geq(z_i,k_i)$, and when we write $a_i\in\R$ instead $a_i\in\R\times\N$, we mean that there exists some $k\in\N$, such that $(z,k)$ works in place of $a_i$.
		Similarly, for $\E_1$ index set and $a_i\in\R\times\N$ write $f\in\A{phg}^{\E_1,a_2,...,a_n}$, if there exist index sets $\E_j$ with $\min(\E_j)=a_j$ for $j\neq1$ such that $f\in\A{}^{\vec{\E}}$.
	\end{definition}
	
	Let us already note that \cref{not:eq:poly_def} in $U$, a neighbhourhood of a single boundary $\{\rho_1=0\}$, is equivalent to the existence of functions $f_{(z_1,k_1)} \in C^\infty(\R^{n+1})$ such that
	\begin{equation}
		f-\sum_{(z,k)\in\E^1_{c_1}}f_{z_1,k_1}(y)\rho_1^{z}\log^k(\rho_1)\in\A{b}^{\infty;c_1+}(U)
	\end{equation}
	
	We note that the restriction of polyhomogeneous functions onto the boundary hypersurfaces are well behaved:
	
	\begin{lemma}[Hypersurface projections]\label{not:lemma:projection}
		Let $f\in\O^{p_\scri,(p_+,k_+),(p_\K,k_\K),(p_\F,k_\F)}(\Mcomp_{\mathfrak{t}})$.
		Then, the following restrictions are well defined
		\begin{subequations}
			\begin{align}
				\rho_+^{p_+}\log^{k_+}\rho_+ f|_{\ip}\in\O^{p_\scri,p_\K}(\ip)\\
				\rho_\K^{p_\K}\log^{k_\K}\rho_\K f|_{\K}\in\O^{p_+,p_\F}(\K)\\
				\rho_\F^{p_\F}\log^{k_\F}\rho_\F f|_{\F}\in\O^{p_\K,p_{\hor}}(\F).
			\end{align}
		\end{subequations}
	\end{lemma}
	
	It is easy to check that the mapping properties of \cref{map:lemma:normal_ops} and the algebra property holds in $\A{phg}$ as well:
	\begin{lemma}\label{app:lemma:map}
		For the wave operator on ERN it holds that
		\begin{equation}
			\Box_{\ern}:\A{phg}^{p_{\scri},p_+,p_\K,p_\F}(\Mcomp_\mathfrak{t})\to\A{phg}^{p_{\scri}+1,p_++2,p_\K,p_\F}(\Mcomp_\mathfrak{t}).
		\end{equation}
		and for $\bullet\in\{K,\ip,F\}$ the model operators defined in \cref{map:def:model}
		\begin{equation}
			\Box_{\ern}-\Normal{\bullet}:\A{phg}^{p_{\scri},p_+,p_\K,p_\F}(\M)\to \rho_\bullet \A{phg}^{p_{\scri}+1,p_++2,p_\K,p_\F}(\M).
		\end{equation}
	\end{lemma}
	
	\begin{lemma}
		Let $f^{(i)}\in\O^{p^{(i)}_\scri,p^{(i)}_+,p^{(i)}_\K,p^{(i)}_\F}(\Mcomp_{\mathfrak{t}})$ for $p^{(i)}_\bullet\in\R$ and $i\in\{1,2\}$.
		Then, for $p_\bullet=p^{(1)}_\bullet+p^{(2)}_\bullet$ it holds that $f^{(1)}f^{(2)}\in\O^{p_\scri,p_+,p_\K,p_\F}(\Mcomp_{\mathfrak{t}})$
	\end{lemma}

	\subsection{Spatially compact region}
	\begin{lemma}[Invertibility of $\Normal{\K}$]
		Fix $a_+\in(0,1),a_\F\in(-1,0)$.
		For $f\in\O^{a_++2,a_\F}(\K)$ there exists a unique $\phi\in\O^{a_+,a_\F}(\K)$ solving $\Normal{\K} \phi=f$.
	\end{lemma}
	\begin{proof}
		From \cref{an:prop:K} we already know that a solution exists in $\A{b}^{a_+,a_\F}(\K)$.
		
		\emph{Expansion:}
		We expand \cref{an:eq:K_eq} by studying the normal operator of $\Normal{\K}$ at the boundaries $\partial K=K_++K_-$ with $K_+=K\cap\{r=\infty\}$ and $K_-=K\cap\{r=M\}$.
		To this end, let us note that $\Normal{\K}-\Delta_{r}:\O^{p,q}(\K)\to\O^{p+3,q}(\K)$.
		The indicial family of $\Delta_r$ at $\{r=\infty\}$ is
		\begin{equation}
			N(r^2\Delta_r,\lambda):= r^{\lambda }r^2\Delta_r r^{-\lambda}=(\lambda+1)\lambda+\slashed{\Delta}\in\Diff^2(\sphere).
		\end{equation}
		For $\lambda\notin\N$ this is clearly invertibly, while, for $\lambda\in\N$ we the kernel is the eigenspace of $\slashed{\Delta}$, $\ker N(r^2\Delta_r,\lambda)=\{f\in C^{\infty}(\sphere):\slashed{\Delta} f=-(\lambda+1)\lambda f\}$.
		Hence, we have that for $f_0=r^{-\lambda}\log^k(r)\chi(r)$, with $\chi$ localising to $r>2M$, there exists $\phi_0=\O^{(\lambda,k+1),\infty}(\K)$ such that $\Delta_r\phi_0-\chi f_0\in\O^{(\lambda+3,k-1),\infty}(\K)$.
		By induction, it follows that $\chi_{r>3}\phi\in\A{phg}^{a_+,\infty}(\K)$.
		
		As the normal operator of $\Normal{\K}$, up to a factor of $r^2_*$, is the same at both ends of $K$, we conclude that for $f\in\O^{a_++2,a_\F}(\K)$, the solution $\phi$ found in the previous section satisfies $\phi\in\O^{a_+,a_\F}(\K)$.
	\end{proof}
	
	\begin{cor}\label{poly:cor:K}
		Fix $U_c=\Mcomp\cap\{r_*/t\in(-c,c)\}$, and $p\geq1$.
		Let $f\in\O^{p+3,p,p}(U_{1/4})$.
		Then, there exists $\phi\in\O^{p+1,p,p}(U_{1/4})$ such that $\supp\phi\in U_{1/4}$ and 
		\begin{equation}
			\Box_{\ern}\phi-f\in\O^{p+3,p+1,p}(U_{1/4}).
		\end{equation}
	\end{cor}
	\begin{proof}
		Let $k$ be minimal with $f\in\O^{p,(p,k),p+3}(U_{1/4})$.
		Let $f_\K:=\big(ft^p\log^{k}(t/\jpns{r_*})\big)|_\K$.
		By \cref{not:lemma:projection}, we have $f_\K\in\O^{3,0}(\K)$.
		Using \cref{an:prop:K}, we get that there exists $\phi\in\O^{1,0}(\K)$ solving $\Normal{\K}\phi=f_\K$.
		Using \cref{map:lemma:normal_ops}, we get that
		\begin{equation}
			\Box_{\ern}t^p\log^{k}t/\jpns{r_*}\phi\chi(t/\jpns{r_*})-f\in\begin{cases}
				\O^{p,(p,k-1),p+3}(U_{2^{1/k-2}}) &k\geq1\\
				\O^{p,p+1,p+3}(U_{2^{1/k-2}} ) &k=0,
			\end{cases}
		\end{equation}
		where $\chi$ is a cutoff localising around 0 such that $\supp\chi'\subset[-2^{1/k-2},-2^{-2}]\cup[2^{-2},2^{1/k-2}]$.
		
		Iterating the above construction by increasing the support of $\chi'$ in each step yields the result. 
	\end{proof}
	
	\subsection{Timelike infinity}
	In this section, we study the invertibility of $\Normal{\ip}$, i.e. the Minkowski wave operator acting on homogeneous functions $u^{-\sigma}f(x/t)$ at $\ip\cong\Bcomp$.
	Note that the polyhomogeneous part has already been studied in precisely the same setting in \cite[Lemma 7.17]{kadar_scattering_2024}.
	Nonetheless,  we recall the necessary material here, as we  need to extend the solvability theory for conformally smooth solutions in \cref{an:sec:i+-inverse-smooth}.
	
	\begin{prop}[Conormal inverse of $\Normal{\ip}$]\label{app:prop:i+_conormal}
		Fix $q_\K\in(q-1,q)$ and $q>3/2$.
		Let $f\in\mathcal{O}^{3,q+2,q_\K}(\MMink)$ and $\psi^\scri\in\mathcal{O}^{q-1}(\scri)$.
		Then, the unique scattering solution to 
		\begin{equation}
			\Box_\eta\phi=f,\qquad r\phi|_\scri=\psi^\scri
		\end{equation}
		satisfies $\phi\in\mathcal{O}^{1,q,q_\K}(\MMink)$.
	\end{prop}
	\begin{proof}
		From \cref{an:prop:i+_conormal} it already holds that $\phi\in\A{b}^{1,(q,q_\K)-}(\MMink)$.
		As in the proof of \cref{an:prop:i+_conormal}, we may assume that $f\in\mathcal{O}^{3,q+2,\infty}(\MMink)\subset \mathcal{O}^{3,q+2}(\MMink_{\emptyset})$ by removing the singularity at $\K$ order by order.

		Let $S=t\partial_t+x\cdot\partial_x$ be the scaling vectorfield.
		We note that for some $k\in\N$ it holds that $(S+q+2)^kf\in \mathcal{O}^{3,q+3}(\MMink_{\emptyset})$, and since $S$ commutes with $r^2\Box_\eta$, we obtain that $(S+q+2)^k\phi\in \in \A{b}^{1,q+1-}(\MMinkempty)$.
		Iterating this construction yields the result.
	\end{proof}

	We can use the above result to deduce appropriate invertibility of $\Box_\eta$ for functions that have an expansions towards $\ip$ and $\scri$.
	
	\begin{cor}\label{poly:corr:i+}
		Let $U_s=\Mcomp\cap\{r>s\}$. 
		Fix a function $f$ such that  $\supp f\subset U_{3M}$  and $f\in\O^{p_\scri,p_++2,p_+}(U_{3M})$  with  $p_\scri>3$ and $\psi_\scri\in\O^{p_++1}(\scri)$.
		Then, there exists $\phi$ with $\supp\phi\subset U_{2M}$ and $\phi\in\O^{(1,0),p_+,p_+ }(U_{2M})$ such that 
		\begin{equation}
			\Box_{\ern}\phi-f\in\O^{3,p_++3,p_+}(U_{2M}),\qquad r\phi|_{\scri}-\psi_\scri\in\O^{p_++2})(\scri)
		\end{equation} 
	\end{cor}
	\begin{proof}
		Follows using part \cref{app:prop:i+_conormal} and \cref{app:lemma:map}.
	\end{proof}
	
	\subsection{Near horizon}
	In this section, we will again work on Minkowski space and put \cref{an:prop:i+_smooth_conormal} into a polyhomogeneous setting.
	Via the Couch-Torres isometry \cref{not:eq:conformal_wave} this is sufficient to also deduce the improvement in the near-horizon region.
	
	\begin{prop}\label{app:prop:i+_smooth_conormal}
		Let $\Vc=\{1,t_*^{-1}r^2\partial_r|_u\}$.
		Fix an inhomogeneity $\Vc^kf\in\O^{3,q+2,q-1}(\MMink)$ $\forall k$.
		Then, there exists $\psi^\scri\in \O^{q-1}(\scri)$ such that the scattering solution 
		\begin{equation}
			\Box_\eta\phi=f,\qquad r\phi|_\scri=\psi^\scri
		\end{equation}
		satisfies $\Vc^k\phi\in\mathcal{O}^{1,q,q-1}(\MMink)$ for all $k$.
	\end{prop}
	\begin{proof}
		As for \cref{app:prop:i+_conormal} it suffices to consider the case $f\in\O^{3,q+2}(\MMinkempty)$, and $\supp f\subset \{t>1\}$.
		
		We follow the proof of \cref{an:prop:i+_smooth_conormal} and consider 
		\begin{equation}
			\Box \Phi=F=\frac{\mathcal{T}^{q+1}(u^2v^2rf)}{u^2v^2r}\in\mathcal{O}^{3,1}(\MMinkempty),\quad \Phi|_{t=0}=T\Phi|_{t=0}=0.
		\end{equation}
		From the proof of \cref{an:prop:i+_smooth_conormal} it already holds that $\Vc^k\Phi\in \A{b}^{1,-1-}(\MMinkempty)$ for all $k$.
		
		We also note, that $(S+1)^kF\in \mathcal{O}^{3,2}(\MMinkempty)$ for some $k\in\N$.
		Using the commutation \cref{an:eq:conformal_commutation} we get that
		\begin{equation}
			\Box(r^{-1}\mathcal{T}(S-1)^k\Phi)=\frac{\mathcal{T}u^2v^2r(S+1)^kF}{u^2v^2r}\in\O^{3,1}(\MMinkempty).
		\end{equation}
		The initial data for $r^{-1}\mathcal{T}r(S-1)^k\Phi$ also vanishes at $\{t=0\}$, and using \cref{an:prop:i+_smooth_conormal} it follows that $\Vc^kr^{-1}\mathcal{T}r(S-1)^k\Phi\in\A{b}^{1,-1-}(\MMinkempty)$ for all $k$.
		This already $\Phi\in \mathcal{O}^{1,-1}(\MMinkempty)+\A{b}^{1,0-}(\MMinkempty)$.
		By induction we can make the error term to have arbitrary fast decay and thus conclude that $\Phi\in\O^{1,-1}(\MMinkempty)$
	\end{proof}

	From this, we immediately get
	
	\begin{cor}\label{poly:corr:i+_smooth}
		Let $U_c=\{r<cM\}\cap\Mcomp$. 
		Fix a function with $\supp f\in U_{2}$ and $f\in\A{phg}^{p_\F,p_\F-1,\mindex{0}}(U_2)$.
		Then, there exists $\phi_\hor\in\O^{p_\F}(\hor)$ together with $\phi\in\A{phg}^{p_\F,p_\F-1,\mindex{0}}(U_3)$ such that
		\begin{equation}
			\Box_{\ern}\phi-f\in\A{phg}^{p_\F,p_\F,\mindex{0}}(U_3),\qquad \phi|_{\hor}=\phi_\hor.
		\end{equation}
	\end{cor}
	\begin{proof}
		This follows from \cref{app:prop:i+_smooth_conormal}  using the isometry \cref{not:eq:conformal_wave} and \cref{app:lemma:map}.
	\end{proof}

	\pagebreak
	\printbibliography
	
\end{document}